\documentclass[aos,preprint]{imsart}

\RequirePackage[OT1]{fontenc}
\RequirePackage[colorlinks,citecolor=blue,urlcolor=blue]{hyperref}
\usepackage[authoryear,round,compress,sort]{natbib} 

\usepackage{amsmath,amssymb,amsfonts}
\usepackage{amsbsy, amsthm,amsmath, epsfig, subfig}
\usepackage{graphicx,rotating, tabularx, booktabs}
\usepackage{enumerate}
\RequirePackage[colorlinks,citecolor=blue,urlcolor=blue]{hyperref}

\usepackage{xr}
\arxiv{arXiv:1502.04237}

\startlocaldefs
\numberwithin{equation}{section}
\theoremstyle{plain}

\newcommand{\bfm}[1]{\ensuremath{\mathbf{#1}}}
\def\ba{\bfm a}          
\def\bb{\bfm b}          
          
     \def\bD{\bfm D}

     \def\bG{\bfm G}     
          
     \def\bI{\bfm I}

     \def\bL{\bfm L}

     \def\bP{\bfm P}

\def\bt{\bfm t}          
\def\bu{\bfm u}     \def\bU{\bfm U}     
\def\bv{\bfm v}     \def\bV{\bfm V}     
\def\bw{\bfm w}     \def\bW{\bfm W}     
\def\bx{\bfm x}     \def\bX{\bfm X}     
\def\by{\bfm y}     \def\bY{\bfm Y}     
\def\bZ{\bfm z}     \def\bZ{\bfm Z}     

\newcommand{\bfsym}[1]{\ensuremath{\boldsymbol{#1}}}

\def\balpha    {\bfsym \alpha}
\def\bbeta     {\bfsym \beta}

\def\bdelta    {\bfsym \delta}

\def\bxi       {\bfsym {\xi}}

\def\beps      {\bfsym \varepsilon}

\def\bGamma  {\bfsym \Gamma}

\def\bSigma  {\bfsym \Sigma}

\renewcommand{\hat}{\widehat}

\makeatletter
\def\singlespace{\def\baselinestretch{1}\@normalsize}

\newcommand{\noi}{\noindent}

\newcommand{\beq}{\begin{eqnarray*}}
\newcommand{\eeq}{\end{eqnarray*}}
\newcommand{\beqn}{\begin{eqnarray}}
\newcommand{\eeqn}{\end{eqnarray}}

\newcommand{\var}{\mbox{Var}}
\newcommand{\ei}{\end{itemize}}
\newcommand{\sta}{\stackrel}

\newcommand{\be}{\begin{equation}}
\newcommand{\ee}{\end{equation}}
\newcommand{\nn}{\nonumber}

\newcommand{\ignore}[1]{}{}
\newcommand{\f}{\frac}

\newcommand{\sn}{\sum_{i=1}^n}

\renewcommand{\P}{\mathbb{P}}

\newcommand{\de}{\delta}
\newcommand{\De}{\Delta}

\newcommand{\s}{\sqrt}

\newcommand{\e}{\mathbb{E}}

\newcommand{\bbr}{\mathbb{R}}
\newcommand{\lea}{\left\langle}
\newcommand{\ria}{\right\rangle}
\newcommand{\wh}{\widehat}
\newcommand{\mo}{\mathbf{0}}

\newcommand{\mI}{\mathbf{I}}
\newcommand{\mS}{\mathbf{S}}
\newcommand{\mM}{\mathbf{M}}
\newcommand{\mx}{\mathbf{x}}
\newcommand{\bS}{\mathbb{S}}
\newcommand{\PP}{\mathbb{P}}

\newtheorem{lemma}{Lemma}[section]
\newtheorem{proposition}{Proposition}[section]
\newtheorem{theorem}{Theorem}[section]
\newtheorem{assumption}{Condition}[section]

\newtheorem{remark}{Remark}[section]
\newtheorem{definition}{Definition}[section]

\def \sMB {\mbox{\scriptsize MB}}
\def \scMB {\mbox{\scriptsize CMB}}

\endlocaldefs

\begin{document}

\begin{frontmatter}
\title{{Are Discoveries Spurious? \\ Distributions of Maximum Spurious Correlations and Their Applications}\thanksref{T1}}

\runtitle{Distributions of Maximum Spurious Correlations}
\thankstext{T1}{J. Fan was partially supported by NSF Grants DMS-1206464, DMS-1406266, and NIH Grant R01-GM072611-10. Q.-M. Shao was supported in part by Hong Kong Research Grants Council GRF-403513 and GRF-14302515. W.-X. Zhou was supported by NIH Grant R01-GM072611-10.}

\begin{aug}
\author{\fnms{Jianqing} \snm{Fan}\thanksref{t0,t1}, \ead[label=e1]{jqfan@princeton.edu}}
\author{\fnms{Qi-Man} \snm{Shao}\thanksref{t3}\ead[label=e2]{qmshao@cuhk.edu.hk}}
\and
\author{\fnms{Wen-Xin} \snm{Zhou}\thanksref{t1,t4}\ead[label=e3]{wenxinz@princeton.edu} \ead[label=e4]{stevewenxin@hotmail.com} }

\runauthor{J. Fan, Q.-M. Shao and W.-X. Zhou}

\affiliation{Fudan University\thanksmark{t0}, Princeton University\thanksmark{t1}, \\
Chinese University of Hong Kong\thanksmark{t3} \\ 
and University of California, San Diego\thanksmark{t4}}

\address{School of Data Science \\
Fudan University \\
Shanghai 200433 \\
China \\
and \\
Department of Operations Research \\
			  and Financial Engineering \\
			   Princeton University \\
			   Princeton, New Jersey 08544 \\
			   USA \\
			   \printead{e1}}

\address{Department of Statistics \\
              Chinese University of Hong Kong  \\
              Shatin, NT \\ Hong Kong \\ \printead{e2}}
              
\address{Department of Mathematics \\
University of California, San Diego \\
La Jolla, California 92093 \\
USA \\
\printead{e4}}
\end{aug}

\begin{abstract}
Over the last two decades, many exciting variable selection methods have been developed for finding a small group of covariates that are associated with the response from a large pool.  Can the discoveries from these data mining approaches be spurious due to high dimensionality and limited sample size?  Can our fundamental assumptions about the exogeneity of the covariates needed for such variable selection be validated with the data?  To answer these questions, we need to derive the distributions of the maximum spurious correlations given a certain number of predictors, namely, the distribution of the correlation of a response variable $Y$ with the best $s$ linear combinations of $p$ covariates $\bX$, even when $\bX$ and $Y$ are independent.  When the covariance matrix of $\bX$ possesses the restricted eigenvalue property, we derive such distributions for both a finite $s$ and a diverging $s$, using Gaussian approximation and empirical process techniques.  However, such a distribution depends on the unknown covariance matrix of $\bX$. Hence, we use the multiplier bootstrap procedure to approximate the unknown distributions and establish the consistency of such a simple bootstrap approach. The results are further extended to the situation where the residuals are from regularized fits.  Our approach is then used to construct the upper confidence limit for the maximum spurious correlation and to test the exogeneity of the covariates. The former provides a baseline for guarding against false discoveries and the latter tests whether our fundamental assumptions for high-dimensional model selection are statistically valid.  Our techniques and results are illustrated with both numerical examples and real data analysis.
\end{abstract}

%\begin{keyword}[class=MSC]
%\kwd[Primary ]{}
%\kwd{}
%\kwd[; secondary ]{}
%\end{keyword}

\begin{keyword}
\kwd{High dimension}
\kwd{spurious correlation}
\kwd{bootstrap}
\kwd{false discovery}.
\end{keyword}

\end{frontmatter}

\section{Introduction} \label{sec1}
Information technology has forever changed the data collection process. Massive amounts of very high-dimensional or unstructured data are continuously produced and stored at an affordable cost. Massive and complex data and high dimensionality characterize  contemporary statistical problems in many emerging fields of science and engineering.  Various statistical and machine learning methods and algorithms have been proposed to find a small group of covariate variables that are associated with given responses such as biological and clinical outcomes. These methods have been very successfully applied to genomics, genetics, neuroscience, economics, and finance. For an overview of high-dimensional statistical theory and methods, see the review article by \cite{FLV10} and monographs by \cite{DVL07}, \cite{HTF09}, \cite{EFR10} and \cite{BVG11}.

Underlying machine learning, data mining, and high-dimensional statistical techniques, there are many
model assumptions and even heuristic arguments.  For example, the LASSO [\cite{TIB96}] and the SCAD [\cite{FLI01}] are based on an exogeneity assumption, meaning that all of the covariates and the residual of the true model are uncorrelated.  However, it is nearly impossible that such a random variable, which is the part of the response variable that can not be explained by a small group of covariates and lives in a low-dimensional space spanned by the response and the small group of variables, is uncorrelated with any of the tens of thousands of coviariates.  Indeed,
\cite{FLI14} and \cite{FHL14} provide evidence that such an ideal assumption might not be valid, although it is a necessary condition for model selection consistency.  Even under the exogenous assumption, conditions such as the restricted eigenvalue condition [\cite{BRT09}] and homogeneity [\cite{FHL14}] are needed to ensure model selection consistency or oracle properties.  Despite their critical importance, these conditions have rarely been verified in practice.  Their violations can lead to false scientific discoveries.  A simpler question is then, for a given data set, do data mining techniques produce results that are better than spurious correlation? The answer depends on not only the correlation between the fitted and observed values, but also on the sample size, the number of variables selected, and the total number of variables.

\begin{figure}[htbp]
  \centering
  \includegraphics[width=4in]{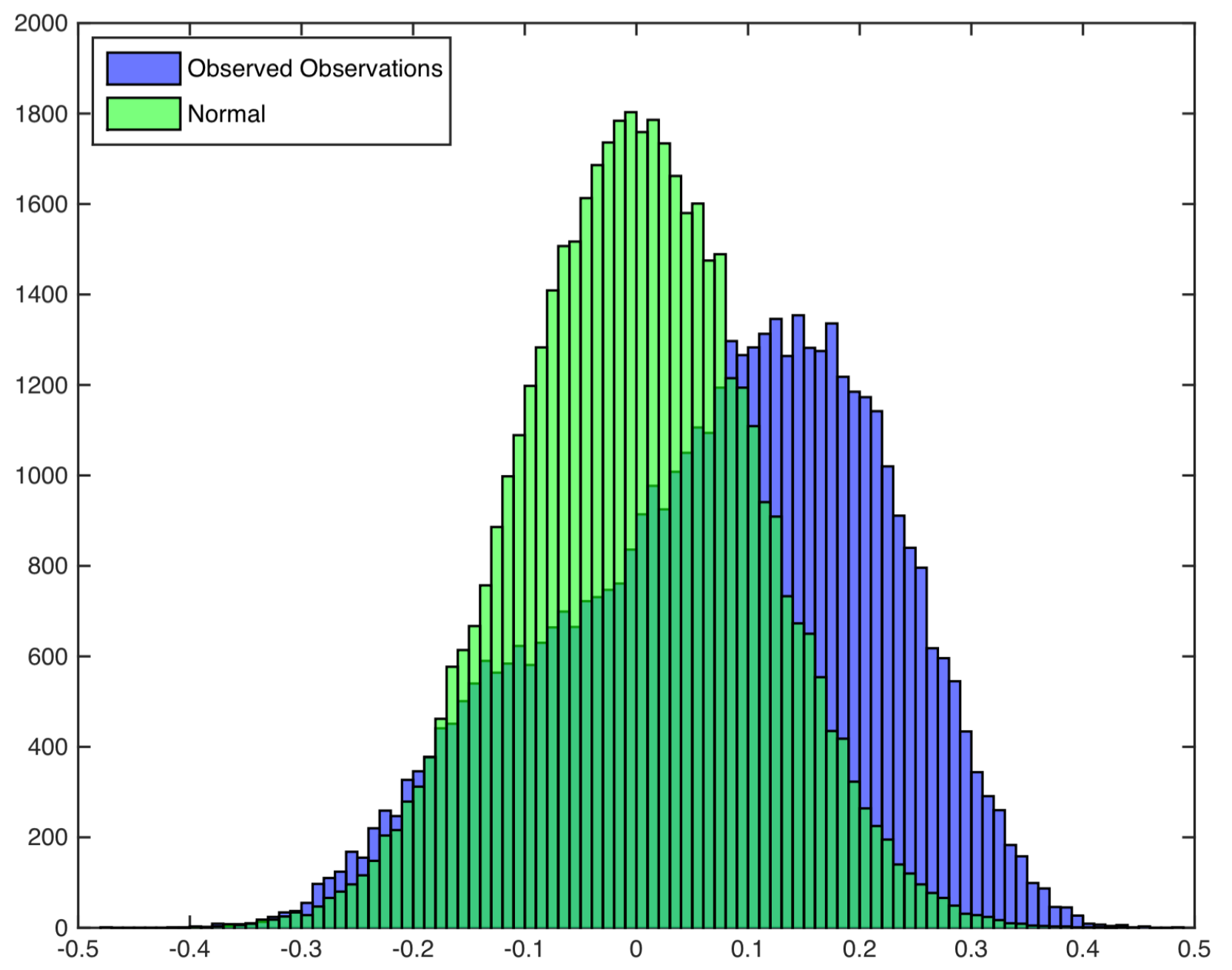}
  \caption{Histogram of the sample correlations between the residuals and each covariate (blue) and histogram of $N(0, 1/\sqrt{n})$ random variables (green). \label{Fig1} }
\end{figure}

To better appreciate the above two questions, let us consider an example.  We take the gene expression data on 90 Asians (45 Japanese and 45 Han Chinese) from the international `HapMap' project [\cite{TSK05}]. The normalized gene expression data are generated with an Illumina Sentrix Human-6 Expression Bead Chip [\cite{SNF07}] and are available on {\tt ftp://ftp.sanger.ac.uk/pub/genevar/}. We take the expressions of gene {\em CHRNA6}, a cholinergic receptor, nicotinic, alpha 6, as the response $Y$ and the remaining expressions of probes as covariates $\bX$ with dimension $p = 47292$. We first fit an $\ell_1$-penalized least-squares regression (LASSO) on the data with a tuning parameter automatically selected via ten-fold cross validation (25 genes are selected). The correlation between the LASSO-fitted value and the response is $0.8991$. Next, we refit an ordinary least-squares regression on the selected model to calculate the fitted response and residual vector. The sample correlation between the post-LASSO fit and observed responses is $0.9214$, a remarkable fit! But is it any better than the spurious correlation?  The model diagnostic plot, which depicts the empirical distribution of the correlations between each covariate $X_j$ and the residual $\hat{\varepsilon}$ after the LASSO fit, is given in Figure~{\ref{Fig1}. Does the exogenous assumption that $\e (\varepsilon X_j ) = 0$ for all $j=1, \ldots, p$ hold?

To answer the above two important questions, we need to derive the distributions of the maximum spurious correlations. Let $\bX$ be the $p$-dimensional random vector of the covariates and $\bX_{S}$ be a subset of covariates indexed by $S$.  Let $\widehat{\mbox{corr}}_n(\varepsilon,  \balpha_S^{{{\rm T}}}  \bX_S )$ be the sample correlation between the random noise $\varepsilon$ (independent of $\bX$) and $\balpha_S^{{{\rm T}}}  \bX_S$ based on a sample of size $n$, where $\balpha_S$ is a constant vector.  Then, the maximum spurious correlation is defined as
\begin{equation}      \label{eq1.1}
    \hat{R}_n(s,p) = \max_{|S| = s}
    \max_{\scriptsize \balpha_S} \widehat{\mbox{corr}}_n(\varepsilon, \balpha_S^{{{\rm T}}}  \bX_S ) ,
\end{equation}
when $\bX$ and $\varepsilon$ are independent, where the maximization is taken over all ${p \choose s}$ subsets of size $s$ and all of the linear combinations of the selected $s$ covariates. Next, let $(Y_i, \bX_i), \ldots , (Y_n, \bX_n)$ be independent and identically distributed (i.i.d.) observations from the linear model $ Y = \bX^{{\rm T}} \bbeta^* + \varepsilon$. Assume that $s$ covariates are selected by a certain variable selection method for some $1\leq s \ll \min(p,n)$. If the correlation between the fitted response and observed response is no more than the $90$th or the $95$th percentile of $\hat{R}_n(s,p)$, it is hard to claim that the fitted value is impressive or even genuine. In this case, the finding is hardly more impressive than the best fit using data that consist of independent response and explanatory variables, 90\% or 95\% of the time. To simplify and unify the terminology, we call this result the spurious discovery throughout this paper.

For the aforementioned gene expression data, as 25 probes are selected, the observed correlation of $0.9214$ between the fitted value and the response should be compared with the distribution of $\hat{R}_n(25, p)$. Further, a simple method to test the null hypothesis
\begin{equation}      \label{eq1.2}
        \e (  \varepsilon X_j ) = 0,  ~\mbox{for all $j=1, \ldots, p$},
\end{equation}
is to compare the maximum absolute correlation in Figure~\ref{Fig1} with the distribution of
$\hat{R}_n(1, p)$.  See additional details in Section~\ref{sec5.2}.

\begin{figure}[hbtp!]
  \centering
  \includegraphics[width=4in]{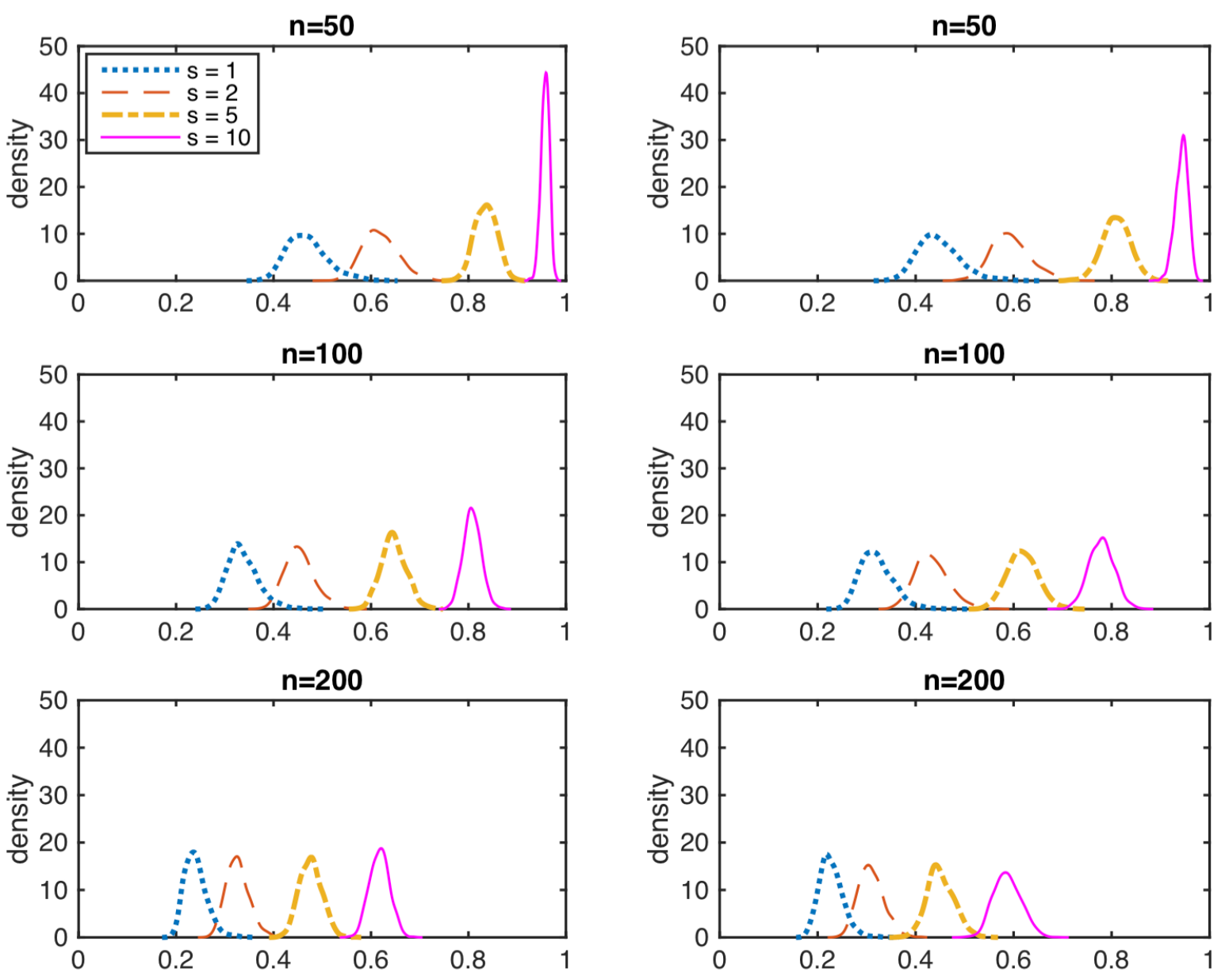}
  \caption{Distributions of maximum spurious correlations for $p=1000$ and
  $s= 1, 2, 5$ and $10$ when $\bSigma$ is the identity matrix (left panel) or block diagonal (right panel) with the first block being a $500\times500$ equi-correlation matrix with a correlation 0.8 and the second block being the $500\times 500$ identity matrix.   From top to bottom: $n = 50, 100$ and $200$. \label{Fig2} }
\end{figure}

The importance of such spurious correlation was recognized by \cite{CJI11}, \cite{FGH12} and \cite{CFJ13}. When the data are independently and normally distributed, they derive the distribution of $\hat{R}_n(1, p)$, which is equivalent to the distribution of the minimum angle to the north pole among $p$ random points uniformly distributed on the $(n+1)$-dimensional sphere. \cite{FGH12} conducted simulations to demonstrate that the spurious correlation can be very high when $p$ is large and grows quickly with $s$. To demonstrate this effect and to examine the impact of correlation and sample size, we conduct a similar but more extensive simulation study based on a combination of the stepwise addition and branch-and-bound algorithms. We simulate $\bX$ from $N(\mo, \bI_p)$ and $N(\mo, \bSigma_0)$, where $\bSigma_0$ is block diagonal with the first block being a $500\times500$ equi-correlation matrix with a correlation 0.8 and the second block being the $(p-500)\times(p-500)$ identity matrix.  $Y$ is simulated independently of $\bX$ and follows the standard normal distribution.  Figure~\ref{Fig2} depicts the simulation results for $n=50, 100$ and $200$.  Clearly, the distributions depend on $(s,p,n)$ and $\bSigma$, the covariance matrix of $\bX$, although the dependence on $\bSigma$ does not seem very strong.  However, the theoretical result of \cite{FGH12} covers only the very specific case where $s=1$ and $\bSigma = \bI_p$.

There are several challenges to deriving the asymptotic distribution of the statistic $\hat{R}_n(s,p)$, as it involves combinatorial optimization.  Further technical complications are added by the dependence among the covariates $\bX$. Nevertheless, under the restricted eigenvalue condition [\cite{BRT09}] on $\bSigma$,  in this paper, we derive the asymptotic distribution of such a spurious correlation statistic for both a fixed $s$ and a diverging $s$, using the empirical process and Gaussian approximation techniques given in \cite{CCK14a}.  As expected, such distributions depend on the unknown covariance matrix $\bSigma$.  To provide a consistent estimate of the distributions of the spurious correlations, we consider the use of a multiplier bootstrap method and demonstrate its consistency under mild conditions. {The multiplier bootstrap procedure has been widely used due to its good numerical performance. Its theoretical validity is guaranteed by the multiplier central limit theorem [\cite{VW96}]. For the most advanced recent results, we refer to \cite{CB2005}, \cite{ABR2010} and \cite{CCK13}. In particular, \cite{CCK13} developed a number of non-asymptotic results on a multiplier bootstrap for the maxima of empirical mean vectors in high dimensions with applications to multiple hypothesis testing and parameter choice for the Dantzig selector.} The use of multiplier bootstrapping enables us to empirically compute the upper confidence limit of $\hat{R}_n(s,p)$ and hence decide whether discoveries by statistical machine learning techniques are any better than spurious correlations.

The rest of this paper is organized as follows. Section~\ref{sec2} discusses the concept of spurious correlation and introduces the main conditions and notation. Section~\ref{sec3} presents the main results of the asymptotic distributions of spurious correlations and their bootstrap approximations, which are further extended in Section~\ref{sec4}. Section~\ref{sec5} identifies three important applications of our results to high-dimensional statistical inference. Section~\ref{sec6} presents the numerical studies. The proof of Theorem~\ref{thm3.1} is provided in Section~\ref{sec7}, and the proofs for the remaining theoretical results are provided in the supplementary material.

\section{Spurious correlation, conditions, and notation} \label{sec2}
Let $\varepsilon, \varepsilon_1, \ldots , \varepsilon_n$ be i.i.d. random variables with a mean of zero and a finite variance $ \sigma^2 >0$, and let $\bX, \bX_1, \ldots , \bX_n$ be i.i.d. $p$-dimensional random vectors with a mean of zero and a covariance matrix $\bSigma= \e ( \bX \bX^{{{\rm T}}} )= (\sigma_{jk})_{1\leq j, k \leq p} $. Write
\beq
	  \bX = ( X_{1}, \ldots , X_{p})^{{{\rm T}}},    \quad    \bX_i = (X_{i1}, \ldots , X_{ip})^{{{\rm T}}} , \ \ i=1, \ldots, n.
\eeq
Assume that the two samples $\{\varepsilon_i \}_{i= 1}^n$ and $\{\bX_i\}_{i = 1}^n$ are independent. Then, the spurious correlation (\ref{eq1.1}) can be written as
\be
		\wh R_{n}(s, p ) = \max_{\balpha \in \bS^{p-1} :   | \balpha |_0=s}  \widehat{{\rm corr}}_n \big(  \varepsilon , \balpha^{{{\rm T}}} \bX \big), \label{eq2.1}
\ee
where the dimension $p$  and sparsity $s$ are allowed to grow with the sample size $n$. Here $\widehat{{\rm corr}}_n(\cdot, \cdot)$ denotes the sample Pearson correlation coefficient and $\bS^{p-1} :=\{\balpha \in \bbr^p: | \balpha |_2=1\}$ is the unit sphere of $\bbr^p$.  Due to the anti-symmetric property of the sample correlation under the sign transformation of $\balpha$, we have also
\begin{equation} \label{eq2.2}
    \wh R_{n}(s,p ) = \max_{\balpha \in \bS^{p-1}  :  | \balpha |_0=s}  | \widehat{{\rm corr}}_n \big(  \varepsilon , \balpha^{{{\rm T}}} \bX \big)|,
\end{equation}
More specifically, we can express $\hat R_n(s,p)$ as
\begin{align}
	  \max_{S \subseteq [p] :  |S|=s} \max_{\balpha \in \bS^{s-1} } \frac{\sn (\varepsilon_i -\bar{\varepsilon}_n )    \lea \balpha,  \bX_{i,S} - \bar{\bX}_{n,S} \ria }{  \s{\sn (\varepsilon_i -\bar{\varepsilon}_n )^2 \cdot \sn  \lea \balpha,  \bX_{i,S}-\bar{\bX}_{n,S} \ria ^2}} . \label{eq2.3}
\end{align}
By the scale-invariance property of $\hat R_n(s,p)$, we assume without loss of generality that $\sigma^2=1$ and $\bSigma$ is a correlation matrix, so that ${\rm diag}(\bSigma)=\mI_p$.

For a random variable $X$, the sub-Gaussian norm $\| X\|_{\psi_2}$ and sub-exponential norm $\| X\|_{\psi_1}$ of $X$ are defined, respectively, as
\beq
	\| X \|_{\psi_2} =\sup_{ q \geq 1}   q^{-1/2}\big( \e |X|^q \big)^{1/q} \ \ \mbox{ and } \ \   \| X \|_{\psi_1} =\sup_{ q \geq 1}   q^{-1 }\big( \e |X|^q  \big)^{1/q}.
\eeq
A random variable $X$ that satisfies $\| X \|_{\psi_2}  < \infty$ (resp., $\| X \|_{\psi_1}<\infty$) is called a sub-Gaussian (resp., sub-exponential) random variable [\cite{V12}].

The following moment conditions for $\varepsilon\in \bbr$ and $\bX\in \bbr^p$ are imposed.

\begin{assumption} \label{moment.cond}
There exists a random vector $\bU$ such that $\bX = \bSigma^{1/2}\bU$, $\e ( \bU )=\mo$, $\e ( \bU  \bU^{{{\rm T}}} )= \bI_p$ and $K_1:= \sup_{\balpha \in \bS^{p-1}} \|   \balpha^{{{\rm T}}}   \bU   \|_{\psi_2} < \infty$. The random variable $\varepsilon$ has a zero mean and unit variance, and is sub-Gaussian with $K_0:= \| \varepsilon \|_{\psi_2} < \infty$. Moreover, write $v_q = \e ( |\varepsilon|^q)$ for $q \geq 3$.
\end{assumption}

The following is our assumption for the sampling process.

\begin{assumption} \label{sampling.condition}
$\{ \varepsilon_i\}_{i=1}^n$ and $\{\bX_i \}_{i=1}^n$ are independent random samples from the distributions of $\varepsilon$ and $\bX$, respectively.
\end{assumption}

For $1\leq s\leq p$, the $s$-sparse minimal and maximal eigenvalues [\cite{BRT09}] of the covariance matrix $\bSigma$ are defined as
\beq
	\phi_{\min}(s)=\min_{\bu \in \bbr^p : 1\leq |\bu|_0 \leq s}  \big( |\bu |_{ {\bSigma}}  / |\bu|_2 \big)^2,   \ \    \phi_{\max}(s)=\max_{\bu \in \bbr^p : 1\leq |\bu |_0 \leq s} \big( |\bu |_{{ \bSigma}}  / |\bu|_2 \big)^2,
\eeq
where  $|\bu |_{{ \bSigma}} = (\bu^{{{\rm T}}}\bSigma \bu)^{1/2}$ and
$|\bu |_2=(\bu^{{{\rm T}}} \bu)^{1/2}$ is the $\ell_2$-norm of $\bu$. Consequently, for $1\leq s\leq p$, the $s$-sparse condition number of $\bSigma$ is given by
\be
		\gamma_s = \gamma_s(\bSigma) = \s{  \phi_{\max}(s)  / \phi_{\min}(s)  } .   \label{eq2.5}
\ee
The quantity $\gamma_s$ plays an important role in our analysis.

The following notation is used. For the two sequences $\{a_n\}$ and $\{ b_n\}$ of positive numbers, we write
$a_n=O(b_n)$ or $a_n \lesssim b_n$ if there exists a constant $C>0$ such that $a_n /b_n \leq C $ for all sufficiently large $n$; we write $a_n \asymp b_n$ if there exist constants $C_1, C_2>0$ such that, for all $n$ large enough, $C_1 \leq  a_n/ b_n \leq C_2$; and we write $a_n \sim b_n$ and $a_n=o(b_n)$ if $\lim_{n\rightarrow \infty} a_n/b_n =1$ and $\lim_{n\rightarrow \infty} a_n/b_n =0$, respectively. For $a, b\in \bbr$, we write $a \vee b=\max(a, b)$ and $a \wedge b=\min(a, b)$. For every vector $\bu$, we denote by $|\bu|_q= \big(  \sum_{i\geq  1} |u_i|^q  \big)^{1/q}$ for $q>0$ and $|\bu|_0=\sum_{i\geq 1} I\{u_i \neq 0\}$. We use $\lea \bu, \bv \ria = \bu^{{{\rm T}}} \bv$ to denote the inner product of two vectors $\bu$ and $\bv$ with the same dimension and $\| \mM\|$ to denote the spectral norm of a matrix $\mM$. For every positive integer $\ell$, we write $[\ell]= \{1, 2, \ldots, \ell\}$, and for any set $S$, we use $S^{{\rm c}}$ to denote its complement and $|S|$ for its cardinality. For each $p$-dimensional vector $\bu$ and $p\times p$ positive semi-definite matrix $\mathbf{A}$, we write $|\bu|_{\mathbf{A}}=( \bu^{{{\rm T}}} \mathbf{A} \bu )^{1/2}$. In particular, put
\be
  \balpha_{{\bSigma}} =\balpha/|\balpha|_{{ \bSigma}}  \label{eq2.6}
\ee
for every $\balpha \in \bbr^p $ and set $\mo_{\bSigma}=\mo$ as the convention.

\section{Distributions of maximum spurious correlations} \label{sec3}

In this section, we first derive the asymptotic distributions of the maximum spurious correlation $\wh R_n(s,p)$.  The analytic form of such asymptotic distributions can be obtained in the isotropic case. As the asymptotic distributions of $\wh R_n(s,p)$ depend on the unknown covariance matrix $\bSigma$, we provide a bootstrap estimate and demonstrate its consistency.

\subsection{Asymptotic distributions of maximum spurious correlations}
\label{sec3.1}

In view of \eqref{eq2.3}, we can rewrite $\wh R_n(s,p)$ as
\begin{align}
	  \wh  R_{n}(s,p)  =  \sup_{ f \in \mathcal{F}} \f{n^{-1}\sn (\varepsilon_i - \bar{\varepsilon}_n ) f(\bX_i - \bar{\bX}_n )}{\s{n^{-1}\sn   (\varepsilon_i - \bar{\varepsilon}_n )^2 } \cdot \s{n^{-1} \sn f^2(\bX_i-\bar{\bX}_n )}} ,  \label{eq3.1}
\end{align}
where $\bar{\varepsilon}_n = n^{-1} \sn \varepsilon_i$, $\bar{\bX}_n = n^{-1} \sn \bX_i$ and
\be
\mathcal{F}= \mathcal{F}(s,p) = \big\{  \mx  \mapsto f_{\balpha}(\mx) := \lea \balpha, \mx \ria : \balpha \in \mathcal{V} \big\}  \label{eq3.2}
\ee
is a class of linear functions $\bbr^p \mapsto \bbr$, where $\mathcal{V} = \mathcal{V}(s,p) = \{ \balpha \in \bS^{p-1} : | \balpha|_0=s  \} $. The dependence of $\mathcal{F}$ and $\mathcal{V}$ on $(s,p)$ is suppressed.

Let $\bZ=(Z_1, \ldots, Z_p)^{{{\rm T}}}$ be a $p$-dimensional Gaussian random vector with a mean of zero and the covariance matrix $\bSigma$, i.e., $\bZ \sta{d}{=} N(\mo, \bSigma)$. Denote by $Z^2_{(1)} \leq Z_{(2)}^2 \leq \cdots \leq Z_{(p)}^2$ the order statistics of $\{Z_1^2, \ldots, Z_p^2\}$.
The following theorem shows that the distribution of the maximum absolute multiple correlation $\wh R_n(s,p)$ can be approximated by that of the supremum of a centered Gaussian process $\mathbb{G}^*$ indexed by $\mathcal{F}$.

\begin{theorem}  \label{thm3.1}
Let Conditions~\ref{moment.cond} and \ref{sampling.condition} hold, $n, p \geq 2$ and $1\leq s\leq p$. Then there exists a constant $C>0$ independent of $(s,p,n)$ such that
\begin{align}
	\sup_{t \geq 0} \big| \PP \big\{ \s{n}  \wh R_n(s,p) \leq t \big\} - &  \PP\big\{  R^*(s,p) \leq t \big\}  \big|   \nn \\
	&  \leq C  (K_0  K_1)^{3/4} \, n^{-1/8} \{s b_n(s,p)\}^{7/8} ,  \label{eq3.3}
\end{align}
where $K_0$ and $K_1$ are defined in Condition~\ref{moment.cond}, $b_n(s,p):= \log ( \gamma_s p /s)  \vee \log n$ for $\gamma_s$ as in (\ref{eq2.5}), $R^*(s,p):= \sup_{f \in \mathcal{F}} \mathbb{G}^*f $ and $\mathbb{G}^*=\{\mathbb{G}^*f    \}_{f  \in \mathcal{F}}$ is a centered Gaussian process indexed by $\mathcal{F}$ defined as, for every $f_{\balpha} \in \mathcal{F}$,
\be
	 \mathbb{G}^*f_{\balpha}  = \balpha_{{ \bSigma}}^{{{\rm T}}}   \bZ  =\frac{  \balpha^{{{\rm T}}} \bZ  }{\s{\balpha^{{{\rm T}}} \bSigma \balpha}}. \label{eq3.4}
\ee
In particular, if $\bSigma =\bI_p$ and $s \log(p   n)  = o(n^{1/7})$, then as $n\rightarrow \infty$,
\begin{align}
	\sup_{t \geq 0} \big| \PP\big\{ n  \wh R^2_n(s,p) \leq t \big\}    - \PP\big\{ Z_{(p)}^2 + \cdots +  Z_{(p-s+1)}^2   \leq t \big\}  \big| \rightarrow 0 . \label{eq3.5}
\end{align}
\end{theorem}

\begin{remark} \label{rmk.0}
{\rm
The Berry-Esseen bound given in Theorem~\ref{thm3.1} depends explicitly on the triplet $(s,p,n)$, and it depends on the covariance matrix $\bSigma$ only through its $s$-sparse condition number $\gamma_s$, defined in \eqref{eq2.5}. The proof of \eqref{eq3.3} builds on a number of technical tools including a standard covering argument, maximal and concentration inequalities for the suprema of unbounded empirical processes and Gaussian processes as well as a coupling inequality for the maxima of sums of random vectors derived in \cite{CCK14a}. Instead, if we directly resort to the general framework in Theorem~2.1 of \cite{CCK14a}, the function class of interest is $\mathcal{F}=\big\{ \bx \mapsto \frac{\balpha^{{\rm T}} \bx}{(\balpha^{{\rm T}} \bSigma \balpha)^{1/2}} : \balpha \in \mathbb{S}^{p-1}, | \balpha |_0 =s \big\}$. Checking high-level conditions in Theorem~2.1 can be rather complicated and less intuitive. Also, dealing with the (uniform) entropy integral that corresponds to the class $\mathcal{F}$ relies on verifying various VC-type properties, and thus can be fairly tedious. Following a strategy similar to that used to prove Theorem~2.1, we provide a self-contained proof of Theorem~\ref{thm3.1} in Section~\ref{sec7.2} by making the best use of the specific structure of $\mathcal{F}$. The proof is more intuitive and straightforward. More importantly, it leads to an explicit non-asymptotic bound under transparent conditions. }
\end{remark}

\begin{remark} \label{rmk.1}
{\rm
In Theorem \ref{thm3.1}, the independence assumption of $\varepsilon$ and $\bX$ can be relaxed as $\e( \varepsilon \bX ) =0 $, $\e ( \varepsilon^2 | \bX ) = \sigma^2$ and $\e  ( \varepsilon^4 | \bX ) \leq C $ almost surely, where $C>0$ is a constant. }
\end{remark}

Expression (\ref{eq3.5}) indicates that the increment
$n \{ \wh R^2_n(s,p)  - \wh R^2_n(s-1, p) \} $ is approximately the same as $Z_{(p-s+1)}^2$. This can simply be seen from the asymptotic joint distribution of $\big( \hat{R}_n(1,p), \hat{R}_n(2,p) , \ldots , \hat{R}_n(s,p) \big)$.  The following proposition establishes the approximation of the joint distributions  when both the dimension $p$ and sparsity $s$ are allowed to diverge with the sample size $n$.

\begin{proposition} \label{prop3.1}
Let Conditions~\ref{moment.cond} and \ref{sampling.condition} hold with $\bSigma=\bI_p$. Assume that the triplet $(s,p,n)$ satisfies $1\leq s < n \leq  p$ and $s^2\log p = o(n^{1/7} )$. Then as $n\to \infty$,
\begin{align}
\sup_{0 \equiv t_0 <t_1<t_2<\cdots < t_s <1 }   \bigg| \PP\bigg[ \bigcap_{k=1}^s \big\{ &\hat{R}_n(k,p) \leq   t_k \big\} \bigg]  \nn \\
&   - \PP\bigg[ \bigcap_{k=1}^s \big\{  Z_{(p-k+1)}^2 \leq n(t_k^2-t_{k-1}^2)  \big\} \bigg]     \bigg| \rightarrow 0 . \nn
\end{align}
\end{proposition}

\begin{remark} \label{rmk3.3}
{\rm
When $s=1$ and if $(n,p)$ satisfies $\log p = o(n^{1/7})$, it is straightforward to verify that, for any $t\in \bbr$,
\be \label{eq3.6}
	\PP\big\{ Z_{(p)}^2  - 2\log p + \log(\log p) \leq t \big\} \rightarrow  \exp(  -\pi^{-1/2} e^{-t/2} ) \ \ \mbox{ as } p\rightarrow \infty.
\ee
This result is similar in nature to (5) in \cite{FGH12}. In fact, it is proved in \cite{SZ14} that the extreme-value statistic $\hat{R}_n(1,p)$ is sensitive to heavy-tailed data in the sense that, under the ultra-high dimensional scheme, even the law of large numbers for the maximum spurious correlation requires exponentially light tails of the underlying distribution. We refer readers to Theorem~2.1 in \cite{SZ14} for details. Therefore, we believe that the exponential-type moment assumptions required in Theorem~3.1 cannot be weakened to polynomial-type ones as long as $\log p$ is allowed to be as large as $n^c$ for some $c\in (0,1)$. However, it is worth mentioning that the factor $1/7$ in Proposition~\ref{prop3.1} may not be optimal, and according to the results in \cite{SZ14}, $1/3$ is the best possible factor to ensure that the asymptotic theory is valid. To close this gap in theory, a significant amount of additional work and new probabilistic techniques are needed. We do not pursue this line of research in this paper.}
\end{remark}

For a general $s\geq 2$, we establish in the following proposition the limiting distribution of the sum of the top $s$ order statistics of i.i.d. chi-square random variables with degree of freedom $1$.

\begin{proposition} \label{prop3.2}
Assume that $s\geq 2$ is a fixed integer. For any $t\in \bbr$, we have as $p \rightarrow \infty$,
\begin{align}
		& \PP\big\{  Z_{(p)}^2 + \cdots + Z_{(p-s+1)}^2 - s a_p \leq t \big\}  \nn  \\
		 \longrightarrow \; & \f{\pi^{(1-s)/2}}{(s-1)!  \Gamma(s-1)} \int_{-\infty}^{t/s} \bigg\{ \int_0^{(t-s v)/2}    u^{s-2}   e^{-u} \, d u \bigg\}  e^{-(s-1)v/2}  g(v)  \, dv , \label{eq3.7}
\end{align}
where $a_p=2\log p - \log( \log p)$, $G(t) =\exp( - \pi^{-1/2} e^{-t/2})$ and $g(t) = G'(t) = \f{e^{-t/2}}{2\s \pi}  G(t)$. The above integral can further be expressed as
\begin{align}
		 & G( t/s ) + \f{\pi^{1-s/2} e^{-t/2} }{(s-1)!}    \int_{-\infty}^{t/s} e^u  g(u)\, du \nn
		+   \f{\pi^{(1-s)/2}  e^{-t/2}}{(s-1)!} \\
 & \quad \times \sum_{j=1}^{s-2}  \bigg\{ G( t/s ) e^{(j+1)t/(2s)}  \pi^{j/2} \prod_{\ell=1}^j (s-\ell)  - \f{1}{j!   2^j} \int_{-\infty}^{t/s} (t-s v)^j e^{v/2}  g(v) \, dv \bigg\} .  \label{eq3.8}
\end{align}
In particular, when $s=2$, the last term on the right-hand side of \eqref{eq3.8} vanishes so that, as $p \to \infty$,
\begin{align*}
	\PP\big\{ Z_{(p)}^2 + Z_{(p-1)}^2 - 2  a_p \leq t \big\}  \rightarrow G( t/2 ) + \frac{e^{-t/2}}{2\s{\pi}} \int_{-\infty}^{t/2} e^{u/2} G(u) \, du .
\end{align*}
\end{proposition}

The proofs of Propositions~\ref{prop3.1} and \ref{prop3.2} are placed in the supplemental material.

\subsection{Multiplier bootstrap approximation}\label{sec3.2}

The distribution of $R^*(s,p)=\sup_{f \in \mathcal{F}} \mathbb{G}^*f $ for $\mathbb{G}^*$ in \eqref{eq3.4} depends on the unknown $\bSigma$ and thus cannot be used for statistical inference. In the following, we consider the use of a Monte Carlo method to simulate a process that mimics $\mathbb{G}^*$, now known as the multiplier (wild) bootstrap method, which is similar to that used in \cite{H96}, \cite{BD03} and \cite{CCK13}, among others.

Let $\hat \bSigma_n$ be the sample covariance matrix based on the data $\{ \bX_i\}_{i=1}^n$ and  $\xi_1, \ldots, \xi_n$ be i.i.d. standard normal random variables that are independent of $\{\varepsilon_i\}_{i=1}^n$ and $\{ \bX_i\}_{i=1}^n$. Then, given $\{ \bX_i\}_{i=1}^n$,
\be
	 \bZ_n = n^{-1/2} \sn \xi_i ( \bX_i  - \bar{\bX}_n ) \sim N(\mo ,  \hat \bSigma_n).  \label{eq3.9}
\ee

The following result shows that the (unknown) distribution of $R^*(s,p)= \sup_{f_{\balpha} \in \mathcal{F}} \frac{ f_{\balpha}(\bZ)}{\sqrt{\balpha^{{\rm T}} \bSigma \balpha} }$ for $\bZ \sta{d}{=} N(\mo, \bSigma)$ can be consistently estimated by the conditional distribution of
\be
	R_n^{\sMB}(s,p) := \sup_{f_{\balpha} \in \mathcal{F}}\f{  f_{\balpha}( \bZ_n )  }{\sqrt{ \balpha^{{{\rm T}}} \hat{\bSigma}_n \balpha  }}.   \label{eq3.10}
\ee

\begin{theorem} \label{thm3.2}
Let Conditions~\ref{moment.cond} and \ref{sampling.condition} hold. Assume that the triplet $(s,p,n)$ satisfies $1\leq s\leq p$ and $s\log( \gamma_s pn )=o(n^{1/5})$. Then as $n\rightarrow \infty$,
\be  \label{eq3.11}
	\sup_{t\geq 0} \big| \PP\big\{ R^*(s,p) \leq t  \big\}  - \PP\big\{   R_n^{{\rm MB}}(s,p) \leq t \,  \big|  \bX_1, \ldots ,\bX_n \big\}  \big|  \xrightarrow  \PP  0 .
\ee
\end{theorem}

\begin{remark} \label{rmk.2}
{\rm
Together, Theorems~\ref{thm3.1} and \ref{thm3.2} show that the maximum spurious correlation $\hat{R}_n(s,p)$ can be approximated in distribution by the multiplier bootstrap statistic $n^{-1/2} R_n^{\sMB}(s,p)$. In practice, when the sample size $n$ is relatively small, the value of $n^{-1/2} R_n^{\sMB}(s,p)$ may exceed 1, which makes it less favorable as a proxy for spurious correlation. To address this issue, we propose using the following corrected bootstrap approximation:
\be
	 {R}_n^{\scMB}(s,p) := \sup_{ f_{\balpha} \in \mathcal{F}} \frac{\sqrt{n}}{| \bxi |_2} \f{  f_{\balpha}( \bZ_n )   }{    \sqrt{ \balpha^{{{\rm T}}} \hat{\bSigma}_n \balpha  }  } ,  \label{eq3.12}
\ee
where $\bxi=(\xi_1, \ldots, \xi_n)^{{{\rm T}}}$ is used in the definition of $\bZ_n$. % \sta{d}{=} N(\mo, \bI_n)$.
By the Cauchy-Schwarz inequality, $n^{-1/2} {R}_n^{\scMB}(s,p)$ is always between 0 and 1. In view of \eqref{eq3.10} and \eqref{eq3.12}, $R_n^{\scMB}(s,p)$ differs from $R_n^{{\rm MB}}(s,p)$ only up to a multiplicative random factor $n^{-1/2} |\bxi |_2$, which in theory is concentrated around 1 with exponentially high probability. Thus, $R_n^{{\rm MB}}$ and $R_n^{\scMB}$ are asymptotically equivalent, and \eqref{eq3.11} remains valid with $R_n^{\sMB}$ replaced by $R_n^{\scMB}$.
}
\end{remark}

\section{Extension to sparse linear models}\label{sec4}

Suppose that the observed response $Y$ and $p$-dimensional covariate
$\bX$ follows the sparse linear model
\begin{equation} \label{eq4.1}
   Y = \bX^{{{\rm T}}} \bbeta^* + \varepsilon ,
\end{equation}
where the regression coefficient $\bbeta^*$ is sparse. The sparsity is typically explored by the LASSO [\cite{TIB96}], the SCAD [\cite{FLI01}], or the MCP [\cite{Zhang10}]. Now it is well-known that, under suitable conditions, the SCAD and the MCP, among other folded concave penalized least-square estimators, also enjoy the unbiasedness property and the (strong) oracle properties. For simplicity, we focus on the SCAD. For a given random sample $\{(\bX_i, Y_i)\}_{i=1}^n$, the SCAD exploits the sparsity by $p_\lambda$-regularization, which minimizes
\begin{equation} \label{eq4.2}
  (2n)^{-1} \sum_{i=1}^n (Y_i - \bX_i^{{{\rm T}}} \bbeta)^2 +   \sum_{j=1}^p p_\lambda(|\beta_j|;a)
\end{equation}
over $\bbeta=(\beta_1,\ldots,\beta_p)^{{{\rm T}}} \in \bbr^p$, where $p_\lambda(\cdot; a)$ denotes the SCAD penalty function [\cite{FLI01}], i.e., $p'_{\lambda}(t;a) = \lambda I(t\leq \lambda) + \frac{(a\lambda-t)_+}{a-1} I(t>\lambda)$ for some $a>2$, and $\lambda = \lambda_n \geq 0$ is a regularization parameter.

Denote by $\mathbb{X}= (\bX_1, \ldots, \bX_n)^{{{\rm T}}}$ the $n\times p$ design matrix, $\mathbb{Y}=(Y_1,\ldots, Y_n)^{{{\rm T}}}$ the $n$-dimensional response vector, and $\beps=(\varepsilon_1,\ldots, \varepsilon_n)^{{{\rm T}}}$, the $n$-dimensional noise vector. Without loss of generality, we assume that $\bbeta^*=(\bbeta_1^{{{\rm T}}}, \bbeta_2^{{{\rm T}}})^{{{\rm T}}}$ with each component of $\bbeta_1 \in \bbr^{s}$ being non-zero and $\bbeta_2= \mo$, such that $S_0:= {\rm supp}(\bbeta^*)=\{1, \ldots, s\}$ is the true underlying sparse model of the indices with $s = | \bbeta^* |_0$. Moreover, write $\mathbb{X}=(\mathbb{X}_1, \mathbb{X}_2)$, where $\mathbb{X}_{1} \in \bbr^{n\times s}$ consists of the columns of $\mathbb{X}$ indexed by $S_0$. In this notation, $\mathbb{Y}= \mathbb{X}\bbeta+\beps= \mathbb{X}_1 \bbeta_1+ \beps$ and the oracle estimator $\hat{\bbeta}^{{\rm oracle}}$ has an explicit form of
\begin{align}   \label{eq4.3}
	\hat{\bbeta}^{{\rm oracle}}_1  = (\mathbb{X}_1^{{{\rm T}}} \mathbb{X}_1)^{-1}   \mathbb{X}_1^{{{\rm T}}} \mathbb{Y} = \bbeta_1 +(\mathbb{X}_1^{{{\rm T}}} \mathbb{X}_1)^{-1}   \mathbb{X}_1^{{{\rm T}}}  \beps ,  \quad  \hat{\bbeta}^{{\rm oracle}}_2 = \mo.
\end{align}
In other words, the oracle estimator is the unpenalized estimator that minimizes $ \sum_{i=1}^n (Y_i - \bX_{i,S_0}^{{{\rm T}}} \bbeta_{S_0})^2$  over the true support set $S_0$.

Denote by $\hat{\beps}^{\,{\rm oracle}}=(\hat{\varepsilon}^{\,{\rm oracle}}_1, \ldots, \hat{\varepsilon}^{\,{\rm oracle}}_n)^{{{\rm T}}} = \bY -\mathbb{X}^{{{\rm T}} } \hat{\bbeta}^{{\rm oracle}}$ the residuals after the oracle fit. Then, we can construct the maximum spurious correlation as in (\ref{eq2.2}), except that $\{ \varepsilon_i \}_{i=1}^n$ is now replaced by $\{ \hat \varepsilon_i^{\,{\rm oracle}} \}_{i=1}^n$, i.e.,
\begin{align} \label{eq4.5}
  & \hat{R}_n^{{\rm oracle}}(1,p) \nn \\
  &  =  \max_{j\in [p]} \frac{   | \sn (\hat{\varepsilon}^{\,{\rm oracle}}_i -  \mathbf{e}_n^{{{\rm T}}} \, \hat{\beps}^{\, {\rm oracle}} ) (X_{ij} - \bar{X}_j) | }{\sqrt{  \sn (\hat{\varepsilon}_i^{\,{\rm oracle}} -  \mathbf{e}_n^{{{\rm T}}} \, \hat{\beps}^{\, {\rm oracle}} )^2 } \cdot \sqrt{ \sn (X_{ij}-\bar{X}_j)^2}} ,
\end{align}
where $\mathbf{e}_n=(1/n , \ldots , 1/n)^{{{\rm T}}} \in \bbr^n$ and $\bar{X}_j = n^{-1} \sn X_{ij}$.  We here deal with the specific case of a spurious correlation of size 1, as this is what is needed for testing the exogeneity assumption \eqref{eq1.2}.

To establish the limiting distribution of $\hat{R}_n^{{\rm oracle}}(1,p)$, we make the following assumptions.

\begin{assumption} \label{reg.cond}
$\mathbb{Y}=\mathbb{X} \bbeta^* + \beps$ with supp$\,(\bbeta^*)=\{1,\ldots, s\}$ and $\beps=(\varepsilon_1,\ldots, \varepsilon_n)^{{{\rm T}}}$ being i.i.d. centered sub-Gaussian satisfying that $K_0 = \| \varepsilon_i \|_{\psi_2}<\infty$. The rows of $\mathbb{X}=(\bX_1,\ldots, \bX_n)^{{{\rm T}}}$ are i.i.d. sub-Gaussian random vectors as in Condition~\ref{moment.cond}.
\end{assumption}

As before, we can assume that $\bSigma=\e (\bX_i \bX_i^{{{\rm T}}})$ is a correlation matrix with diag$\,(\bSigma)=\bI_p$.  Set $d=p-s$ and partition
\begin{align}
	\bSigma = \small\left(
\begin{array}{cc}
\bSigma_{11} & \bSigma_{12} \\
\bSigma_{21} & \bSigma_{22}
\end{array}
\right)  \quad \mbox{ with } \quad \bSigma_{11} \in \bbr^{s\times s}, \, \bSigma_{22} \in \bbr^{d\times d}, \,  \bSigma_{21} = \bSigma_{12}^{{{\rm T}}}.   \label{eq4.6}
\end{align}
Let $\bSigma_{22.1} = (\widetilde \sigma_{jk})_{1\leq j,k\leq d} = \bSigma_{22} - \bSigma_{21} \bSigma_{11}^{-1} \bSigma_{12}$ be the Schur complement of $\bSigma_{11}$ in $\bSigma$.

\begin{assumption} \label{cov.cond}
$\widetilde \sigma_{\min}  =\min_{1\leq j\leq d}  \widetilde \sigma_{jj}$ is bounded away from zero.
\end{assumption}

\begin{theorem} \label{thm4.1}
Assume that Conditions~\ref{reg.cond} and \ref{cov.cond} hold, and that the triplet $(s,p,n)$ satisfies $s\log p=o(\sqrt{n})$ and $\log p=o(n^{1/7})$. Then the maximum spurious correlation $\hat{R}^{{\rm oracle}}_n(1,p)$ in \eqref{eq4.5} satisfies that, as $n\rightarrow \infty$,
\begin{align} \label{eq4.7}
  \sup_{t\geq 0} \big|    \PP\big\{ \sqrt{n} \hat{R}_n^{{\rm oracle}}(1,p)   \leq t \big\}    -  \PP\big( |\widetilde{\bZ} |_{\infty} \leq t \big) \big| \rightarrow 0,
\end{align}
where $\widetilde{\bZ} \sta{d}{=}  N(\mo,\bSigma_{22.1} )$ is a $d$-variate centered Gaussian random vector with covariance matrix $\bSigma_{22.1}$.
\end{theorem}

As $p_{\lambda}$ is a folded-concave penalty function, \eqref{eq4.2} is a non-convex
optimization problem. The local linear approximation (LLA) algorithm can be applied to produce a certain local minimum for any fixed initial solution [\cite{ZL08}, \cite{FXZ14}]. In particular, \cite{FXZ14} prove that the LLA algorithm can deliver the oracle estimator in the folded concave penalized problem with overwhelming probability if it is initialized by some appropriate initial estimator.

Let $\hat{\bbeta}^{{\rm LLA}}$ be the estimator computed via the one-step LLA algorithm initiated by the LASSO estimator [\cite{TIB96}].  That is,
\be  \label{eq4.8}
 \hat{\bbeta}^{{\rm LLA}}  = \arg\min_{\bbeta}  \bigg\{ (2n)^{-1} \sn ( Y_i-\bX_i^{{{\rm T}}} \bbeta )^2 + \sum_{j=1}^p p_\lambda'(| \hat{\beta}_j^{\,{\rm LASSO}}|)
  |\beta_j | \bigg\} , 
\ee
where $p_\lambda$ is a folded concave penalty, such as the SCAD and MCP penalties, and $
 \hat{\bbeta}^{{\rm LASSO}}  = \arg\min_{\bbeta}  \{ (2n)^{-1}  \sn ( Y_i-\bX_i^{{{\rm T}}} \bbeta )^2 + \lambda_{{\rm }} |\bbeta |_1 \}$. Accordingly, denote by $\hat{R}^{{\rm LLA}}_n(1,p)$ the maximum spurious correlation as in \eqref{eq4.5} with $\hat{\varepsilon}^{\, {\rm oracle}}_i$ replaced by $\hat{\varepsilon}_i^{\,{\rm LLA}} = Y_i - \bX_i^{{{\rm T}}} \hat{\bbeta}^{{\rm LLA}}$. Applying Theorem~\ref{thm4.1}, we derive the limiting distribution of $\hat{R}^{{\rm LLA}}_n(1,p)$ under suitable conditions. First, let us recall the {\it Restricted Eigenvalue} concept formulated by \cite{BRT09}.

\begin{definition}  \label{RE.condition}
For any integer $s_0 \in [p]$ and positive number $c_0$, the RE$\,(s_0,c_0)$ parameter $\kappa(s_0,c_0,\mathbf{A})$ of a $p\times p$ matrix $\mathbf{A}$ is defined as
\be \label{eq4.9}
 \kappa(s_0,c_0,\mathbf{A}) := \min_{S \subseteq [p] : |S| \leq s_0} \; \; \;\min_{\bdelta \neq 0 : |\bdelta_{S^{{\rm c}}}|_1 \leq c_0 | \bdelta_S |_1} \frac{    \bdelta^{{{\rm T}}} \mathbf{A}  \bdelta }{ | \bdelta_S |_2^2 } .
\ee
\end{definition}

\begin{theorem}  \label{thm4.2}
Assume that Conditions~\ref{reg.cond} and \ref{cov.cond} hold, the minimal signal strength of $\bbeta^*$ satisfies $
\min_{j\in S_0 } |\beta_j | > (a+1) \lambda$ for $a,\lambda$ as in \eqref{eq4.2}, and that the triplet $(s,p,n)$ satisfies $s\log p=o(\sqrt{n})$, $\frac{s\log p}{\kappa(s,3+\epsilon, \bSigma)}=o(n)$ for some $\epsilon>0$ and $\log p=o(n^{1/7})$. If the regularization parameters $(\lambda, \lambda_{{\rm LASSO}})$ are such that $\lambda\geq \frac{8\sqrt{s}}{\kappa(s,3,\bSigma)}\lambda_{{\rm LASSO}}$ and $\lambda_{{\rm LASSO}} \geq C K_0 \sqrt{( \log p )/n}$ for $C>0$ large enough, then as $n\rightarrow \infty$,
\begin{align} \label{eq4.10}
  \sup_{t\geq 0} \big|    \PP\big\{ \sqrt{n} \hat{R}_n^{{\rm LLA}}(1,p)   \leq t \big\}    -  \PP \big( |\widetilde{\bZ} |_{\infty} \leq t \big) \big| \rightarrow 0,
\end{align}
where $\widetilde{\bZ} \sta{d}{=}  N(\mo,\bSigma_{22.1} )$.
\end{theorem}

\section{Applications to high-dimensional inferences}
\label{sec5}

This section outlines three applications in high-dimensional statistics. The first determines whether discoveries by machine learning and data mining techniques are any better than those reached by chance. Second, we show that the distributions of maximum spurious correlations can also be applied to model selection. In the third application, we validate the fundamental assumption of exogeneity (\ref{eq1.2}) in high dimensions.

\subsection{Spurious discoveries} \label{sec5.1}

Let $q_\alpha^{\scMB}(s, p)$ be the upper $\alpha$-quantile of the random variable $R_n^{\scMB}(s,p)$ defined by \eqref{eq3.12}.  Then,
an approximate $1-\alpha$ upper confidence limit of the spurious correlation is given by $q_\alpha^{\scMB}(s, p)$. In view of Theorems~\ref{thm3.1} and \ref{thm3.2}, we claim that
\begin{equation} \label{eq5.1}
  \PP\big\{ \hat{R}_n(s,p) \leq q_\alpha^{\scMB}(s, p) \big\} \to 1 - \alpha.
\end{equation}
To see this, recall that $R_n^{\scMB} = \sqrt{n} R_n^{\sMB} /  | \bxi |_2$ for $\bxi=(\xi_1,\ldots, \xi_n)^{{{\rm T}}}$ as in \eqref{eq3.12}, and given $\{\bX_i\}_{i=1}^n$, $R_n^{\sMB}$ is the supremum of a Gaussian process. Let $F_n^{\sMB}(t) = \PP\{ R_n^{\sMB}(s,p) \leq t \, | \bX_1,\ldots, \bX_n \} $ be the (conditional) distribution function of $R_n^{\sMB}$ and define $t_0 = \inf\{ t: F_n^{\sMB}(t) > 0 \}$. By Theorem~11.1 of \cite{DLS98}, $F_n^{\sMB}$ is absolutely continuous with respect to the Lebesgue measure and is strictly increasing on $(t_0 , \infty)$, indicating that $\PP\{R_n^{\scMB}  \leq q_\alpha^{\scMB}(s, p) | \bX_1, \ldots, \bX_n \} = \alpha$ almost surely. This, together with \eqref{eq3.3} and \eqref{eq3.11}, proves \eqref{eq5.1} under Conditions~\ref{moment.cond}, \ref{sampling.condition}, and when $s\log(\gamma_s pn) = o(n^{1/7})$.

Let $\hat{Y}_i$ be fitted values using $s$ predictors indexed by $\hat{S}$ selected by a data-driven technique and $Y_i$ be the associated response value. They are denoted in the vector form by
$\hat{\mathbb{Y}}$ and $\mathbb{Y}$, respectively. If
\begin{equation}  \label{eq5.2}
   \big| \widehat{\mbox{corr}}_n(\mathbb{Y}, \hat{\mathbb{Y}}) \big| \leq q_\alpha^{\scMB}(s, p),
\end{equation}
then the discovery of variables $\hat{S}$ can be regarded as spurious; that is, no better than by chance.  Therefore, the multiplier bootstrap quantile $q_\alpha^{\scMB}(s, p)$ provides an important critical value and yardstick for judging whether the discovery is spurious, or whether the selected set $\hat{S}$ includes too many spurious variables. This yardstick is independent of the method used in the fitting.

\begin{remark}
\label{rmk5.1}
{\rm
The problem of judging whether the discovery is spurious is intrinsically different from that of testing the global null hypothesis $H_0: \bbeta^* =\mo$, which itself is an important problem in high-dimensional statistical inference and has been well-studied in the literature since the seminal work of \cite{Gvv06}. For example, the global null hypothesis $H_0:  \bbeta^* = \mo$ can be rejected by a test; still, the correlation between $Y$ and the variables $\hat{S}$ selected by a statistical method can be smaller than the maximum spurious correlation, and we should interpret the findings of $\hat{S}$ with caution. We need either more samples or more powerful variable selection methods.  This motivates us to derive the distribution of the maximum spurious correlation $\hat R_n(s,p)$. This distribution serves as an important benchmark for judging whether the discovery (of $s$ features from $p$ explanatory variables based on a sample of size $n$) is spurious. The magnitude of $\hat R_n(s,p)$ gives statisticians an idea of how big a spurious correlation can be, and therefore an idea of how much the covariates really contribute to the regression for a given sample size.
}
\end{remark}

\subsection{Model selection}
\label{secMS}

{In the previous section, we consider the reference distribution of the maximum spurious correlation statistic $\hat{R}_n(s,p)$ as a benchmark for judging whether the discovery of $s$ significant variables (among all of the $p$ variables using a random sample of size $n$) is impressive, regardless of which variable selection tool is applied. In this section, we show how the distribution of $\hat{R}_n(s,p)$ can be used to select a model. Intuitively, we would like to select a model that fits better than the spurious fit.  This limits the candidate sets of models and provides an upper bound on the model size.  In our experience, this upper bound itself provides a model selector.}

{
We now use LASSO as an illustration of the above idea.
Owing to spurious correlation, almost all of the variable selection procedures will, with high probability, select a number of spurious variables in the model so that the selected model is over-fitted. For example,
the LASSO method with the regularization parameter selected by cross-validation typically selects a far larger model size, as the bias caused by the $\ell_1$ penalty forces the cross-validation procedure to choose a smaller value of $\lambda$. Thus, it is important to stop the LASSO path earlier and the quantiles of $\hat{R}_n(s,p)$ provide useful guards.}

{Specifically, consider the LASSO estimator $\hat{\bbeta}_\lambda$ for the sparse linear model \eqref{eq4.1} with $\hat{s}_\lambda= | \mathrm{supp}(\hat{\bbeta}_\lambda ) |$, where $\lambda>0$ is the regularization parameter. We consider the LASSO solution path with the largest knot $\lambda_{\mathrm{ini}} := | \mathbb{X}^{{\rm T}} \mathbb{Y} |_{\infty}$ and the smallest knot $\lambda_{\mathrm{cv}}$ selected by ten-fold cross-validation. To avoid over-fitting, we propose using $q_\alpha^{\scMB}$ as a guide to choose the regularization parameter that guards us from selecting too many spurious variables. For each $\lambda$ in the  path, we compute $\hat{\mathrm{corr}}_n(\hat{\mathbb{Y}}_\lambda, \mathbb{Y})$, the sample correlation between the post-LASSO fitted and observed responses, and $q_\alpha^{\scMB}(\hat{s}_\lambda, p)$. 
Let $\hat{\lambda}_{\alpha}$ be the largest $\lambda$ such that the sign of $\hat{\mathrm{corr}}_n(\hat{\mathbb{Y}}_\lambda, \mathbb{Y}) - q_\alpha^{\scMB}(\hat{s}_\lambda, p)$ is nonnegative and then flips in the subsequent knot.
%Let $\hat{\lambda}_{\alpha}$ be the smallest $\lambda$ satisfying $\hat{\mathrm{corr}}_n(\hat{\mathbb{Y}}_\lambda, \mathbb{Y})  \geq q_\alpha^{\scMB}(\hat{s}_\lambda, p)$. 
The selected model is given by $\hat{S}_{\alpha} = \mathrm{supp}(\hat{\bbeta}_{ \hat \lambda_\alpha })$. As demonstrated by the simulation studies in Section~\ref{sec6.2+}, this procedure selects a much smaller model size that is closer to the real data.}

\subsection{Validating exogeneity}
\label{sec5.2}

\cite{FLI14} show that the exogenous condition \eqref{eq1.2} is necessary for penalized least-squares to achieve a model selection consistency.  They question the validity of such an exogeneous assumption, as it imposes too many equations.  They argue further that even when the exogenous model holds for important variables $\bX_S$, i.e.,
\begin{equation} \label{eq5.3}
    Y = \bX_S^{{{\rm T}}} \bbeta_{S}^* + \varepsilon, \qquad \e ( \varepsilon \bX_S ) = \mo,
\end{equation}
the extra variables $\bX_N$ (with $N = S^{{\rm c}}$) are collected in an effort to cover the unknown set $S$ --- but no verification of the conditions
\begin{equation} \label{eq5.4}
    \e (  \varepsilon \bX_N ) = \e \{ (  Y - \bX_S^{{{\rm T}}} \bbeta_{S}^*) \bX_N \} =\mo
\end{equation}
has ever been made.  The equality $\e \{ (  Y - \bX_S^{{{\rm T}}} \bbeta_{S}^*) X_j \} = 0$ in \eqref{eq5.4} holds by luck for some covariate $X_j$, but it can not be expected that this holds for all $j \in N$. They propose a focussed generalized method of moment (FGMM) to avoid the unreasonable assumption (\ref{eq5.4}).  Recognizing \eqref{eq5.3} is not identifiable in high-dimensional linear models, they impose additional conditions such as $\e ( \varepsilon \bX_S^2 ) = \mo$.

Despite its fundamental importance to high-dimensional statistics, there are no available tools for validating (\ref{eq1.2}).  Regarding (\ref{eq1.2}) as a null hypothesis, an asymptotically $\alpha$-level test can be used to reject assumption (\ref{eq1.2}) when
\begin{equation} \label{eq5.5}
  \hat{T}_{n,p} = \max_{j \in [p]} \big| \sqrt{n} \, \wh{\mbox{corr}}_n( X_j, \varepsilon) \big| \geq  q_\alpha^{\scMB}(1, p).
\end{equation}
By Theorems~\ref{thm3.1} and \ref{thm3.2}, the test statistic has an approximate size $\alpha$.  The $p$-value of the test can be computed via the distribution of the Gaussian multiplier process $R_n^{{\rm CMB}}(1,p)$.

As pointed out in the introduction, when the components of $\bX$ are weakly correlated, the distribution of the maximum spurious correlation does not depend very sensitively on $\bSigma$. See also Lemma~6 in \cite{CLX14}. In this case, we can approximate it by the identity matrix, and hence one can compare the renormalized test statistic
\begin{equation} \label{eq5.6}
    J_{n,p} = \hat{T}_{n,p}^2 - 2 \log p + \log(\log p)
\end{equation}
with the limiting distribution in (\ref{eq3.6}).  The critical value for test statistic $J_{n,p}$ is
\begin{equation} \label{eq5.7}
   J_\alpha = -2 \log\{ -\sqrt{\pi} \log(1-\alpha)\} ,
\end{equation}
and the associated $p$-value is given by
\begin{equation} \label{eq5.8}
   \exp ( - \pi^{-1/2}  e^{-J_{n,p}/2} ).
\end{equation}
Expressions (\ref{eq5.7}) and (\ref{eq5.8}) provide analytic forms for a quick validation of the exogenous assumption (\ref{eq1.2}) under weak dependence. In general, we recommend using the wild bootstrap, which takes into account the correlation effect and provides more accurate estimates especially when the dependence is
strong. See \cite{CZZZ2017} for more empirical evidences.

In practice, $\varepsilon$ is typically unknown to us. Therefore, $\hat{T}_{n,p}$ in \eqref{eq5.5} is calculated using the fitted residuals $\{\hat{\varepsilon}_i^{\,{\rm LLA}}\}_{i=1}^n$.  In view of Theorem~\ref{thm4.2}, we need to adjust the null distribution according to \eqref{eq4.10}. By Theorem~\ref{thm3.2}, we adjust the definition of the process $\bZ_n$ in \eqref{eq3.9} by
\be \label{eq5.9}
\bZ_n^{{\rm LLA}} = n^{-1/2} \sn \xi_i ( \bX_i^{{\rm LLA}}  - \overline{\bX}^{{\rm LLA}}_n )   \in \bbr^{ p- |\hat{S}|} ,
\ee
where $\bX_i^{{\rm LLA}} = \bX_{i,\hat{N}} - \hat{\bSigma}_{\hat{N} \hat{S}} \hat{\bSigma}_{\hat{S} \hat{S}}^{-1}  \bX_{i,\hat{S}}$ is the residuals of $\bX_{\hat N}$ regressed on $\bX_{\hat S}$, where $\hat{S}$ is the set of selected variables, $\hat{N} =[p] \setminus \hat{S}$, and $\hat{\bSigma}_{S S'}$ denotes the sub-matrix of $\hat{\bSigma}_n$ containing entries indexed by $(k,\ell) \in S \times  S'$. From \eqref{eq5.9}, the multiplier bootstrap approximation of $|\widetilde{\bZ}|_\infty$ is $R_n^{\sMB,{\rm LLA}}(1,p)=|\hat{\bD}^{-1/2}\bZ_n^{{\rm LLA}}  |_\infty$, where $\hat{\bD}=$ diagonal matrix of the sample covariance matrix of $\{ \bX_i^{{\rm LLA}}\}_{i=1}^n$. Consequently, we reject \eqref{eq1.2} if $\hat{T}_{n,p} > q^{\sMB, {\rm LLA}}_\alpha(1,p) $, where $q^{\sMB, {\rm LLA}}_\alpha(1,p)$ is the (conditional) upper $\alpha$-quantile of $R_n^{\sMB , {\rm LLA}}(1,p)$ given $\{\bX_i\}_{i=1}^n$.

\begin{remark}
{{\rm To the best of our knowledge, this is the first paper to consider testing the exogenous assumption \eqref{eq1.2}, for which we use the maximum correlation between covariates and fitted residuals as the test statistic. A referee kindly informed us in his/her review report that in the context of specification testing, \cite{CCK13} propose a similar extreme value statistic and use the multiplier bootstrap to compute a critical value for the test. To construct marginal test statistics, they use self-normalized covariances between generated regressors and fitted residuals obtained via ordinary least squares, whereas we use sample correlations between the covariates and fitted residuals obtained by the LLA algorithm. We refer readers to Appendix~M in the supplementary material of \cite{CCK13} for more details.}}
\end{remark}

\section{Numerical studies} \label{sec6}

In this section, Monte Carlo simulations are used to examine the finite-sample performance of the bootstrap approximation (for a given data set) of the distribution of the maximum spurious correlation (MSC).

\subsection{Computation of spurious correlation}

First, we observe that $\hat{R}_n(s,p)$ in \eqref{eq2.2} can be written as
$ \hat{R}^2_n(s,p) =  \hat{\sigma}^{-2}_{\varepsilon}  \max_{S\subseteq [p]: |S|=s}  \mathbf{v}_{n,S}^{{{\rm T}}} \hat \bSigma^{-1}_{S S} \mathbf{v}_{n,S}$, where $\hat{\sigma}^2_\varepsilon = n^{-1}\sn (\varepsilon_i -\bar{\varepsilon}_n)^2$ and $\mathbf{v}_n =n^{-1} \sn (\varepsilon_i -\bar{\varepsilon}_n)(\bX_i - \bar{\bX}_n)$. Therefore, the computation of $\hat{R}_{n}(s,p)$ requires solving the combinatorial optimization problem
\be
	\hat{S} %:=  \arg\max_{S\subseteq [p]: |S|=s} T^2_n(S)
    = \arg\max_{S\subseteq [p] : |S|=s}  \mathbf{v}_{n,S}^{{{\rm T}}} \hat \bSigma^{-1}_{S S}  \mathbf{v}_{n,S} .	 \label{eq6.1}
\ee

It is computationally intensive to obtain $\hat{S}$ for large values of $p$ and $s$ as one essentially needs to enumerate all ${p \choose s}$ possible subsets of size $s$ from $p$ covariates. A fast and easily implementable approach is to use the stepwise addition (forward selection) algorithm as in \cite{FGH12}, which results in some value that is no larger than $\hat{R}_n(s,p)$ but avoids computing all ${p\choose s}$ multiple correlations in \eqref{eq6.1}. Note that the optimization \eqref{eq6.1} is equivalent to finding the best subset regression of size $s$.
When $p$ is relatively small, say if $p$ ranges from $20$ to $40$, the branch-and-bound procedure is commonly used for finding the best subset of a given size that maximizes multiple $R^2$ [\cite{BS05}]. However, this approach becomes computational infeasible very quickly when there are hundreds or thousands of potential predictors. As a trade-off between approximation accuracy and computational intensity,  we propose using a two-step procedure that combines the stepwise addition and branch-and-bound algorithms. First, we use the forward selection to pick the best $d$ variables, say $d=40$, which serves as a pre-screening step. Second, across the ${d\choose s}$ subsets of size $s$, the branch-and-bound procedure is implemented to select the best subset that maximizes the multiple-$R^2$. This subset is used as an approximate solution to \eqref{eq6.1}.   Note that when $s > 40$, which is rare in many applications, we only use the stepwise addition to reduce the computational cost.

\subsection{Accuracy of the multiplier bootstrap approximation}\label{sec6.1}

For the first simulation, we consider the case where the random noise $\varepsilon$ follows the uniform distribution standardized so that $\e (\varepsilon)=0$ and $\e ( \varepsilon^2)=1$. Independent of $\varepsilon$, the $p$-variate vector $\bX$ of covariates has i.i.d. $N(0,1)$ components. In the results reported in Table~\ref{tab1}, the ambient dimension $p=2000$, the sample size $n$ takes a value in $\{400, 800, 1200\}$, and $s$ takes a value in $\{1, 2, 5, 10\}$. For a given significance level $\alpha\in (0,1)$, let $q_\alpha(s,p)$ be the upper $\alpha$-quantile of $\hat{R}_n(s,p)$ in \eqref{eq2.1}. For each data set $\mathcal{X}_n = \{\bX_1, \ldots, \bX_n\}$, a direct application of Theorems~\ref{thm3.1} and \ref{thm3.2} is that
$$
c^{{\rm MB}}(\mathcal{X}_n,\alpha) :=	\PP \big\{ R_n^{{\rm CMB}}(s,p) \geq  q_\alpha(s,p) | \mathcal{X}_n \big\} \rightarrow \alpha \ \ \mbox{ as } n\to \infty.
$$
The difference $c^{{\rm MB}}(\mathcal{X}_n,\alpha)- \alpha$, however, characterizes the extent of the size distortions and the finite-sample accuracy of the multiplier bootstrap approximation (MBA). Table~\ref{tab1} summarizes the mean and the standard deviation (SD) of $c^{{\rm MB }}(\mathcal{X}_n,\alpha) $ based on 200 simulated data sets with $\alpha \in \{0.05,0.1\}$. The $\alpha$-quantile $q_\alpha(s,p)$ is calculated from 1600 replications, and $c^{{\rm MB }}(\mathcal{X}_n,\alpha)$ for each data set is simulated based on 1600 bootstrap replications. In addition, we report in Figure~\ref{Fig3} the distributions of the maximum spurious correlations and their multiplier bootstrap approximations conditional on a given data set $\mathcal{X}_n$ when $p\in \{2000, 5000\}$, $s\in \{1,2,5, 10\}$ and $n=400$. Together, Table~\ref{tab1} and Figure~\ref{Fig3} show that the multiplier bootstrap method indeed provides a quite good approximation to the (unknown) distribution of the maximum spurious correlation.
\begin{table}[h]
\centering
\caption{\label{size} (Isotropic case) The mean of 200 empirical sizes $c^{{\rm MB}}(\cdot, \alpha) \times 100$, with its estimate of SD in the parenthesis, when $p=2000$, $s= 1, 2, 5, 10$, $n = 400, 800, 1200$, and $\alpha = 0.1, 0.05$}
\label{tab1}
\scriptsize{\begin{tabular}{ccccccccc}
&\multicolumn{2}{c}{$s=1$}&\multicolumn{2}{c}{$s=2$}&\multicolumn{2}{c}{$s=5$}&\multicolumn{2}{c}{$s=10$}\\
\cline{2-3}\cline{4-5}\cline{6-7}\cline{8-9}\vspace{-0.3cm}\\
$n$  & $\alpha=0.1$ &   $\alpha=0.05$ &  $\alpha=0.1$ & $\alpha=0.05$ &  $\alpha=0.1$ & $\alpha=0.05$ &  $\alpha=0.1$ & $\alpha=0.05$
\\ \midrule
$400 $&$9.54$ &$4.68$&$9.13$ &$4.38$&$9.08$ &$3.78$&$8.67$ &$4.44$\\
&(0.643)&(0.294)&(0.568)&(0.284)&(0.480)&(0.245)&(0.506)&(0.291) \\
\midrule%\midrule
$800 $&$9.43$ &$4.93$&$9.47$ &$4.42$&$9.73$ &$4.73$&$9.94$ &$5.62$\\
&(0.444)&(0.296)&(0.474)&(0.296)&(0.488)&(0.294)&(0.557)&(0.331) \\
\midrule%\midrule
$1200 $&$9.09$ &$4.32$&$9.00$ &$4.46$&$9.42$ &$4.87$&$9.97$ &$5.15$\\
&(0.507)&(0.261)&(0.542)&(0.278)&(0.543)&(0.322)&(0.579)&(0.318) \\
\bottomrule
\end{tabular}}
\end{table}\par

\begin{figure}[hbtp!]
  \centering
  \includegraphics[width=5in, trim=1.7cm 0.75cm 1.7cm 0.75cm, clip]{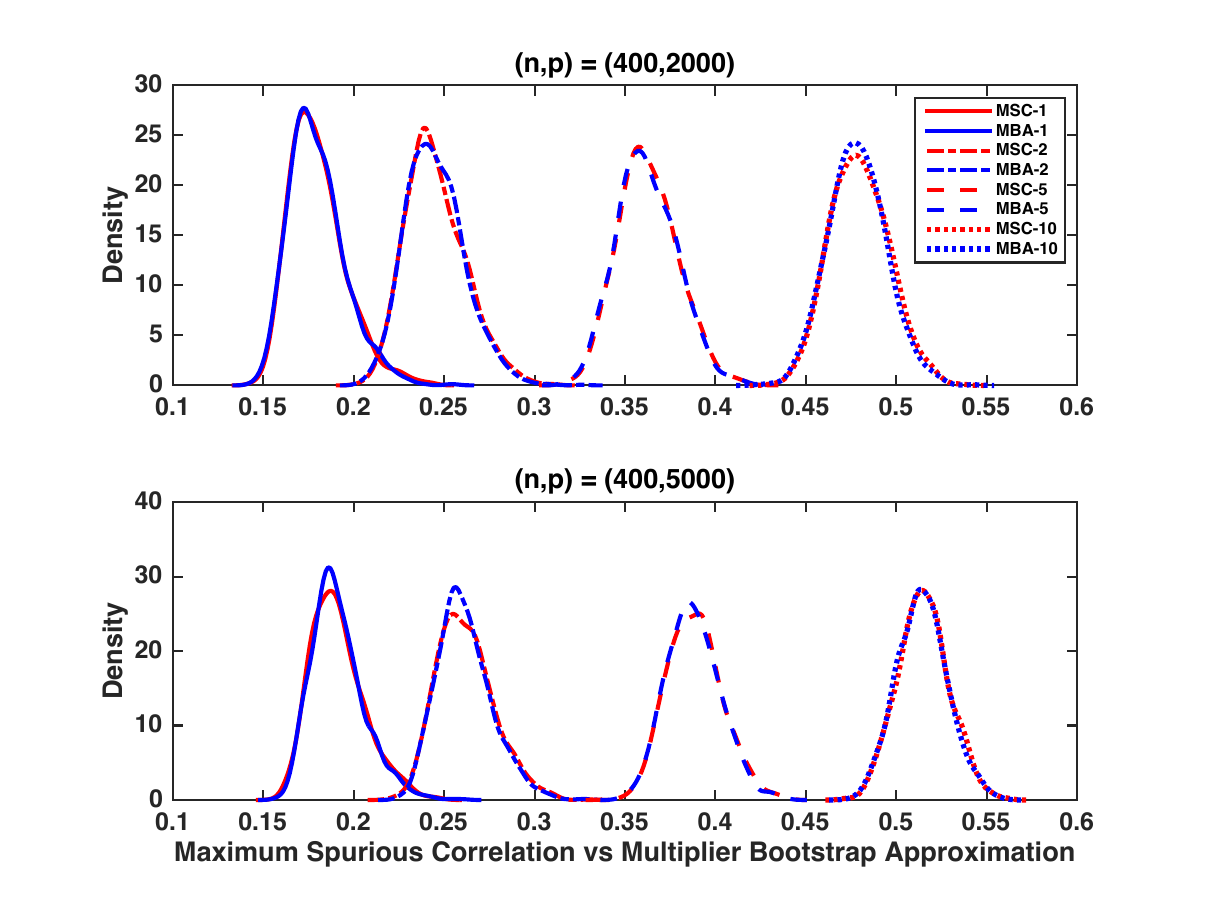}
  \caption{Distributions of maximum spurious correlations (blue) and multiplier bootstrap approximations (for a given data set; red) based on 1600 simulations with combinations of $p = 2000, 5000$, $s =1 , 2, 5, 10$, and $n=400$ when $\bSigma$ is an identity matrix. \label{Fig3} }
\end{figure}

For the second simulation, we focus on an anisotropic case where the covariance matrix $\bSigma$ of $\bX$ is non-identity, and the condition number of $\bSigma$ is well-controlled. Specifically, we assume that $\varepsilon$ follows the centered Laplace distribution rescaled so that $\e (\varepsilon)=0$ and $\e ( \varepsilon^2)=1$. To introduce dependence among covariates, first we denote with $\mathbf{A}$ a $10\times 10$ symmetric positive definite matrix with a pre-specified condition number $c>1$ and let $\rho\in (0,1)$. Then the $p$-dimensional vector $\bX$ of the covariates is generated according to $\bX = \mathbf{G}_1(\rho) \, \mathbf{Z}_1 + \mathbf{G}_2(\rho) \, \mathbf{Z}_2$,
where $\mathbf{Z}_1 \sta{d}{=} N(\mo ,\mathbf{A})$, $\mathbf{Z}_2 \sta{d}{=} N(\mo, \mathbf{I}_{p-10})$, and $\mathbf{G}_1(\rho)  \in \bbr^{p\times 10}$, $\mathbf{G}_2(\rho) \in \bbr^{p\times (p-10)}  $ are given respectively by $\mathbf{G}_1(\rho)^{{{\rm T}}} = (\mathbf{I}_{10}, \frac{\rho}{\sqrt{1+\rho^2}} \mathbf{I}_{10},  \mathbf{G}_{11}(\rho)^{{{\rm T}}})$ with
$$	
  \mathbf{G}_{11}(\rho)  = \frac{1-\rho}{\sqrt{1+(1-\rho)^2}} {\small\left(
\begin{array}{cccc}
1 &  0 & \ldots  & 0 \\
 1 & 0  & \ldots & 0 \\
 \vdots &  \vdots & \cdots & \vdots \\
 1 & 0 & \ldots & 0
\end{array}
\right) \in \bbr^{(p-20)\times 10} }
$$
and
\begin{align*}
\mathbf{G}_2(\rho) =  {\small\left(
\begin{array}{cc}
  \mathbf{0}_{10 \times 10} & \mathbf{0}_{10 \times (p-20)}   \\
 \frac{1}{\sqrt{1+\rho^2}} \mathbf{I}_{10}    &   \mathbf{0}_{10\times (p-20)} \\
    \mathbf{0}_{(p-20)\times 10}  &  \frac{1}{\sqrt{1+(1-\rho)^2}} \mathbf{I}_{p-20}
\end{array}
\right) } .
\end{align*}
In particular, we take $c=5$ and $\rho=0.8$ in the simulations reported in Table~\ref{tab2}, which summarizes the mean and the standard deviation (SD) of the size $c^{{\rm MB}}(\mathcal{X}_n,\alpha)$ based on 200 simulated data sets with $\alpha \in \{0.05,0.1\}$. Comparing the simulation results shown in Tables ~\ref{tab1} and ~\ref{tab2}, we find that the bootstrap approximation is fairly robust against heterogeneity in the covariance structure of the covariates.

\begin{table}[h]
\centering
\caption{\label{size} (Anisotropic case) Mean of 200 empirical sizes $c^{{\rm MB}}(\cdot, \alpha)\times 100$, with its estimate of SD in the parenthesis, when $p=2000$, $s= 1, 2, 5, 10$, $n = 400, 800, 1200$, and $\alpha = 0.1, 0.05$}
\label{tab2}
\scriptsize{\begin{tabular}{ccccccccc}
&\multicolumn{2}{c}{$s=1$}&\multicolumn{2}{c}{$s=2$}&\multicolumn{2}{c}{$s=5$}&\multicolumn{2}{c}{$s=10$}\\
\cline{2-3}\cline{4-5}\cline{6-7}\cline{8-9}\vspace{-0.3cm}\\
$n$  & $\alpha=0.1$ &  $\alpha=0.05$ &  $\alpha=0.1$ & $\alpha=0.05$ &  $\alpha=0.1$ & $\alpha=0.05$ &  $\alpha=0.1$ & $\alpha=0.05$
\\ \midrule
$400 $&$9.83$ &$4.39$&$9.04$ &$4.75$&$9.27$ &$4.65$&$9.34$ &$4.53$\\
&(0.426)&(0.222)&(0.402)&(0.208)&(0.492)&(0.273)&(0.557)&(0.291) \\
\midrule%\midrule
$800 $&$10.18$ &$5.19$&$10.48$ &$5.12$&$9.98$ &$4.86$&$9.21$ &$4.73$\\
&(0.556)&(0.296)&(0.519)&(0.272)&(0.576)&(0.220)&(0.474)&(0.269) \\
\midrule%\midrule
$1200 $&$9.42$ &$4.41$&$9.60$ &$5.71$&$9.90$ &$4.85$&$10.11$ &$5.19$\\
&(0.500)&(0.233)&(0.543)&(0.339)&(0.553)&(0.333)&(0.606)&(0.337) \\
\bottomrule
\end{tabular}}
\end{table}\par

\subsection{Detecting spurious discoveries} \label{sec6.2}

To examine how the multiplier bootstrap quantile $q^{\scMB}_\alpha(s,p)$ (see Section~\ref{sec5.1}) serves as a benchmark for judging whether the discovery is spurious, we compute the Spurious Discovery Probability (SDP) by simulating $200$ data sets from \eqref{eq4.1} with $n=100, 120, 160$, $p=400$, $\bbeta^* = (1, 0, -0.8, 0, 0.6, 0, -0.4, 0, \ldots, 0)^{{{\rm T}}} \in \bbr^{p}$, and standard Gaussian noise $\varepsilon \sta{d}{=} N(0,1)$. For some integer $s\leq r\leq p$, we let $\bx \sta{d}{=} N(\mo, \bI_r) $ be an $r$-dimensional Gaussian random vector. Let $\bGamma_r$ be a $p\times r$ matrix satisfying $\bGamma_r^{{{\rm T}}} \bGamma_r=\bI_r$. The rows of the design matrix $\mathbb{X}$ are sampled as i.i.d. copies from $\bGamma_r \bx \in \bbr^p$, where $r$ takes a value in $\{120, 160, 200, 240, 280, 320, 360 \}$.  To save space, we give the numerical results for the case of non-Gaussian design and noise in the supplementary material.

Put $\mathbb{Y}=(Y_1,\ldots, Y_n)^{{{\rm T}}}$ and let $\hat{\mathbb{Y}} = \mathbb{X}_{\hat{S}} \hat{\bbeta}^{{\rm pLASSO}}$ be the $n$-dimensional vector of fitted values, where $\hat{\bbeta}^{{\rm pLASSO}}=(\mathbb{X}_{\hat{S}}^{{{\rm T}}}  \mathbb{X}_{\hat{S}})^{-1} \mathbb{X}_{\hat{S}}^{{{\rm T}}}  \mathbb{Y}$ is the post-LASSO estimator using covariates selected by the ten-fold cross-validated LASSO estimator. Let $\hat{s}=|\hat{S}|_0$ be the number of variables selected. For $\alpha \in (0,1)$, the level-$\alpha$ SDP is defined as $\PP \{ |\hat{{\rm corr}}_n(\mathbb{Y}, \hat{ \mathbb{Y} })| \leq q_{\alpha}^{\scMB}(\hat{s},p)  \} $.  As the simulated model is not null, this SDP is indeed a type II error.  Given $\alpha=5 \%$ and for each simulated data set, $q_{\alpha}^{\scMB}(s,p)$ is computed based on 1000 bootstrap replications. Then we compute the empirical SDP based on 200 simulations. The results are given in Table~\ref{tab3}.

In this study, the design matrix is chosen so that there is a low-dimensional linear dependency in the high-dimensional covariates. The collected covariates are highly correlated when $r$ is much smaller than $p$. It is known that collinearity and high dimensionality add difficulty to the problem of variable selection and deteriorate the performance of the LASSO. The smaller the $r$ is, the more severe the problem of collinearity becomes. As reflected in Table~\ref{tab3}, the empirical SDP increases as $r$ decreases, indicating that the correlation between fitted and observed responses is more likely to be smaller than the spurious correlation.

\begin{table}[h]
\centering
\caption{\label{size} Empirical $\alpha$-level spurious discovery probability (ESDP) based on 200 simulations when $p=400$, $n=100, 120, 160$, and $\alpha=5\%$.}
\label{tab3}
\scriptsize{\begin{tabular}{cccccccc}
   & $r=120$ &  $r=160$ &  $r=200$ & $r=240$ &  $r=280$ & $r=320$ & $r=360$   \\ \midrule
  $ n = 100 $ & $ 0.6950  $ & $ 0.6650 $ &$ 0.6000 $&$ 0.5200 $ &$ 0.5100 $ & $0.4500$ & $0.4000$    \\
\midrule%\midrule
$ n=120 $ & $ 0.6600  $ & $ 0.5350 $ &$ 0.3350 $&$ 0.3800 $ &$ 0.2500 $ & $0.2850$ & $0.1950$    \\
\midrule%\midrule
$ n=160 $& $ 0.1950   $ & $0.1300$ & $0.0500$ & $0.0400$ & $0.0550$ & $0.0700$ & $0.0250$   \\
\bottomrule
\end{tabular}}
\end{table}\par

\subsection{Model selection} \label{sec6.2+}

We demonstrate the idea in Section~\ref{secMS} through the following toy example. Consider the linear model \eqref{eq4.1} with $(n,p)=(160, 400)$ and $\bbeta^* = (1, 0, -0.8, 0, 0.6, 0, -0.4, 0, \ldots, 0)^{{{\rm T}}} \in \bbr^p$. The covariate vector is taken to be $\bX = \bGamma  \bx $ with $\bx = (x_1, \ldots, x_{200})^{{\rm T}}$, where $x_1,\ldots, x_{200}$ are i.i.d. random variables following the continuous uniform distribution on $[-1,1]$ and $\bGamma$ is a $400 \times 200$ matrix satisfying $\bGamma^{{{\rm T}}} \bGamma=\bI_{200}$. The noise variable $\varepsilon$ follows a standardized $t$-distribution with 4 degrees of freedom. Moreover, let $S_0 = \{ j: \beta^*_j \neq 0 \}$ be the true model.

Applying ten-fold cross-validated LASSO selects $35$ variables. Along the solution path, we compute the number of correctly selected variables $| \hat{S} \cap S_0|$, the fitted correlation, and the upper $5\%$-quantile of the multiplier bootstrap approximation of the maximum spurious correlation based on 1000 bootstrap samples. The results are provided in Table~\ref{tab3'}, from which we see that the cross-validation procedure under the guidance of MSC selects 15 variables including all of the signal covariates.

\begin{table}[h]
\centering
\caption{\label{size} Number of true positive results, the sample correlation between fitted and observed responses, and the upper $5\%$-quantile of the multiplier bootstrap approximation based on $1000$ bootstrap samples.}
\label{tab3'}
\footnotesize{\begin{tabular}{cccc}
   & $| \hat{S}_\lambda \cap S_0 |$  &     $\hat{\mbox{corr}}_n({\mathbb{Y}}, \hat{\mathbb{Y}}^{{\rm pLASSO}}_{\lambda}) $ &  $q^{\scMB}_{0.05}(\hat{s}_\lambda, p)  $   \\ \midrule
$\lambda = 0.3410 \atop (\hat{s}_\lambda = 1 )$  & 1 & 0.3314 & 0.3040
  \\ \midrule
$\lambda = 0.2703 \atop (\hat{s}_\lambda = 2 )$  & 2  & 0.4802 & 0.3870
   \\ \midrule
$\lambda = 0.2580 \atop (\hat{s}_\lambda = 3 )$  & 3  & 0.5255 & 0.4435
 \\ \midrule
$\lambda = 0.2351 \atop (\hat{s}_\lambda = 4 )$  & 3 & 0.5536 & 0.4907
\\ \midrule
$\lambda = 0.2142 \atop (\hat{s}_\lambda =5 ) $ & 3   & 0.5791 & 0.5297 \\ \midrule
$\lambda = 0.2044 \atop (\hat{s}_\lambda =6 )$  & 3  & 0.5971 & 0.5608  \\ \midrule
$\lambda = 0.1952 \atop (\hat{s}_\lambda = 8 )$ & 3  & 0.6205 & 0.6131 \\ \midrule
$\lambda = 0.1778 \atop (\hat{s}_\lambda = 9 )$ & 3   & 0.6377 & 0.6365 \\ \midrule
$\lambda = 0.1697 \atop (\hat{s}_\lambda =11 )$ & 4   & 0.6953 & 0.6758 \\ \midrule
$\lambda = 0.1620 \atop (\hat{s}_\lambda =14 )$  & 4  & 0.7380 & 0.7208 \\ \midrule
$\lambda = 0.1409 \atop (\hat{s}_\lambda =15 )$  & 4 & \textbf{0.7490} & \textbf{0.7346} \\ \midrule
$\lambda = 0.1345 \atop (\hat{s}_\lambda =19 )$  & 4 & 0.7685 & 0.7799 \\ \midrule
$ \vdots $  & \vdots & \vdots & \vdots \\ \midrule
%$\lambda = 0.1284 \atop (\hat{s}_\lambda =20 )$ & 4 & 0.7711 & 0.7898 \\ \midrule
%$\lambda = 0.1226 \atop (\hat{s}_\lambda =22 )$ & 4 & 0.7846 & 0.8073 \\ \midrule
%$\lambda = 0.1170 \atop (\hat{s}_\lambda =24 )$ & 4 & 0.7900 & 0.8234 \\ \midrule
%$\lambda = 0.1066 \atop (\hat{s}_\lambda =27 )$ & 4 & 0.8135 & 0.8431 \\ \midrule
%$\lambda = 0.1018 \atop (\hat{s}_\lambda =28 )$ & 4 & 0.8191 & 0.8490 \\ \midrule
%$\lambda = 0.0971 \atop (\hat{s}_\lambda = 30 )$ & 4 & 0.8247 & 0.8607 \\ \midrule
%$\lambda = 0.0927 \atop (\hat{s}_\lambda =34 )$  &4 & 0.8361 & 0.8800 \\ \midrule
$\lambda = 0.0885 \atop (\hat{s}_\lambda =35 )$  &4 & 0.8428 & 0.8847 \\
\bottomrule
\end{tabular}}
\end{table} \par

\subsection{Gene expression data}\label{sec6.3}
%  CHRNA6  GI_21361147-S
In this section, we extend the previous study in Section~\ref{sec6.2} to an analysis of a real life data set. To further address the question that for a given data set, whether the discoveries based on certain data-mining technique are any better than spurious correlation, we consider again the gene expression data from 90 individuals (45 Japanese and 45 Chinese, JPT-CHB) from the international `HapMap' project [\cite{TSK05}] discussed in the introduction.

The gene {\em CHRNA6} is thought to be related to the activation of dopamine-releasing neurons with nicotine, and therefore has been the subject of many nicotine addiction studies [\cite{T10}]. We took the expressions of {\em CHRNA6} as the response $Y$ and the remaining $p=47292$ expressions of probes as covariates $\bX$. For a given $\lambda>0$, LASSO selects $\hat{s}_\lambda$ probes indexed by $\hat{S}_\lambda$. In particular, using ten-fold cross-validation to select the tuning parameter gives $\hat{s}_{\lambda_{0}}=25$ probes with $\lambda_{0}=0.0674$. Define fitted vectors $\hat{\mathbb{Y}}_\lambda^{{\rm LASSO}} = \mathbb{X}  \hat{\bbeta}^{{\rm LASSO}}_\lambda$ and $\hat{\mathbb{Y}}_\lambda^{{\rm pLASSO}} = \mathbb{X}_{\hat{S}_\lambda} \hat{\bbeta}^{{\rm pLASSO}}_{\lambda}$, where $\hat{\bbeta}^{{\rm LASSO}}_\lambda$ is the LASSO estimator and $\hat{\bbeta}^{{\rm pLASSO}}_\lambda$ is the post-LASSO estimator, which is the least-square estimator based on the LASSO selected set.

We depict the observed correlations between the fitted value and the response as well as the median and upper $\alpha$-quantile of the multiplier bootstrap approximation with $\alpha=10 \%$ based on $1000$ bootstrap replications in Table~\ref{tab4}.
Even though $\hat{\mbox{corr}}_n({\mathbb{Y}}, \hat{\mathbb{Y}}^{{\rm LASSO}}) = 0.8991$ and $\hat{\mbox{corr}}_n({\mathbb{Y}}, \hat{\mathbb{Y}}^{{\rm pLASSO}}) = 0.9214$, the discoveries appear to be no better than chance.  We therefore increase $\lambda$, which decreases the size of discovered probes.  From
Table 4, only the discovery of three probes is above chance results at $\alpha = 10\%$. The three probes are BBS1 -- Homo sapiens Bardet-Biedl syndrome 1, POLE2 -- Homo sapiens polymerase (DNA directed), epsilon 2 (p59 subunit), and TG737 -- Homo sapiens Probe hTg737 (polycystic kidney disease, autosomal recessive), transcript variant 2. Figure~\ref{Fig4} shows the observed correlations of the fitted values and observed values compared to the reference null distribution.

\begin{table}[h]
\centering
\caption{\label{size} Sample correlations between fitted and observed responses, and the empirical median and upper $\alpha$-quantile of the multiplier bootstrap approximation based on $1200$ bootstrap samples when $\alpha=10 \%$.}
\label{tab4}
\footnotesize{\begin{tabular}{ccccc}
{{\rm T}}rule
   &   $\hat{\mbox{corr}}_n({\mathbb{Y}}, \hat{\mathbb{Y}}^{{\rm LASSO}}_{\lambda}) $ & $\hat{\mbox{corr}}_n({\mathbb{Y}}, \hat{\mathbb{Y}}^{{\rm pLASSO}}_{\lambda}) $ &  $q^{\scMB}_{0.5}(\hat{s}_\lambda, p)  $  & $ q^{\scMB}_{0.1}(\hat{s}_\lambda, p)  $  \\ \midrule
$\lambda = 0.1789 \atop (\hat{s}_\lambda =2 )$ & 0.6813 & 0.6879 & 0.5585 & 0.5988  \\ \midrule
$\lambda = 0.1708 \atop (\hat{s}_\lambda =3 )$ & 0.6915 & 0.7010 & 0.6555 & 0.6904   \\ \midrule
$\lambda = 0.1630 \atop (\hat{s}_\lambda =4 )$ & 0.7059 & 0.7260 & 0.7252 & 0.7554 \\ \midrule
$\lambda = 0.1556 \atop (\hat{s}_\lambda =5 )$ & 0.7141 & 0.7406 & 0.7797 & 0.8044 \\ \midrule
$\lambda = 0.1292 \atop (\hat{s}_\lambda =8 )$ & 0.7454 & 0.7641 & 0.8828 & 0.8988 \\ \midrule
$\lambda = 0.1177 \atop (\hat{s}_\lambda =14 )$ & 0.7714 & 0.8307 & 0.9658 & 0.9724 \\ \midrule
$\lambda = 0.1073 \atop (\hat{s}_\lambda =17 )$ & 0.8026 & 0.8739 & 0.9817 & 0.9860 \\ \midrule
$\lambda = 0.0933 \atop (\hat{s}_\lambda =21 )$ & 0.8451 & 0.9019 & 0.9915 & 0.9945 \\ \midrule
$\lambda = 0.0891 \atop (\hat{s}_\lambda =23 )$ & 0.8561 & 0.9109 & 0.9937 & 0.9966 \\ \midrule
$\lambda = 0.0674 \atop (\hat{s}_\lambda =25 )$ & 0.8991 & 0.9214 & 0.9953 & 0.9979 \\
\bottomrule
\end{tabular}}
\end{table} \par

\begin{figure}[hbtp!]
  \centering
  \includegraphics[width=5in, trim=1.6cm 1.0cm 1.6cm 0.9cm, clip]{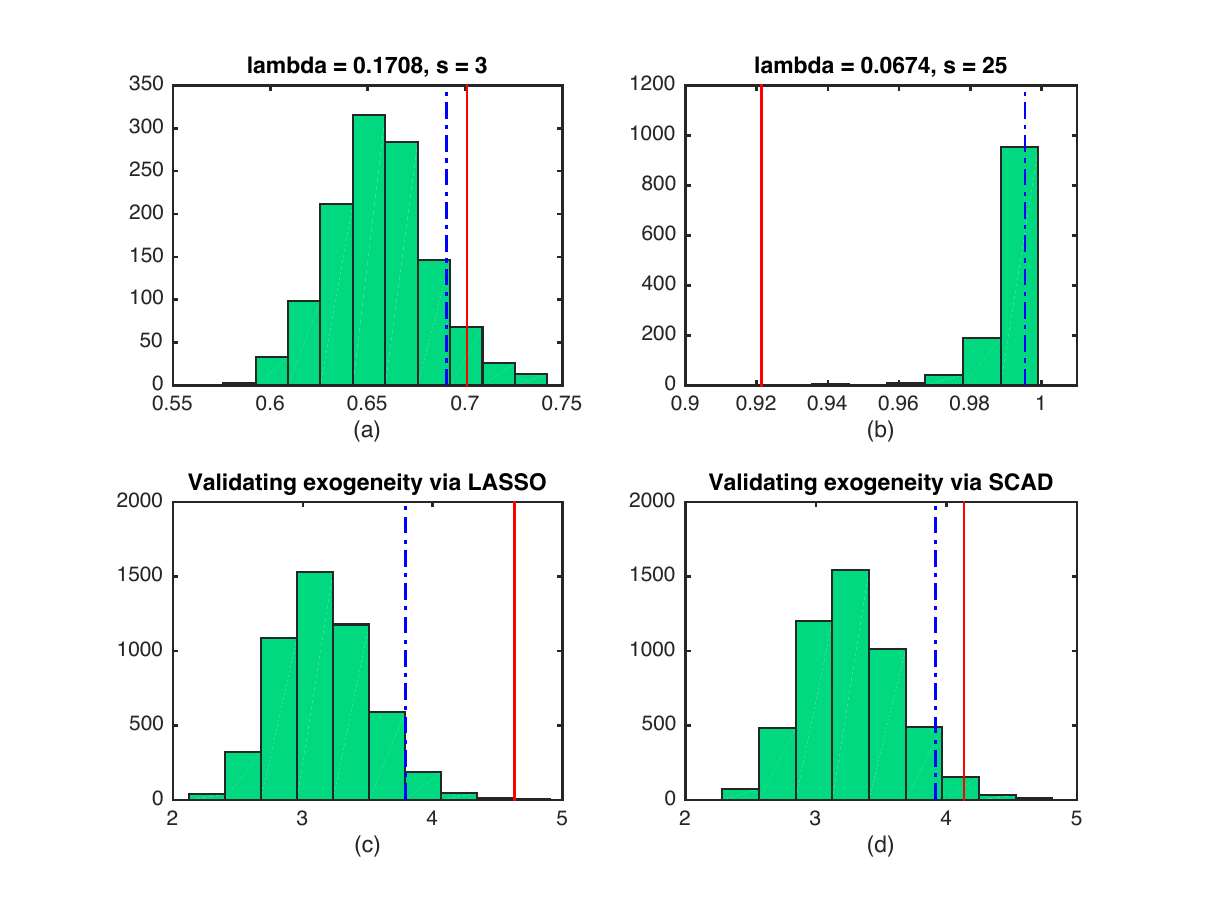}
  \caption{Top panel: Distributions of the spurious correlation $\hat{R}_n(s, p)$ estimated by the bootstrap approximation for (a) $s=3$ and (b) $s=25$ and the sample correlation between fitted and observed responses (see Table~\ref{tab4}).
  Red solid lines are observed correlations and blue dash-dot lines mark the 90th percentile in (a) the median and (b) the distributions of the median.
  Bottom panel:  Null distributions for testing exogeneity \eqref{eq1.2} and its 95th percentile (indicated by dash blue line) using bootstrap approximation \eqref{eq4.10} and observed test statistics $ \hat{T}_{n,p}^{{\rm obs}} $ (indicated by solid red line) based on the residuals of the LASSO and SCAD.
 % (Middle) Histograms of the bootstrap approximations and the sample correlation when $\lambda=\lambda_{{\rm MinMSE}}$; (Bottom) Histograms of the bootstrap approximations for the null distribution in \eqref{eq4.10} and the computed statistic using the SCAD.
\label{Fig4} }
\end{figure}

We now use the test statistic \eqref{eq5.5} to test whether the null hypothesis \eqref{eq1.2} holds.  We take $\lambda_0 =0.0674$ and compute the observed test statistic $ \hat{T}^{{\rm obs}} _{n,p} = 4.6318$.  This corresponds to $\sqrt{n}$ times the maximum correlation presented in Figure~\ref{Fig1}. Using the null distribution provided by (\ref{eq4.10}), which can be estimated via the multiplier bootstrap, yields the $p$-value $0.001$. Further, using the SCAD gives $ \hat{T}_{n,p}^{{\rm obs}} = 4.1324$ and a $p$-value $0.0164$.  Both calculations are based on 5000 bootstrap replications. Therefore, the evidence against the exogeneity assumption is very strong. Figure~\ref{Fig4} depicts the observed test statistics relative to the null distribution.

\section{Proofs}
\label{sec7}

We first collect several technical lemmas in Section~\ref{lemmas} before proving our main result, Theorem~\ref{thm3.1} in Section~\ref{sec3}. The proofs of Theorems~\ref{thm3.2}, \ref{thm4.1} and \ref{thm4.2} are given in the supplemental material, where the proofs of Propositions~\ref{prop3.1} and \ref{prop3.2} and Lemmas~\ref{lem7.2}--\ref{lem7.6} can also be found. Throughout, the letters $C, C_1, C_2, \ldots$ and $c,c_1,c_2,\ldots$ denote  generic positive constants that are independent of $(s,p,n)$, whose values may change from line to line.

\subsection{Technical lemmas}  \label{lemmas}

The following lemma combines Propositions~5.10 and 5.16 in \cite{V12}.

\begin{lemma} \label{lem7.1}
Let $X_1,\ldots, X_n$ be independent centered random variables and write $\bx_n=(X_1, \ldots, X_n)^{{{\rm T}}} \in \bbr^n$. Then for every $\ba=(a_1,\ldots, a_n)^{{{\rm T}}} \in \bbr^n$ and every $t\geq 0$, we have
\be
	\PP\big(  |\ba^{{{\rm T}}} \bx_n|  \geq t \big) \leq 2\exp\bigg\{ -c_{{\rm B}} \min\bigg(  \frac{t^2}{ B_1^2 | \ba |_2^2}, \frac{t}{ B_1 |\ba|_\infty} \bigg) \bigg\} \label{eq7.1}
\ee
and
\be
	\PP\big(  |\ba^{{{\rm T}}} \bx_n|  \geq t \big) \leq 2\exp\bigg( -c_{{\rm H}} \frac{t^2}{B_2^2 | \ba |_2^2} \bigg), \label{eq7.2}
\ee
where $B_v = \max_{1\leq i\leq n}\| X_i \|_{\psi_v}$ for $v=1,2$ and $c_{{\rm B}},c_{{\rm H}}>0$ are absolute constants.
\end{lemma}

\begin{lemma} \label{lem7.2}
Let Conditions~\ref{moment.cond} and \ref{sampling.condition} be fulfilled. Write
$$
	D_n = D_n(s,p)   := \sup_{ \balpha \in \mathcal{V}} \big|   \balpha^{{{\rm T}}} \hat \bSigma_n \balpha / \balpha^{{{\rm T}}} \bSigma \balpha   -1 \big| \ \ \mbox{ and } \ \  \hat{\sigma}_{\varepsilon}^2= n^{-1} \sn ( \varepsilon_i - \bar{\varepsilon}_n )^2,
$$
where $\hat \bSigma_n =   n^{-1}\sn (\bX_i -\bar{\bX}_n)(\bX_i -\bar{\bX}_n)^{{{\rm T}}} $ and $\mathcal{V}$ is as in \eqref{eq3.2}. Then, there exists a constant $C_1 >0$ such that, for every $t \geq 1$,
\begin{align}
	 D_n  \leq C_1   K_1^2\bigg[   \s{ \frac{s}{n} \log (\gamma_s ep /s )  } +  \max\bigg\{ \s{\f{t}{n}}  ,  c_n(s,p) \f{t  }{n} \bigg\} \bigg] \label{eq7.3}
\end{align}
holds with probability at least $1-8 e^{-t}$, where $c_n(s,p) := s \log (\gamma_s ep /s )   \vee \log n$. Moreover, for every $t >0$,
\be
	  \big|   \hat{\sigma}_{\varepsilon}^2 -1 \big| \leq   K_0^2 \, n^{-1} t  + 4 K_0^2 \max\big( n^{-1/2}\sqrt{t}  , n^{-1} t \big)	  \label{eq7.4}
\ee
holds with probability greater than $1- 2\exp(-c_{{\rm B}} t ) - 2\exp(-c_{{\rm H}} t )$, where $c_{{\rm B}}, c_{{\rm H}} >0$ are absolute constants as in \eqref{eq7.1} and \eqref{eq7.2}.
\end{lemma}

The following results address the concentration and anti-concentration phenomena of the supremum of the Gaussian process $\mathbb{G}^*$ indexed by $\mathcal{F}$ (see \eqref{eq3.4}). In line with \cite{CCK13}, inequalities \eqref{eq7.15} and \eqref{eq7.16} below are referred to as the concentration and anti-concentration inequalities, respectively.

\begin{lemma} \label{lem7.3}
Let $R^*(s,p) = \sup_{f_{\balpha} \in \mathcal{F}}   f_{\balpha}(\bZ)  / |\balpha|_{{ \bSigma}}$
for $\mathcal{F}=\mathcal{F}(s,p)$ given in \eqref{eq3.2} and $\bZ \sta{d}{=} N(\mo,\bSigma)$. Then there exists an absolute constant $C>0$ such that, for every $p\geq 2$, $1\leq s\leq p$ and $t  >0$,
\begin{align}
\PP\big\{ R^*(s,p) \geq  C \s{ s \log(  \gamma_s e p / s)  } + t  \big\}   \leq e^{-t^2/2} \label{eq7.15} \\
 \mbox{ and }  \ \  \sup_{ x \geq 0 } \PP\big\{ | R^*(s,p) -x | \leq t \big\}   \leq C   t  \s{   s \log(  \gamma_s e p / s) }  , \label{eq7.16}
\end{align}
where $\gamma_s=\s{\phi_{\max}(s)}/\s{\phi_{\min}(s)}$.
\end{lemma}

\begin{lemma}  \label{lem7.4}
Suppose that $a\geq 1$ and $b_j, c_j >0$ for $j=1,\ldots, m$ are positive constants. Let $X_1, \ldots, X_m$ be real-valued random variables that satisfy
$$
	\PP(|X_j| \geq t ) \leq a  \exp\{-t^2/(2 b_j)\},    \quad \mbox{ for } t>0, \ \ j=1, \ldots, m.
$$
Then, for all $m \geq 4/a$, we have $\e ( \max_{1\leq j\leq m} |X_j|  ) \leq 2 \s{ \log(am ) \max_{1\leq j\leq m} b_j}$. Furthermore, suppose that $\PP(|X_j| \geq t ) \leq a   \exp(-t/c_j )$ holds for all $t>0$ and $j=1, \ldots, m$. Then, for any $m \geq 4/a$, we have $\e ( \max_{1\leq j\leq m} |X_j|  ) \leq \{ \log(am) + 1 \} \max_{1\leq j\leq m} c_j$.
\end{lemma}

To save space, we leave the proofs of Lemmas~\ref{lem7.2}--\ref{lem7.4} to Appendix~\ref{Appendix.A} in the supplemental material.

\subsection{Proof of Theorem~\ref{thm3.1}}
\label{sec7.2}

In view of \eqref{eq3.1}, we have
\begin{align*}
	\wh  R_{n}(s,p)   = \sup_{ \balpha  \in \mathcal{V}} \f{n^{-1}\sn   \balpha^{{\rm T}} ( \varepsilon_i \bX_i  ) - \bar{\varepsilon}_n   \balpha^{{{\rm T}}} \bar{\bX}_n }{ (\balpha^{{{\rm T}}} \hat{\bSigma}_n \balpha)^{1/2}	\cdot \{ n^{-1}\sn (\varepsilon_i-\bar{\varepsilon}_n )^2 \}^{1/2} } ,
\end{align*}
where $\mathcal{V}$ is as in \eqref{eq3.2}.

By Lemma~\ref{lem7.2}, instead of dealing with $\wh R_n(s,p)$ directly, we first investigate the asymptotic behavior of its standardized counterpart given by
\be
 R_{n}(s,p)  = \sup_{\balpha \in \mathcal{V}}
	n^{-1}\sn \f{   \balpha^{{\rm T}}  ( \varepsilon_i \bX_{i } ) }{|\balpha |_{{ \bSigma}}}  =   \sup_{\balpha \in \mathcal{V}}   n^{-1} \sn   \balpha_{{ \bSigma}}^{{\rm T}} \by_i  ,  \label{eq7.17}
\ee
where $\by_i = \varepsilon_i \bX_i=(Y_{i1},\ldots, Y_{ip})^{{{\rm T}}}$ are i.i.d. random vectors with mean zero and covariance matrix $\bSigma$. Let $\PP_\by$ be the probability measure on $\bbr^p$ induced by $\by = \varepsilon \bX$. Further, define rescaled versions of $\wh R_n(s,p)$ and $R_n(s,p)$ as
\be
\hat L_{n }= \hat L_n(s,p) = \s{n}  \wh R_n(s,p), \ \ 	L_{n} = L_n(s,p)  = \sup_{\balpha \in \mathcal{V}}  n^{-1/2} \sn  \balpha_{{ \bSigma}}^{{{\rm T}}}  \by_{i} .  \label{eq7.18}
\ee

The main strategy is to prove the Gaussian approximation of $L_{n} $ by the supremum of a Gaussian process $\mathbb{G}^*$ indexed by $\mathcal{F}$ with covariance function
$$
	\e \big( \mathbb{G}^*f_{\balpha_1} \mathbb{G}^*f_{\balpha_2} \big)  = \frac{ \balpha_1^{{{\rm T}}} \bSigma \balpha_2}{  |\balpha_1|_{{ \bSigma}} \cdot |\balpha_2|_{{ \bSigma}} } , \ \ \balpha_1, \balpha_2 \in \mathcal{V}.
$$
Let $\bZ$ be a $p$-variate centered Gaussian random vector with covariance matrix $\bSigma$. Then the aforementioned Gaussian process $\mathbb{G}^*$ can be induced by $\bZ$ in the sense that for every $\balpha \in \mathcal{V}$, $\mathbb{G}^*f_{\balpha} =  \balpha_{{ \bSigma}}^{{{\rm T}}} \bZ $. The following lemmas show that, under certain moment conditions, the distribution of $L_{n}=\s{n}  R_{n}(s,p)$ can be consistently estimated by that of the supremum of the Gaussian process $\mathbb{G}^*$, denoted by $R^*(s,p)= \sup_{\balpha  \in \mathcal{V}} \mathbb{G}^*f_{\balpha}$, and $\hat{L}_n$ and $L_n$ are close.  We state them first in the following two lemmas and prove them in Appendix~\ref{Appendix.A} of the supplemental material.

\begin{lemma}  \label{lem7.5}
Under Conditions~\ref{moment.cond} and \ref{sampling.condition}, there exists a random variable $T^* = T^*(s,p) \stackrel{d}{=} \sup_{\balpha \in \mathcal{V}}    \balpha_{{ \bSigma}}^{{\rm T}} \bZ $ for $\bZ \sta{d}{=} N(\mo, \bSigma)$ such that, for any $\de \in (0 ,  K_0 K_1]$,
\begin{align}
	   |  L_n  -  T^*  |  \lesssim   n^{-1} c^{1/2}_{n}(s,p)  +   K_0 K_1 \, n^{-3/2}  c^2_{n}(s,p)  + \delta 	\label{eq7.19}
\end{align}
holds with probability at least $1-C  \De_{n }(s, p\, ; \de) $, where $c_{n}(s,p) = s\log( \gamma_s e p /s )  \vee \log n$ and
\beq
	\De_{n}(s, p\, ;\de)  =  (K_0 K_1)^3  \frac{ \{ s  b_n(s,p) \}^2}{\de^3 \s{n}} + (K_0 K_1)^4 \f{\{ s  b_n(s,p) \}^5}{\de^4  n}
\eeq
with $b_n(s,p) = \log( \gamma_s p/s ) \vee \log n$.
\end{lemma}

\begin{lemma}  \label{lem7.6}
Let Conditions~\ref{moment.cond} and \ref{sampling.condition} hold. Assume that the sample size satisfies $n\geq C_1 (K_0 \vee K_1)^4 c_n(s,p)$. Then, with probability at least $1-C_2 \, n^{-1/2}  c^{1/2}_n(s,p)$,
\begin{align}
	 | \hat L_n -  L_n  |  \lesssim    (K_0\vee K_1)^2 K_0K_1  \, n^{-1/2}   c_n(s,p) , \label{eq7.20}
\end{align}
where $c_n(s,p)  = s\log(  \gamma_s e p /s ) \vee \log n$.
\end{lemma}

Let $b_n(s,p) = \log (\gamma_s p /s ) \vee \log n$. Applying Lemmas~\ref{lem7.5} and \ref{lem7.6} with
$$
\de= \delta_n(s,p) =  ( K_0K_1)^{3/4} \min\big[ 1, n^{-1/8} \{s b_n(s,p)\}^{3/8}  \big]
$$
yields that, with probability at least $1- C  ( K_0K_1)^{3/4}  \, n^{-1/8} \{s b_n(s,p)\}^{7/8}$,
$$
	| \hat L_n -  T^* |  \lesssim    ( K_0K_1)^{3/4}  \, n^{-1/8}  \{s  b_n(s,p)\}^{3/8}.
$$
Together with the inequality \eqref{eq7.16}, this proves \eqref{eq3.3}.

Further, using \eqref{eq3.2}, \eqref{eq3.4} and the identity $\mathbf{v}^{{{\rm T}}} \mathbf{A}^{-1} \mathbf{v} = \max_{\balpha \in \bS^{s-1}}  \frac{  (\balpha^{{{\rm T}}} \mathbf{v})^2 }{ \balpha^{{{\rm T}}} \mathbf{A} \balpha }$ that holds for any $s\times s$ positive definite matrix $\mathbf{A}$, we find that with probability one,
\begin{align}
	R^*(s,p) = \max_{S\subseteq [p]: |S|=s} \max_{\balpha \in \bS^{s-1} } \frac{\balpha^{{{\rm T}}} \bZ_S}{\sqrt{\balpha^{{{\rm T}}} \bSigma_{SS} \balpha}} = \max_{S\subseteq [p]: |S|=s} \sqrt{\bZ_S^{{{\rm T}}} \bSigma_{SS}^{-1} \bZ_S } , \label{eq7.21}
\end{align}
where for each $S\subseteq [p]$ fixed, the second maximum over $\balpha$ is achieved when $\balpha=\bSigma_{SS}^{-1/2} \bZ_S / |  \bSigma_{SS}^{-1/2} \bZ_S  |_2$, as for each $p\geq 1$ fixed, all of the coordinates of $\bZ$ are non-zero almost surely. In particular, when $\bSigma=\bI_p$, the right-hand side of \eqref{eq7.21} is reduced to $ \max_{S\subseteq [p]: |S|=s} |\bZ_S|_2$ and therefore, $\{ R^*(s,p) \}^2 =  \max_{S\subseteq [p]: |S|=s} \sum_{j \in S} Z_j^2 = Z^2_{(p)} + \cdots + Z^2_{(p-s+1)}$ happens with probability one. This and \eqref{eq3.3} complete the proof of \eqref{eq3.5}.  \qed

\newpage

\def \sMB {\mbox{\scriptsize MB}}
\def \scMB {\mbox{\scriptsize CMB}}
\def \T {{\rm T}}

\begin{frontmatter}
\title{Supplement to ``Are Discoveries Spurious?  Distributions of Maximum Spurious Correlations and Their Applications''}
\runtitle{Supplementary material}

\end{frontmatter}

\appendix

\section{Proof of Lemmas~7.2--7.6}
\label{Appendix.A}

Here we prove Lemmas~7.2--7.6.

\subsection{Proof of Lemma~7.2}

For every $\balpha \in \mathcal{V}$, recall that $\balpha_{\bSigma}=\balpha/ |\balpha|_{\bSigma}$ with $|\balpha|_{\bSigma} =( \balpha^{{\rm T}} \bSigma \balpha )^{1/2}$. Then, we have
$$
  \frac{ \balpha^{{\rm T}} \hat \bSigma \balpha }{ \balpha^{{\rm T}} \bSigma \balpha  }   =   n^{-1} \sn  ( \balpha_{\bSigma}^{{\rm T}} \bX_i )^2      -( \balpha_{\bSigma}^{{\rm T}} \bar{\bX}_n )^2.
$$
In view of this identity, we define
$$
	D_{n,1} = \sup_{ \balpha \in \mathcal{V} } \bigg|  n^{-1} \sn ( \balpha_{\bSigma}^{{\rm T}} \bX_i )^2  -1  \bigg| , \quad D_{n,2} = \sup_{ \balpha \in \mathcal{V}}  ( \balpha_{\bSigma}^{{\rm T}} \bar{\bX}_n )^2,
$$
such that $D_n \leq D_{n,1} + D_{n,2}$. In what follows, we bound the two terms $D_{n,1}$ and $D_{n,2}$ respectively.

Let $\mathcal{G}$ be a class of functions $\bbr^p \mapsto \bbr$ given by $\mathcal{G} =  \{  \mx \mapsto g_{\balpha}(\mx) = \lea  \balpha_{\bSigma}, \mx  \ria^2 : \balpha \in \mathcal{V}  \}$, and denote by $\PP_\bX$ the probability measure on $\bbr^p$ induced by $\bX$. In this notation, we have $D_{n,1}  =  \sup_{g\in \mathcal{G}}  |  n^{-1} \sn  g(\bX_i) - \PP_\bX g    |$. To bound $D_{n,1}$, we follow a standard procedure: first we show concentration of $D_{n,1}$ around its expectation $\e D_{n,1}$, and then upper bound the expectation. To prove concentration, applying Theorem 4 in \cite{A08} implies that there exists an absolute constant $C>0$ such that, for every $t>0$,
\be
	D_{n,1} \leq 2   \e D_{n,1} +  \max\bigg\{ 2\sigma_{\mathcal{G}}  \frac{\s{t}}{n} ,  C   \frac{t}{n}  \bigg\| \max_{1\leq i\leq n} \sup_{g\in \mathcal{G}} |g(\bX_i)|  \bigg\|_{\psi_1}  \bigg\}  \label{eq7.5}
\ee
holds with probability at least $1-4 e^{-t}$, where $\sigma^2_{\mathcal{G}} := \sup_{g\in \mathcal{G}} \sn \PP_{\bX_i} g^2$. Under Condition~2.1, it follows from the fact $|\bSigma^{1/2}\balpha_{\bSigma} |_2=1$ and the definition of $\| \cdot \|_{\psi_2}$~that
\begin{align}
	 \sigma^2_{\mathcal{G}}\leq n   \sup_{\balpha \in \mathcal{V}} \e  \{  ( \bSigma^{1/2}  \balpha_{\bSigma})^{{\rm T}} \bU  \}^4 \leq  16    n     \sup_{\balpha \in \bS^{p-1}} \|   \balpha^{{\rm T}} \bU   \|_{\psi_2}^4 \leq 16   K_1^4 \,   n ,  \label{eq7.6}
\end{align}
which further leads to $\sigma_{\mathcal{G}} \leq 4 K_1^2 \s{n}$ for $K_1$ as in Condition~2.1. In the last term of \eqref{eq7.5}, note that
$$
\sup_{g\in \mathcal{G}} |g(\bX_i)| = \sup_{ \balpha \in \mathcal{V}}  (   \balpha_{\bSigma}^{{\rm T}} \bX_i )^2 =  \bigg( \sup_{\balpha \in \mathcal{V}}  \balpha_{\bSigma}^{{\rm T}} \bX_i   \bigg)^2.
$$
For every $\epsilon \in (0 ,  \gamma_s^{-1/2})$, a standard argument can be used to prove that there exists an $\epsilon$-net $\mathcal{N}_\epsilon$ of $\mathcal{V}$ such that $d_\epsilon = |\mathcal{N}_\epsilon| \leq   \{ (2+\epsilon) ep/( \epsilon s )  \}^s$ and
\be
	\sup_{\balpha \in \mathcal{V}}     \balpha_{\bSigma}^{{\rm T}} \bX_i   \leq  (1-   \gamma_s \epsilon)^{-1} \max_{\balpha \in \mathcal{N}_\epsilon}   \balpha_{\bSigma}^{{\rm T}} \bX_i    .  \label{eq7.7}
\ee
See, for example, the proof of \eqref{eq7.25} below. In particular, under Condition~2.1, using Lemma~2.2.2 in \cite{VW96} implies by taking $\epsilon_s =( 4 \gamma_s )^{-1}$ and $\mathcal{N}  = \mathcal{N}_{\epsilon_s}$ that
\begin{align}
	    \bigg\| \max_{1\leq i\leq n} \sup_{g\in \mathcal{G}} |g(\bX_i)| \bigg\|_{\psi_1}  \nn
	 &=\bigg\| \max_{1\leq i\leq n} \sup_{\balpha \in \mathcal{V}}  \{ ( \bSigma^{1/2}  \balpha_{\bSigma} )^{{\rm T}} \bU_i \}^2  \bigg\|_{\psi_1} \nn \\
	& \lesssim \bigg\| \max_{1\leq i\leq n}\max_{\balpha \in \mathcal{N}  }  \{ ( \bSigma^{1/2}  \balpha_{\bSigma})^{{\rm T}} \bU_i \}^2 \bigg\|_{\psi_1}  \nn \\
	& \lesssim  \{   s \log (\gamma_s ep /s )  \vee \log n \}   \sup_{\balpha \in \bS^{p-1}}  \|    \balpha^{{\rm T}} \bU_i  \|^2_{\psi_2} \nn \\
	& \lesssim  K_1^2 \, c_n(s,p) ,  \label{eq7.8}
\end{align}
where $c_n(s,p)  = s\log ( \gamma_s e p /s )  \vee \log n$. Consequently, combining \eqref{eq7.5}, \eqref{eq7.6} and \eqref{eq7.8} yields, with probability at least $1-4 e^{-t}$,
\be
	D_{n,1} \leq 2 \e D_{n,1} + C   K_1^2 \max\bigg\{   \sqrt{\frac{t}{n}} , c_n(s,p)  \frac{t}{n} \bigg\}. \label{eq7.9}
\ee

To bound the expectation $\e D_{n,1}$, we use a result that involves the generic chaining complexity, $\gamma_m(T, d)$, of a semi-metric space $(T, d)$. We refer to \cite{T05} for a systematic introduction. A tight upper bound for $\e D_{n,1}$ can be obtained by a direct application of Theorem~A in \cite{M10}. To this end, note that
$ \sup_{\balpha \in \mathcal{V}}  \|    \balpha_{\bSigma}^{{\rm T}} \bX_i  \|_{\psi_1} =  \sup_{\balpha \in \mathcal{V}}  \| ( \bSigma^{1/2}  \balpha_{\bSigma})^{{{\rm T}}}  \bU_i   \|_{\psi_1}    \leq   K_1 $
and for every $\balpha, \balpha' \in \mathcal{V}$,
$$
	 \| ( \balpha_{\bSigma}-  \balpha_{\bSigma}')^{{\rm T}} \bX_i   \|_{\psi_2} = \Big\|  \{ \bSigma^{1/2}( \balpha_{\bSigma} -  \balpha_{\bSigma}' ) \}^{{\rm T}} \bU_i    \Big\|_{\psi_2} \leq  K_1  |  \balpha_{\bSigma} -  \balpha'_{\bSigma} |_{\bSigma}.
$$
Successively, it follows from Theorem~A in \cite{M10} and Theorems~1.3.6, 2.1.1 in \cite{T05} that
\begin{align}
 \e D_{n,1}   \lesssim  K_1^2   \bigg\{ \f{\gamma_2(\mathcal{F},|\cdot |_{\bSigma})}{\s{n}}  + \f{\gamma^2_2(\mathcal{F},|\cdot |_{\bSigma})}{n} \bigg\}   \lesssim K_1^2  \bigg\{ \f{\mathcal{M}(s,p)}{\s{n}}  + \f{\mathcal{M}^2(s,p)}{n} \bigg\},
 \label{eq7.10}
\end{align}
where $\mathcal{M}(s,p) := \e  ( \sup_{\balpha \in \mathcal{V} }   \balpha_{\bSigma}^{{\rm T}}\bZ   )$ with $\bZ \stackrel{d}{=}  N(\mo,\bSigma)$. In addition, a similar argument to that leading to \eqref{eq7.7} can be used to show that
\begin{align}
	 \mathcal{M}(s,p) \leq \frac{4}{3}   \e \bigg(\max_{\balpha \in \mathcal{N} }  \balpha_{\bSigma}^{{\rm T}} \bZ   \bigg)  \leq 2 \s{ \log( |\mathcal{N}  | )}  \lesssim \sqrt{s\log ( \gamma_s e p /s ) }.  \label{eq7.11}
\end{align}

Next we study $D_{n,2}$. Observe that $\sqrt{ D_{n,2}} =\sup_{\balpha \in \mathcal{V}} |  n^{-1} \sn   \balpha_{\bSigma}^{{\rm T}} \bX_i   |$. Again, we use a concentration inequality due to \cite{A08}. Theorem~4 there implies that, for every $t\geq 0$,
\be
	\sqrt{ D_{n,2}} \leq 2 \e \sqrt{ D_{n,2}} + \max\bigg\{ 2 \sigma_{\mathcal{V}}  \frac{\sqrt{t}}{n} ,  C  \frac{t}{n}  \bigg\|  \max_{1\leq i\leq n}\sup_{\balpha \in \mathcal{V}} |    \balpha_{\bSigma}^{{\rm T}} \bX_i  | \bigg\|_{\psi_1} \bigg\}  \label{eq7.12}
\ee
with probability at least $1-4  e^{-t}$, where
$\sigma_{\mathcal{V}}^2 = \sup_{\balpha \in \mathcal{V}} \sn \e   ( \balpha_{\bSigma}^{{\rm T}} \bX_i  )^2=n$. Under Condition~2.1, $ \sup_{\balpha \in \mathcal{V}} \|  \balpha_{\bSigma}^{{\rm T}} \bX_i   \|_{\psi_1}  \leq  \sup_{\balpha \in \mathcal{V}} \|  \balpha_{\bSigma}^{{\rm T}} \bX_i  \|_{\psi_2}  \leq K_1$. Recall that $\epsilon_s =(4\gamma_s)^{-1}$, a similar argument to that leading to \eqref{eq7.8} gives
\begin{align}
	   \bigg\|  \max_{1\leq i\leq n}\sup_{\balpha \in \mathcal{V}} |   \balpha_{\bSigma}^{{\rm T}} \bX_i   |  \bigg\|_{\psi_1} &   \lesssim \bigg\|  \max_{1\leq i\leq n}\sup_{\balpha \in \mathcal{N}  } \big|  ( \bSigma^{1/2}  \balpha_{\bSigma} )^{{\rm T}}  \bU_i    \big|  \bigg\|_{\psi_2}  \nn \\
	& \lesssim c^{1/2}_n(s,p)  \sup_{\balpha \in \bS^{p-1}}  \|    \balpha^{{\rm T}} \bU_i  \|_{\psi_2}    \lesssim   K_1 \, c^{1/2}_n(s,p).   \label{eq7.13}
\end{align}
For the expectation $\e \sqrt{D_{n,2}}$, it follows from \eqref{eq7.7} with $\epsilon_s=( 4\gamma_s)^{-1}$ that
$$
 \e \sqrt{D_{n,2}}   = \e \bigg(  \sup_{\balpha \in \mathcal{V}}  \balpha_{\bSigma}^{{\rm T}} \bar{\bX}_n    \bigg) \leq \frac{4}{3}   \e \bigg( \max_{\balpha \in \mathcal{N}  }    \balpha_{\bSigma}^{{\rm T}} \bar{\bX}_n   \bigg).
$$
For $\balpha \in \mathcal{V}$ and $t>0$, a direct consequence of (7.2) is that $\PP( |  \balpha_{\bSigma}^{{{\rm T}}}  \bar{\bX}_n |    \geq t  )  \leq 2 \exp(  - c_{{\rm H}} n t^2 / K_1^2 )$. This, together with Lemma~7.4 and the previous display implies
\be
	\e \sqrt{D_{n,2}} \lesssim   K_1   \sqrt{ (s/n) \log( \gamma_s e p /s ) } . \label{eq7.14}
\ee

Together, \eqref{eq7.9}--\eqref{eq7.14} completes the proof of (7.3).

Finally, to prove (7.4), note that $|\hat{\sigma}_{\varepsilon}^2-1| \leq |n^{-1}\sn \varepsilon_i^2-1| + \bar{\varepsilon}_n^2$. For $t_1, t_2 \geq 0$, applying (7.2) and (7.1) gives $\PP(  | \bar{\varepsilon}_n    | \geq t_1 ) \leq 2\exp( - c_{{\rm H}}  n t_1^2/ K_0^2 )$ and $\PP( |  n^{-1}\sn \varepsilon_i^2   -  1  | \geq t_2  ) \leq 2  \exp\{ - c_{{\rm B}}  n \min( t_2^2/A^2 , t_2/A ) \}$, respectively, where $A=\|\varepsilon^2-1 \|_{\psi_1 }\leq 2 \| \varepsilon^2 \|_{\psi_1 } \leq 4\| \varepsilon  \|_{\psi_2}^2=4 K_0^2$. Consequently, taking $t_1=K_0\, n^{-1/2} \sqrt{t}$ and $t_2 = 4K_0^2 \max( n^{-1/2}\sqrt{t}, n^{-1} t  )$ proves (7.4). \qed

\subsection{Proof of Lemma~7.3}

By (2.5), we have $R^*(s,p)=\sup_{\balpha\in \mathcal{V}}    \balpha_{\bSigma}^{{\rm T}} \bZ  $ and for every $\balpha \in \mathcal{V}$, $\e (  \balpha_{\bSigma}^{{\rm T}} \bZ  ) =0$ and $\e ( \balpha_{\bSigma}^{{\rm T}} \bZ )^2 =1$. Consequently, in view of \eqref{eq7.11}, inequalities (7.5) and (7.6) follow from Borell's inequality [Proposition~A.2.1 in \cite{VW96}] and Lemma~A.1 of \cite{CCK14a}, respectively. \qed

\subsection{Proof of Lemma~7.4}

Put $B=\max_{1\leq j\leq m} b_j$. For any $T>0$, we have
\begin{align*}
	\e \bigg( \max_{1\leq j\leq m} |X_j| \bigg) &= \int_0^\infty \P\bigg( \max_{1\leq j\leq p} |X_j| > t \bigg) \, dt \\
	& \leq T   + \int_{T }^\infty \P\bigg( \max_{1\leq j\leq m} |X_j| > t \bigg) \, dt  \\
	&  \leq T  + a \sum_{j=1}^m \s{b_j}  \cdot \int_{T/\s{b_j} }^\infty \exp( -t^2 /2 ) \, dt    \\
	& \leq T   +  a  \s{B}  m  \min\big( \s{\pi/2}, \s{B}/T \big)  \exp\{ -T^2/(2 B)\}.
\end{align*}
In particular, this implies by taking $T=\s{2B \log(a m)} \geq \s{2 B}$ that
\begin{align*}
\e \bigg( \max_{1\leq j\leq m} |X_j| \bigg)  & \leq \s{B}   \big[  \s{2 \log(a m)} + \{2\log(a m)\}^{-1/2} \big]  \\
& \leq \bigg( \s{2} + \f{1}{\s{2}\log 4} \bigg)   \s{B \log(a m)}.
\end{align*}

A completely analogous argument will lead to the desired bound under the condition that $\P(|X_j|\geq t)\leq a  \exp(-t/c_j)$ for all $t>0$ and $j=1,\ldots, m$. \qed

\subsection{Proof of Lemma~7.5}
\label{AppE.4}

Recall that $\bZ \sta{d}{=} N(\mo, \bSigma)$ and write $\bW_n =n^{-1/2}\sn \by_i$ with $\by_i=\varepsilon_i \bX_i$, such that $L_{n} = \sup_{\balpha \in \mathcal{V}} \lea  \balpha_{\bSigma}, \bW_n \ria$ for $L_n$ as in (7.8).

To prove (7.9), a new coupling inequality for maxima of sums of random vectors in \cite{CCK14a} plays an important role in our analysis. We divide the proof into three steps. First we discretize the index space $\mathcal{V}=\mathcal{V}(s,p)$ using a net, $\mathcal{V}_\varepsilon$, via a standard covering argument. Then we apply the aforementioned coupling inequality to the discretized process, and finish the proof based on the concentration and anti-concentration inequalities for Gaussian processes.

\medskip
\noi
{\it \textbf{Step~1: Discretization.}}
The goal is to establish \eqref{eq7.25}, which  approximates the supremum over an infinite index space $\mathcal{V}$ by the maximum over its $\epsilon$-net $\mathcal{V}_\epsilon$.

Let $\bbr^p$ be equipped with the Euclidean metric $\rho( \bx, \by) = |\bx- \by|_2$ for $\bx, \by \in \bbr^p$. Subsequently, the induced metric on the space of all linear functions $\mx \mapsto f_{\balpha}(\mx)=\lea \balpha, \mx  \ria$ is defined as $\rho( f_{\balpha} , f_{\bbeta}  ) = \sup_{ \bx \in \bS^{p-1}} | f_{\balpha}( \bx ) - f_{\bbeta} (\bx)| = \sup_{ \bx \in \bS^{p-1}} |\lea \balpha - \bbeta, \bx \ria |  = |\balpha -\bbeta |_2$. For every $\epsilon \in (0,1)$, denote by $N(\mathcal{V}, \rho, \epsilon)$ the $\epsilon$-covering number of $(\mathcal{V},\rho)$. For the unit Euclidean sphere $\bS^{p-1}$ equipped with the Euclidean metric $\rho$, it is well-known that $N(\bS^{p-1}, \rho, \epsilon) \leq (1+ 2/ \epsilon )^p$. Together with the decomposition
\be
	\big\{ \balpha \in \bS^{p-1}: |\balpha|_0=s \big\} = \bigcup_{S \subseteq [p] : |S|=s} \big\{ \balpha \in \bS^{p-1}: \mbox{ supp}(\balpha) = S \big\}  \label{eq7.22}
\ee
and the binomial coefficient bound ${p \choose s} \leq  (ep/s)^s$, this yields
\be
	 N(\mathcal{V}, \rho, \epsilon ) \leq {p \choose s} ( 1+ 2/ \epsilon )^s \leq \{  (2+\epsilon)e p / ( \epsilon s) \}^s .  \label{eq7.23}
\ee

For $\epsilon \in (0,1) $ and $S\subseteq [p]$ fixed, let $\mathcal{N}_{S,\epsilon}$ be an $\epsilon$-net of the unit ball in $(\bbr^S, \rho)$ with $|\mathcal{N}_{S,\epsilon}|\leq  ( 1+ 2/\epsilon )^s$. Thus the function class $\mathcal{N}_{\epsilon} : = \cup_{S\subseteq [p]} \mathcal{N}_{S,\epsilon} =\{ \mx \mapsto \lea \balpha, \mx \ria:  \balpha \in \mathcal{N}_{S,\epsilon} , S\subseteq [p] \}$ forms an $\epsilon$-net of $(\mathcal{V}, \rho)$. Denote by $d = d_{\epsilon} = |\mathcal{N}_\epsilon|$ the cardinality of $\mathcal{N}_\epsilon$. Then it is easy to see that $d \leq {p\choose s}  (1+ 2/ \epsilon )^s$ and hence $\log d \lesssim s \log \{  e p/ (\epsilon s)\}$.

For every $\balpha \in \mathcal{V}$ with supp$(\balpha)=S$, there exists some $\balpha' \in \mathcal{N}_{S,\epsilon}$ satisfying that supp$(\balpha')= {\rm supp}(\balpha)$ and $|\balpha -  \balpha'|_2\leq \epsilon$. Further, note that
\begin{align}
	 |  \balpha_{\bSigma} -  \balpha_{\bSigma}' |^2_{\bSigma}
 %=  \bigg(\f{\bSigma^{1/2}\balpha }{|\balpha |_{\bSigma}}-\f{\bSigma^{1/2} \balpha'}{|\balpha'|_{\bSigma}}\bigg)^{{{\rm T}}} \, \bigg( \f{\bSigma^{1/2}\balpha }{|\balpha|_{\bSigma}}-\f{\bSigma^{1/2} \balpha'}{|\balpha'|_{\bSigma}} \bigg)  \nn  \\
	& = 2- 2 \f{\lea \balpha  , \bSigma \balpha' \ria }{|\balpha |_{\bSigma} |\balpha'|_{\bSigma}} \nn \\
 & =\f{\lea \balpha  - \balpha' , \bSigma (\balpha  -\balpha') \ria - (|\balpha |_{\bSigma}-|\balpha'|_{\bSigma})^2}{|\balpha |_{\bSigma} |\balpha'|_{\bSigma}} \leq \gamma_s^2 |\balpha - \balpha'|_2^2,  \label{eq7.24}
\end{align}
from which we obtain
\begin{align*}
	  \lea  \balpha_{\bSigma}, \bW_n \ria & = \lea  \balpha_{\bSigma} - \balpha_{\bSigma}', \bW_n \ria + \lea  \balpha_{\bSigma}' , \bW_n \ria   \\
%	&  = | \balpha_{\bSigma} - \balpha_{\bSigma}'|_2  \lea ( \balpha_{\bSigma} - \balpha_{\bSigma}')/| \balpha_{\bSigma} - \balpha_{\bSigma}'|_2 , \bW_n \ria   + \lea  \balpha_{\bSigma}' , \bW_n \ria   \\
	& = | \balpha_{\bSigma} - \balpha_{\bSigma}'|_{\bSigma}  \f{ \lea ( \balpha_{\bSigma} - \balpha_{\bSigma}')/| \balpha_{\bSigma} - \balpha_{\bSigma}'|_2 , \bW_n \ria}{|( \balpha_{\bSigma} - \balpha_{\bSigma}')/| \balpha_{\bSigma} - \balpha_{\bSigma}'|_2 |_{\bSigma}}  + \lea  \balpha_{\bSigma}' , \bW_n \ria \nn \\
 & \leq \gamma_s \epsilon   \sup_{\balpha \in \mathcal{V}} \lea  \balpha_{\bSigma} , \bW_n \ria+\lea  \balpha_{\bSigma}', \bW_n\ria
\end{align*}
for $\gamma_s$ as in (2.5), and hence
$$
	\sup_{\balpha\in \bS^{p-1}:\, {\rm supp}(\balpha)=S} \lea  \balpha_{\bSigma}, \bW_n \ria   \leq \gamma_s \epsilon   \sup_{\balpha \in \mathcal{V}}\lea  \balpha_{\bSigma} , \bW_n \ria + \max_{\balpha \in \mathcal{N}_{S,\epsilon}} \lea  \balpha_{\bSigma}, \bW_n \ria.
$$
Taking maximum over $S\subseteq [p]$ with $|S|=s$ on both sides yields
 $$
	\sup_{\balpha \in \mathcal{V}}\lea  \balpha_{\bSigma}, \bW_n \ria  \leq  \gamma_s \epsilon  \sup_{\balpha \in \mathcal{V}}\lea  \balpha_{\bSigma}, \bW_n \ria  + \max_{\balpha \in \mathcal{N}_\epsilon} \lea  \balpha_{\bSigma}, \bW_n \ria.
$$
Therefore, as long as $\epsilon \in ( 0 , \gamma_s^{-1} )$,
\be
	\max_{\balpha \in \mathcal{N}_{\epsilon}} \lea  \balpha_{\bSigma}, \bW_n \ria
 \leq  L_n  \leq (1-\gamma_s \epsilon)^{-1}   \max_{\balpha \in \mathcal{N}_{\epsilon}} \lea  \balpha_{\bSigma}, \bW_n \ria. \label{eq7.25}
\ee

\medskip
\noi
{\it \textbf{Step~2: Coupling.}}  This aims to carry the Gaussian approximation over the discrete index set $\mathcal{V}_{\epsilon}$ and to establish \eqref{eq7.26}
or its more explicit bound \eqref{eq7.31}.

Write $\mathcal{V}_{\epsilon}=\{ \balpha_j: j=1,\ldots, d\}$ and let $\bV_1,\ldots,\bV_n$ be i.i.d. $d$-variate random vectors such that $\bV_i=(V_{i1},\ldots, V_{id})^{{{\rm T}}}$, where $V_{ij}= \balpha_{j, \bSigma}^{{\rm T}}  \,\by_i  $ satisfies that $\e ( V_{ij} )=0$ and $\e ( V_{ij}^2 )=1$. Define the $d$-variate Gaussian random vector $ \bG = ( G_1 ,\ldots, G_d )^{{{\rm T}}}$, where $
	  G_j=   \balpha_{j, \bSigma}^{{\rm T}}  \bZ  =   \balpha_j^{{\rm T}} \bZ /|\balpha_j|_{\bSigma} $ for $j=1,\ldots, d$. Note that, for each $1\leq j\neq \ell \leq d$, $\e (V_{1j}V_{1\ell})= \e (G_{j} G_{\ell})$. By Corollary~4.1 of \cite{CCK14a}, there exists a random variable $T^*_\epsilon  \stackrel{d}{=} \max_{1\leq j\leq d} G_j$ such that, for every $\de>0$,
\begin{align}
	& \P\bigg( \bigg| \max_{1\leq j\leq d} n^{-1/2} \sn V_{ij} - T^*_\epsilon \bigg| \geq 16 \de \bigg) \nn \\
	 \lesssim  &  \,  B_1 \frac{  \log(d  n)}{\de^2  n}  +  B_2 \frac{\{ \log(d n)\}^2}{\de^3 n^{3/2}}  + B_3 \frac{  \{ \log(d n)\}^3 }{\de^{4}   n^2}  + \frac{\log n}{n}   ,   \label{eq7.26}
\end{align}
where
\begin{align}
	B_1 & = \e \bigg\{ \max_{1\leq  j,\ell \leq d} \bigg| \sn \big( V_{ij}V_{i\ell}-  \e V_{ij}V_{i\ell} \big) \bigg| \bigg\},  \nn \\% \label{B1} \\
	B_2 & = \e \bigg( \max_{1\leq j\leq d} \sn   | V_{ij}|^3 \bigg), \ \ B_3 = \sn \e \bigg( \max_{1\leq j\leq d} V_{ij}^4 \bigg).  \nn % \label{B2.B3}
\end{align}
In what follows, we bound the three terms $B_1$--$B_3$ respectively.

First, by Lemma~2.2.2 in \cite{VW96} we have
\begin{align*}
	 \e \bigg( \max_{1\leq j\leq d} V_{ij}^4 \bigg) &= \e ( \varepsilon_i^4 ) \cdot \e \bigg\{ \max_{1\leq j\leq d} ( \balpha_{ j, \bSigma}^{{\rm T}} \bX_i  )^4  \bigg\}   \\
	  &  \leq 16 \e (\varepsilon_i^4 )  \cdot  \bigg\| \max_{1\leq j\leq d}  ( \bSigma^{1/2}  \balpha_{ j, \bSigma}  )^{{\rm T}} \bU_i  \bigg\|_{\psi_2}^4  \\
	  &  \lesssim v_4  (\log d)^2 \cdot  \sup_{\balpha \in \bS^{p-1}} \|   \balpha^{{\rm T}} \bU_i  \|_{\psi_2}^4   \nn \\
	  & \lesssim v_4  K_1^4 (\log d)^2
\end{align*}
for $v_4=\e (\varepsilon^4)$ as in Condition~2.1, leading to
\be
   B_3 \lesssim  v_4 K_1^4  \,  n  [  s \log\{ ep /( \epsilon s )\}]^2 .   \label{eq7.27}
\ee

For $B_2$, we apply Lemma~9 in \cite{CCK14b} to obtain
\begin{align*}
	B_2 \lesssim  \max_{1\leq j\leq d} \sn \e |V_{ij}|^3 + (\log d) \cdot \e \bigg( \max_{1\leq i\leq n}\max_{1\leq j\leq d}|V_{ij}|^3 \bigg)  .
\end{align*}
For every integer $q \geq 1$, by the definition of the $\|\cdot \|_{\psi_2}$ norm we have
\begin{align*}
	\e |V_{ij}|^q  = v_q \e \big|	 ( \bSigma^{1/2}  \balpha_{j, \bSigma} )^{{\rm T}} \bU_i   \big|^q \leq  q^{q/2} v_q  \big\| ( \bSigma^{1/2}  \balpha_{j, \bSigma} )^{{\rm T}} \bU_i    \big\|^q_{\psi_2}  \leq  q^{q/2}  v_q K_1^q,
\end{align*}
and once again, it follows from Lemma~2.2.2 in \cite{VW96} that
\begin{align*}
	 \e \bigg( \max_{1\leq i\leq n}\max_{1\leq j\leq d}|V_{ij}|^q \bigg) & \leq q^{q}   \bigg\| \max_{1\leq i\leq n}\max_{1\leq j\leq d} |\varepsilon_i|  \cdot |  \balpha_{j, \bSigma}^{{\rm T}} \bX_i   | \bigg\|_{\psi_1}^q \\
	&  \lesssim q^q \{ \log(d n)\}^q \max_{1\leq j\leq d} \Big\|  \varepsilon \cdot ( \bSigma^{1/2} \balpha_{j, \bSigma} )^{{\rm T}} \bU    \Big\|_{\psi_1}^q   \\
	& \lesssim q^q \{ \log(dn)\}^q    \|  \varepsilon \|_{\psi_2}^q    \max_{1\leq j\leq d} \Big\|  ( \bSigma^{1/2}  \balpha_{j, \bSigma} )^{{\rm T}} \bU   \Big\|_{\psi_2}^q \\
	& \lesssim 	q^q (K_0 K_1)^q \{ \log(d n) \}^q.
\end{align*}
The last three displays together imply by taking $q=3$ that
\be
	B_2 \lesssim v_3 K_1^3 \, n   +  (K_0 K_1)^3   [  s \log\{ ep /( \epsilon s )\}]^4 .   \label{eq7.28}
\ee

Turning to $B_1$, a direct consequence of Lemma~1 in \cite{CCK14b} is that
\begin{align}
	B_1&  \lesssim  (\log d)^{1/2}  \max_{1\leq j\leq d} \bigg( \sn \e V_{ij}^4 \bigg)^{1/2} +  ( \log d )  \cdot \bigg( \e \max_{1\leq i\leq n, 1\leq j\leq d} V_{ij}^4 \bigg)^{1/2}   \nn \\
	& \lesssim  v_4^{1/2} K_1^2  [ n s \log\{ ep / (\epsilon s ) \} ]^{1/2}   + (K_0 K_1)^2  [  s \log\{ ep/( \epsilon s )\}  \vee \log n ]^3 .     \label{eq7.29}
\end{align}

Putting \eqref{eq7.26}--\eqref{eq7.29} together, we obtain that for every $\de>0$ and $\epsilon \in (0,1)$,
\begin{align}
	& \P\bigg( \bigg| \max_{\balpha \in \mathcal{N}_{\epsilon}}   \balpha_{\bSigma}^{{\rm T}} \bW_n  - T^*_\epsilon \bigg| \geq 16   \de  \bigg) \nn \\
	& \lesssim v_4^{1/2}   K_1^2  \f{c^{3/2}_{n}(s,p,\epsilon)}{\de^2  \s{n}} + v_3  K_1^3   \f{c^2_{n}(s,p,\epsilon)}{\de^3  \s{n}} +  v_4 K_1^4 \f{c^5_{n }(s,p,\epsilon)}{ \de^4  n}  \nn \\
	& \qquad +   (K_0 K_1)^2   \f{c^4_{n }(s,p,\epsilon)}{\de^2 n } +   (K_0 K_1)^3 \f{c^6_{n }(s,p,\epsilon)}{\de^3  n^{3/2}}   +\f{\log n}{n},    \label{eq7.30}
\end{align}
where $c_{n }(s,p,\epsilon) := s \log \{ ep/( \epsilon s )\} \vee \log n$. Because this upper bound is only meaningful when it is less than 1, it can be further reduced to
\begin{align}
   b_{n }(s,p,\epsilon ,\delta)  : =  (K_0 K_1)^3 \f{c^2_{n }(s,p,\epsilon) }{\de^3 \s{n}} +  (K_0 K_1)^4  \f{c^5_{n }(s,p,\epsilon) }{ \de^4 n}   \label{eq7.31}
\end{align}
for $\epsilon \in (0,1)$ and $\de \in (0,  K_0 K_1]$.

\medskip
\noi
{\it \textbf{Step~3.}} For every $\epsilon \in (0,1)$, put $\epsilon_s=\gamma_s \epsilon$. A similar argument to that leading to \eqref{eq7.25} now gives
$$
	\max_{\balpha \in \mathcal{N}_\epsilon}    \balpha_{\bSigma}^{{\rm T}}  \bZ   \leq \sup_{\balpha \in \mathcal{V}}   \balpha_{\bSigma}^{{\rm T}}  \bZ    \leq \epsilon_s  \sup_{\balpha \in \mathcal{V}}  \balpha_{\bSigma}^{{\rm T}}  \bZ    + \max_{\balpha \in \mathcal{N}_\epsilon}    \balpha_{\bSigma}^{{\rm T}}  \bZ  .
$$
Further, it is concluded from (7.5) and \eqref{eq7.41} in the proof of Lemma~7.6 that, with probability at least $1-C n^{-1}$,
\begin{align}
	\bigg| \sup_{\balpha \in \mathcal{V}}  \balpha_{\bSigma}^{{\rm T}}  \bZ  - \max_{\balpha \in \mathcal{N}_\epsilon }    \balpha_{\bSigma}^{{\rm T}}  \bZ  \bigg|  & \lesssim   c^{1/2}_n(s,p)   \epsilon_s \label{eq7.32}
\end{align}
with $c_n(s,p) = s\log(\gamma_s  e p/s) \vee \log n$, and
\begin{align}
	&   \bigg| \sup_{\balpha \in \mathcal{V}}   \balpha_{\bSigma}^{{\rm T}}  \bW_n   - \max_{\balpha \in \mathcal{N}_\epsilon }    \balpha_{\bSigma}^{{\rm T}}  \bW_n  \bigg|  \nn \\
	&  \lesssim      \big\{ c^{1/2}_n(s,p) + K_0 K_1 \,  n^{-1/2} c^2_n(s,p) \big\} \epsilon_s .  \label{eq7.33}
\end{align}
For the Gaussian maxima $ \sup_{\balpha \in \mathcal{V}}\lea  \balpha_{\bSigma}, \bZ \ria $ and $ \max_{\balpha \in \mathcal{N}_\epsilon}\lea  \balpha_{\bSigma}, \bZ \ria $, it follows from \eqref{eq7.32} that for any Borel subset $\mathcal{B}$ of $\bbr$,
\beq
	\P\big(  T^*_\epsilon \in \mathcal{B} \big)
	\leq \P\bigg\{ \sup_{\balpha \in \mathcal{V} }  \balpha_{\bSigma}^{{\rm T}}  \bZ  \in \mathcal{B}^{C  \epsilon_s  c^{1/2}_n(s,p)   }\bigg\} +  n^{-1} ,
\eeq
where $\mathcal{B}^{u}:= \{x\in \bbr: |x-y|\leq u ,  \forall y \in \mathcal{B}\}$ for $u>0$. This, together with Lemma~4.1 in \cite{CCK14a}, a variant of Strassen's theorem, implies that there exits a random variable $T^*  \stackrel{d}{=} \sup_{\balpha\in \mathcal{V}}\lea  \balpha_{\bSigma}, \bZ\ria$ such that
\be
	\P\big\{  \big| T^* - T^*_\epsilon \big| > C \epsilon_s  c^{1/2}_n(s,p)   \big\} \leq n^{-1} . \label{eq7.34}
\ee

Finally, assembling \eqref{eq7.30}--\eqref{eq7.33} completes the proof of (7.9) by taking $\epsilon= (\gamma_s n)^{-1}$.   \qed

\subsection{Proof of Lemma~7.6}

Let $\mS_n = n^{-1}\sn \bX_i \bX_i^{{{\rm T}}}$ and write
\be
	D_n    = \sup_{\balpha \in \mathcal{V}} \big|  \balpha^{{{\rm T}}}_{\bSigma} \hat \bSigma_n  \balpha_{\bSigma}   -1 \big| = \sup_{\balpha \in \mathcal{V}} \Big|  \balpha_{\bSigma}^{{{\rm T}}}  \mS_n  \balpha_{\bSigma} -1 - \big( \balpha_{\bSigma}^{{{\rm T}}} \bar{\bX}_n \big)^2 \Big| .  \label{eq7.35}
\ee
For ease of exposition, define $\hat \sigma_\varepsilon^2= n^{-1}\sn (\varepsilon_i - \bar{\varepsilon}_n )^2$ and for $\balpha \in \bbr^p$, let
\be
	 \hat \sigma_{\balpha}^2 = \hat{\sigma}_{\balpha}^2(\bSigma)= n^{-1} \sn (  \balpha_{\bSigma}^{{\rm T}} \bX_i )^2 -  (  \balpha_{\bSigma}^{{\rm T}}  \bar{\bX}_n )^2 .    \label{eq7.36}
\ee
In this notation, $D_n=\sup_{\balpha \in \mathcal{V}}|\hat \sigma_{  \balpha}^2 - 1 |$ and
$$
	\hat  L_n = \sup_{\balpha \in \mathcal{V}} \, ( \hat \sigma_\varepsilon \hat \sigma_{ \balpha}  )^{-1} \bigg( n^{-1/2}\sn  \varepsilon_i \cdot  \balpha_{\bSigma}^{{\rm T}} \bX_i   - \sqrt{n} \, \bar{\varepsilon}_n \balpha_{\bSigma}^{{{\rm T}}} \bar{\bX}_n  \bigg) .
$$
Comparing this with $L_n$ in (7.8), it is easy to see that
\be
	 | \hat L_n - L_n | \leq \sup_{\balpha \in \mathcal{V}}\big|  (\hat \sigma_\varepsilon \hat \sigma_{ \balpha})^{-1} -1 \big| \cdot  L_n  +  \sqrt{n} \, \hat \sigma_\varepsilon^{-1}|\bar{\varepsilon}_n |  \cdot  \sup_{\balpha \in \mathcal{V}} \hat \sigma_{ \balpha}^{-1} |\balpha_{\bSigma}^{{{\rm T}}} \bar{\bX}_n | .  \label{eq7.37}
\ee
In what follows, we bound the two terms on the right-hand side of \eqref{eq7.37} separately, starting with the first one.

For every $t>0$, let $\mathcal{E}_{\bX}(t)$ and $\mathcal{E}_{\varepsilon}(t)$ be the events that (7.3) and (7.4) hold, respectively. In particular, taking $t_1=A_1 \log n$ for $A_1>0$ and $t_2= \min [ \log n,  \{ n/c_n(s,p)\}^{1/2} ]$ yields $\P\{  \mathcal{E}_{\varepsilon}(t_1)^{{\rm c}}  \} \leq  2 n^{-c_{{\rm H}} A_1} + 2 n^{-c_{{\rm B}} A_1}$ and
$$
\P\big\{ \mathcal{E}_{\bX}(t_2)^{{\rm c}} \big\}  \leq 8 \exp\big[- \{ n/c_n(s,p)\}^{1/2} \big] \leq 3 n^{-1/2}  c^{1/2}_n(s,p)  ,
$$
where $c_{{\rm H}}, c_{{\rm B}}>0$ are as in Lemma~7.1. Here, the last step comes from the inequality $\sup_{t\geq 0} te^{-t} \leq e^{-1}$. On the event $\mathcal{E}_{\varepsilon}(t_1) \cap \mathcal{E}_{\bX}(t_2)$,
\begin{align}
	\big|  \hat \sigma_\varepsilon^2 -1 \big| \lesssim  K_0^2\sqrt{\frac{\log n}{n}} \leq \f{1}{2}, \quad  D_n  \lesssim    K_1^2 \sqrt{\frac{c_n(s,p)}{n}} \leq \f{1}{2 }   \label{eq7.38}
\end{align}
whenever the sample size $n$ satisfies $n \gtrsim  \max\{ K_0^4 \log n ,  K_1^4  c_n(s,p)  \}$. Together, \eqref{eq7.37}, \eqref{eq7.38} and the identity
\begin{align*}
	 &  1-   (\hat \sigma_\varepsilon \hat \sigma_{ \balpha})^{-1}  \\
	 &= (\hat \sigma_\varepsilon \hat \sigma_{ \balpha})^{-1}  \big\{ (\hat \sigma_\varepsilon -1) (\hat \sigma_{ \balpha}-1)+\hat \sigma_\varepsilon-1 + \hat \sigma_{\balpha} -1 \big\}   \\
	& = \frac{ (\hat \sigma_\varepsilon^2 -1) (\hat \sigma_{ \balpha}^2-1)+( \hat \sigma^2_\varepsilon-1 )(\hat \sigma_{\balpha} +1)+ (\hat \sigma_{\balpha}^2 -1)(\hat \sigma_{\varepsilon} +1)}{\hat \sigma_\varepsilon \hat \sigma_{ \balpha}  (\hat \sigma_\varepsilon +1)(\hat \sigma_{\balpha} +1)}
\end{align*}
imply, on $\mathcal{E}_{\varepsilon}(t_1) \cap \mathcal{E}_{\bX}(t_2)$ with $n$ sufficiently large,
\begin{align}
	\sup_{\balpha \in \mathcal{V}}\big|  (\hat \sigma_\varepsilon \hat \sigma_{ \balpha})^{-1} -1 \big|   \lesssim   (K_0 \vee K_1)^2 \,  n^{-1/2} c_n^{1/2}(s,p)  . \label{eq7.39}
\end{align}

Next we deal with $L_n$, which can be written as $\sup_{\balpha \in \mathcal{V}} n^{-1/2} \sn  \balpha_{\bSigma}^{{\rm T}} \by_i $, where $\by_i=\varepsilon_i \bX_i$ satisfies that, under Condition~2.1,
\be
	\e  ( \balpha_{\bSigma}^{{\rm T}}  \by_i )^2 =1 \, \mbox{ for all } \,   \balpha \in \mathcal{V}  \ \ \mbox{ and } \ \   \sup_{\balpha \in \mathcal{V}} \|   \balpha_{\bSigma}^{{\rm T}}  \by_i   \|_{\psi_1}  \leq 2 K_0 K_1.  \label{eq7.40}
\ee
As in the proof of \eqref{eq7.12}, using Theorem 4 in \cite{A08} gives, for any $t\geq 0$,
\be
	L_n \leq 2 \e L_n + \max\bigg\{ 2 \sqrt{t} ,  C  \frac{t}{\s{n}}  \bigg\|  \max_{1\leq i\leq n}\sup_{\balpha \in \mathcal{V}} |  \balpha_{\bSigma}^{{\rm T}} \by_i  | \bigg\|_{\psi_1} \bigg\}  \label{eq7.41}
\ee
holds with probability at least $1-4  e^{-t}$. For the last term of \eqref{eq7.41}, a similar argument to that leading to \eqref{eq7.12} gives, on this occasion with $\epsilon_s=(4\gamma_s)^{-1}$ and $\mathcal{N} = \mathcal{N}_{\epsilon_s}$ that
\begin{align*}
	 \bigg\|  \max_{1\leq i\leq n}\sup_{\balpha \in \mathcal{V}} |   \balpha_{\bSigma}^{{\rm T}} \by_i   |  \bigg\|_{\psi_1}
	&  = \bigg\|  \max_{1\leq i\leq n}\sup_{\balpha \in \mathcal{V}} \Big|  \varepsilon_i \cdot ( \bSigma^{1/2}  \balpha_{\bSigma})^{{\rm T}}  \bU_i  \Big|  \bigg\|_{\psi_1} \\
	&  \lesssim \bigg\|  \max_{1\leq i\leq n}\sup_{\balpha \in \mathcal{N} }\Big| \varepsilon_i \cdot  ( \bSigma^{1/2}  \balpha_{\bSigma} )^{{\rm T}}  \bU_i    \Big|  \bigg\|_{\psi_1} \\
	 &  \lesssim  \{   s \log( \gamma_s p /s ) + \log n  \}  \| \varepsilon \|_{\psi_2}   \sup_{\balpha \in \bS^{p-1}}  \|   \balpha^{{\rm T}} \bU_i    \|_{\psi_2} \\
	&    \lesssim  K_0 K_1 \, c_n(s,p) .
\end{align*}
Here, we used the property that the cardinality of the $\epsilon_s$-net $\mathcal{N}$ of $\mathcal{V}$, denoted by $d =|\mathcal{N} |$, is such that $\log d  \lesssim s\log ( \gamma_s e p/s)$.

To bound $\e L_n$, observe that for every $\bu \in\bbr^p$, $\sup_{\balpha\in \mathcal{V}}|   \balpha_{\bSigma}^{{\rm T}}  \bu  |= \sup_{\balpha\in \mathcal{V}}     \balpha_{\bSigma}^{{\rm T}} \bu $. For $\bW_n = n^{-1/2} \sn \by_i$, it follows from \eqref{eq7.25} with $\epsilon_s =( 4\gamma_s)^{-1}$ that
$$
 \e L_n    =  \e \bigg(  \sup_{\balpha \in \mathcal{V}}   \balpha_{\bSigma}^{{\rm T}} \bW_n   \bigg) \leq \frac{4}{3}  \e \bigg( \max_{\balpha \in \mathcal{N} }   \balpha_{\bSigma}^{{\rm T}} \bW_n   \bigg).
$$
For every $\balpha \in \mathcal{V}$ and $t>0$, from (7.1) and \eqref{eq7.40} we get
\begin{align}
	\P\big(   \balpha_{\bSigma}^{{\rm T}} \bW_n \geq  t \big) \leq 2\exp\bigg\{ -c_{{\rm B}} \min\bigg( \frac{t^2}{4K_0^2 K_1^2}, \frac{\sqrt{n} t}{2K_0 K_1} \bigg) \bigg\}.  \label{eq7.42}
\end{align}
Hence, applying Lemma~7.3 with slight modification gives
\begin{align*}
\e \bigg( \max_{\balpha \in \mathcal{N} }    \balpha_{\bSigma}^{{\rm T}} \bW_n   \bigg) \lesssim  K_0K_1 \s{ \log d } + K_0 K_1 n^{-1/2} \log d .
\end{align*}
Plugging this into \eqref{eq7.41} and taking $t=\min [ \log n, \{n/c_n(s,p)\}^{1/2}]$ imply, with probability at least $1-4 \exp[-\{n/c_n(s,p)\}^{1/2} ] \geq 1 -2 n^{-1/2} c_n^{1/2}(s,p)$,
\be
	 L_n  \lesssim  K_0 K_1 \,  c^{1/2}_n(s,p)  \label{eq7.43}
\ee
whenever the sample size $n$ satisfies $n\gtrsim c_{n}(s,p)$.

Again, on the event $\mathcal{E}_{\varepsilon}(t_1) \cap \mathcal{E}_{\bX}(t_2)$ with $n$ sufficiently large as above for \eqref{eq7.38}, the second term on the right-hand side of \eqref{eq7.37} is bounded by some multiple of $\sqrt{n}  \, \hat{\sigma}_{\varepsilon}^{-1} |\bar{\varepsilon}_n | \sqrt{D_{n,2}}$, where $D_{n,2}$ is as in \eqref{eq7.12}. Arguments similar to those in the proof of Lemma~7.2 permit us to show that, with probability at least $1 -2 n^{-1/2} c_n^{1/2}(s,p)$,
\be
  \sqrt{D_{n,2}}  \lesssim   K_1  \, n^{-1/2} c^{1/2}_n(s,p)  .  \label{eq7.44}
\ee
Further, put $S_n = \sn \varepsilon_i, V_n^2=\sn \varepsilon_i^2$, such that $\sqrt{n}\, \hat{\sigma}_{\varepsilon}^{-1}  \bar{\varepsilon}_n = V_n^{-1} S_n \{	1-n^{-1}(S_n/V_n)^2 \}^{1/2}  $. Then it follows from Theorems 2.16 and 2.19 in \cite{DLS09} that for every $t\in (0, \sqrt{n} ]$
\begin{align*}
 & 	\P \big( \sqrt{n}\, \hat{\sigma}_{\varepsilon}^{-1}  | \bar{\varepsilon}_n |  \geq t  \big) \\
 &  \leq \P\big\{   | S_n |  \geq t(1+t^2 n^{-1})^{-1/2}  V_n \big\} \\
	& \leq \P\big\{  | S_n |  \geq t(1+t^2 n^{-1})^{-1/2} (4\sqrt{2}+1)^{-1} ( 4\sqrt{n} + V_n ) \big\}  + \P\big(  V_n^2 \leq n/2 \big)  \\
	& \leq 4\exp(-c_{{\rm SN}}  t^2 ) + \exp\{ -n / (8v_4) \},
\end{align*}
where $v_4 = \e (\varepsilon^4)$ and $c_{{\rm SN}}>0$ is an absolute constant. In particular, taking $t=A_2 ( \log n )^{1/2} $ with $A_2>0$ yields, with probability greater than $1-4 n^{-c_{{\rm SN}}A_2^2}- \exp\{ -n / (8v_4) \}$,
\be
\sqrt{n}\, \hat{\sigma}_{\varepsilon}^{-1}  | \bar{\varepsilon}_n | \lesssim \sqrt{\log n} .  \label{eq7.45}
\ee

Finally, combing \eqref{eq7.37}, \eqref{eq7.39}, \eqref{eq7.43}, \eqref{eq7.44} and \eqref{eq7.45} completes the proof of (7.10).   \qed

\section{Proof of Theorem~3.2}

We divide the proof into three key steps.  The first step is to establish \eqref{eqA.1} using the results on discretization in the proof of Theorem~3.1, and then analyze separately the order of the stochastic terms \eqref{eqA.2} and \eqref{eqA.3}.

\medskip
\noi
{\it \textbf{Step~1.}} For any $\balpha \in \bbr^p$, let $ \| \balpha \|_n^2 = \balpha^{{{\rm T}}} \hat{\bSigma}  \balpha$ with $\hat{\bSigma} = \hat{\bSigma}_n$, and let $\gamma_s=\gamma_s(\bSigma)$ be as in (2.4). First, we prove that there exits a $(\gamma_s n)^{-1}$-net of $\mathcal{V}$, denoted by $\mathcal{V}_{n}=\mathcal{V}_n(s,p)$, such that $\log(|\mathcal{V}_n|)\lesssim s  b_n(s,p) $ and
\beqn
	& \sup_{t\geq 0} \Big| \P\Big( \sup_{\balpha \in \mathcal{V}} \lea  \balpha_{\bSigma}, \bZ  \ria \leq t \Big) -  \P \Big( \sup_{\balpha \in \mathcal{V} } \lea  \balpha_n, \bZ_n \ria \leq t \, \Big|  \mathcal{X}_n \Big) \Big| \nn \\
	& \qquad \qquad  \lesssim  \Big[  \hat\gamma_{s} (\gamma_s n)^{-1}   s \bar{b}_{n}(s,p)  +  	 \De_{n}^{1/3}   \big\{ s \bar{b}_{n}(s,p) + \log(1/ 	 \De_{n } ) \big\}^{2/3} \Big], \label{eqA.1}
\eeqn
where $\balpha_n=\balpha/\|\balpha\|_n$, $\mathcal{X}_n = \{ \bX_i\}_{i=1}^n$, $b_{n}(s,p)=\log ( \gamma_s p/s) \vee \log n$,  $\bar{b}_{n}(s,p)=\log(  \bar{ \gamma}_s p /s ) \vee \log n$ with $\bar{\gamma}_s=\max( \gamma_s, \hat\gamma_{s} )$, and
\be
	\hat\gamma_{s} := \gamma_{s}(\hat \bSigma)  =  \f{\max_{\bu\in \bS^{p-1} : 1\leq |\bu|_0\leq s } \| \bu \|_n }{\min_{\bu \in \bS^{p-1}: 1\leq |\bu|_0\leq s }\| \bu \|_n  }  \label{eqA.2}
\ee
denotes the $s$-sparse condition number of $\hat \bSigma$ and
\be
   	\De_{n} = \Delta_n(s,p) = \max_{\balpha, \bbeta \in \mathcal{V}_{n}} \left|   \balpha_{\bSigma}^{{{\rm T}}} \bSigma \bbeta_{\bSigma}   -    \balpha_n^{{{\rm T}}}   \hat \bSigma   \bbeta_n   \right|   \label{eqA.3}
\ee
with $\balpha_n=\balpha/\|\balpha\|_n$ and $\bbeta_n = \bbeta/\|\bbeta\|_n$. \\

\noi
{\it Proof of \eqref{eqA.1}}. As in the proof of Lemma~7.5 in Appendix~\ref{AppE.4}, for every $ \epsilon \in (0,1)$, there exists an $\epsilon$-net $\mathcal{N}_\epsilon$ of $\mathcal{V}$ satisfying $d_\epsilon=|\mathcal{N}_\epsilon|\leq  \{  (2+\epsilon) ep / ( \epsilon s  ) \}^s$, such that
\begin{align}
 \bigg| \sup_{\balpha \in \mathcal{V}} \lea  \balpha_{\bSigma}, \bZ \ria - \max_{\balpha\in \mathcal{N}_\epsilon}\lea    \balpha_{\bSigma}, \bZ \ria \bigg| &  \leq \gamma_s \epsilon    \sup_{\balpha \in \mathcal{V}}\lea \balpha_{\bSigma}, \bZ\ria, \label{eqA.4} \\
  \bigg| \sup_{\balpha \in \mathcal{V}} \lea  \balpha_n , \bZ_n \ria - \max_{\balpha\in \mathcal{N}_\epsilon}\lea    \balpha_n, \bZ_n \ria \bigg| &  \leq	 \hat\gamma_{s} \epsilon   \sup_{\balpha \in \mathcal{V}}\lea  \balpha_n, \bZ_n \ria.  \label{eqA.5}
\end{align}
For notational convenience, write $d=d_{\epsilon}$, $\mathcal{N}_\epsilon=\{\balpha_1, \ldots, \balpha_d\}$ and let
\begin{align*}
	\bG & = (G_1,\ldots, G_d)^{{{\rm T}}} = \left( \lea \balpha_{1,\bSigma}, \bZ\ria, \ldots, \lea \balpha_{d,\bSigma} , \bZ \ria  \right)^{{{\rm T}}}, \\
	  \bG_n & = (G_{n1}, \ldots, G_{nd})^{{{\rm T}}} = \left( \lea \balpha_{1,n}, \bZ_n \ria, \ldots, \lea \balpha_{d,n} , \bZ_n \ria  \right)^{{{\rm T}}}
\end{align*}
be two $d$-dimensional centered Gaussian random vectors. Conditional on $\mathcal{X}_n$, applying Theorem~2 in \cite{CCK14b} to $\bG$ and $\bG_n$ respectively gives
\begin{align*}
	  \sup_{t\in \bbr} \bigg| \P\bigg( \max_{1\leq j\leq d} G_j \leq t \bigg)  &- \P \bigg( \max_{1\leq j\leq d} G_{nj} \leq t \, \bigg| \mathcal{X}_n \bigg) \bigg|  \\
	& \qquad \lesssim ( \De  \log d )^{1/3}   \big\{ \log d + \log(1/\De ) \big\}^{1/3},
\end{align*}
where $\De  =\De({\mathcal{N}_\epsilon})= \max_{\balpha, \bbeta \in \mathcal{N}_\epsilon}  \big|  \balpha_{\bSigma}^{{{\rm T}}}  \bSigma   \bbeta_{\bSigma} -  \balpha_n^{{{\rm T}}} \hat \bSigma     \bbeta_n \big|$.

By Lemma~7.3, we have for every $t>0$,
\begin{align*}
	\P \bigg\{ \sup_{\balpha \in \mathcal{V}}\lea \balpha_{\bSigma} , \bZ\ria \geq C  \s{s \log(\gamma_s ep /s )} + t \bigg\} & \leq e^{-t^2/2}, \\
		\P \bigg\{ \sup_{\balpha \in \mathcal{V}}\lea \balpha_n , \bZ_n \ria \geq C    \s{s \log ( \hat\gamma_{s} e p/s) } + t  \, \bigg| \mathcal{X}_n \bigg\} & \leq e^{-t^2/2}.
\end{align*}
The last three displays, together with \eqref{eqA.4} imply that, for every $\epsilon \in (0,1)$ and $t>0$,
\begin{align*}
	&	\P\bigg( \sup_{\balpha\in \mathcal{V}} \lea  \balpha_n , \bZ_n \ria \leq  t \, \bigg| \mathcal{X}_n  \bigg) \\
	& 	\leq \P\bigg( \max_{\balpha\in \mathcal{N}_\epsilon} \lea  \balpha_n , \bZ_n \ria \leq  t \, \bigg| \mathcal{X}_n  \bigg) \\
	& \leq   \P\bigg( \max_{\balpha\in \mathcal{N}_\epsilon} \lea  \balpha_{\bSigma}, \bZ  \ria \leq t  \bigg) + C  \De^{1/3} (\log d)^{1/3}   \{ \log(d/\De)  \}^{1/3} \\
	& \leq \P\bigg\{ \sup_{\balpha\in \mathcal{V}} \lea  \balpha_{\bSigma}, \bZ  \ria \leq  t + C   \epsilon \gamma_s \, c^{1/2}_n(s,p) \bigg\} \\
	& \qquad \qquad \quad  \qquad \qquad \qquad + C \De^{1/3} (\log d)^{1/3} \{  \log(d/\De)\}^{1/3} + n^{-1} \\
	& \leq \P \bigg( \sup_{\balpha  \in \mathcal{V}} \lea  \balpha_{\bSigma}, \bZ  \ria \leq  t   \bigg) +  C \epsilon  \gamma_s \,  c_n(s,p) \\
	& \qquad \qquad \qquad  \quad \qquad \qquad+ C \De^{1/3} (\log d)^{1/3}  \{ \log(d/\De) \}^{1/3} + n^{-1},
\end{align*}
where $ c_n(s,p)= s\log ( \gamma_s e p/s) \vee \log n$ and the last inequality comes from (7.6). For the lower bound, in view of \eqref{eqA.5}, it can be similarly obtained that
\begin{align*}
&	\P\bigg( \sup_{\balpha\in \mathcal{V}} \lea  \balpha_n , \bZ_n \ria \leq  t \, \bigg| \mathcal{X}_n  \bigg)  \\
&	\geq \P\bigg\{ \max_{\balpha\in \mathcal{N}_\epsilon} \lea  \balpha_n , \bZ_n \ria \leq  t- C \epsilon  \hat\gamma_{s} \sqrt{ s \log (  \hat\gamma_{s}  e p /s ) \vee \log n } \, \bigg| \mathcal{X}_n \bigg\} -  n^{-1}  \\
& \geq \P\bigg( \sup_{\balpha \in \mathcal{V}} \lea  \balpha_{\bSigma}, \bZ \ria \leq  t \bigg)  -   C \epsilon \hat\gamma_{s}  \{  s \log ( \hat\gamma_{s}   e p /s ) \vee \log n \}   \\
& \qquad  \qquad \qquad \qquad \qquad  -  C \De^{1/2} (\log d)^{1/3} \{ \log(d/\De) \}^{1/3}-  n^{-1}.
\end{align*}
Taking $\mathcal{V}_{n}=\mathcal{N}_{\epsilon}$ with $\epsilon=( \gamma_s  n)^{-1}$ proves \eqref{eqA.1}.

\medskip
\noi
{\it \textbf{Step~2.}} Next, we study $\De_{n}$ in \eqref{eqA.3}, which is bounded by
\begin{align}
   \max_{\balpha, \bbeta \in \mathcal{V}_{n}} \big|  \balpha_{\bSigma}^{{{\rm T}}} (\hat \bSigma-\bSigma)  \bbeta_{\bSigma}   \big| + \max_{\balpha, \bbeta \in \mathcal{V}_{n}} \big| \balpha^{{{\rm T}}}_n \hat \bSigma  \bbeta_n  -  \balpha_{\bSigma}^{{{\rm T}}} \hat \bSigma   \bbeta_{\bSigma}  \big|   :=   \De_{n,1} + \De_{n,2}.  \label{eqA.6}
\end{align}
In what follows, we bound the two terms on the right side separately, starting with $\De_{n,2}$. For every $\balpha , \bbeta \in \mathcal{V}_{n}$,
\begin{align}
	& \big|  \balpha_n^{{{\rm T}}} \hat \bSigma   \bbeta_n -  \balpha_{\bSigma}^{{{\rm T}}} \hat \bSigma   \bbeta_{\bSigma} \big| \nn \\
  &=  \big|  \balpha_n^{{{\rm T}}} \hat \bSigma ( \bbeta_n - \bbeta_{\bSigma} ) + \balpha_n^{{{\rm T}}} \hat \bSigma \bbeta_{\bSigma}  -  \balpha_{\bSigma}^{{{\rm T}}} \hat \bSigma   \bbeta_{\bSigma} \big| \nn \\
	&  \leq \| \balpha_n  \|_n \|  \bbeta_n-  \bbeta_{\bSigma} \|_n +  \| \bbeta_{\bSigma} - \bbeta_n \|_n \|  \balpha_{\bSigma}- \balpha_n \|_n + \| \bbeta_n \|_n \| \balpha_n - \balpha_{\bSigma} \|_n \nn  \\
	& \leq  \big| \| \balpha_{\bSigma} \|_n -1 \big| + \big| \|  \bbeta_{\bSigma} \|_n -1 \big|+\big| \| \balpha_{\bSigma}  \|_n -1 \big| \cdot  \big| \| \bbeta_{\bSigma} \|_n -1 \big|.  \nn
\end{align}
Using this together with Lemma~7.2 yields, with probability at least $1-3 n^{-1/2}c^{1/2}_n(s,p)$,
\be
	\De_{n,2} \lesssim K_1^2  \, n^{-1/2}  c^{1/2}_n(s,p).   \label{eqA.8}
\ee

Turning to $\De_{n,1}$, it suffices to focus on
\be
	\max_{\balpha, \bbeta \in \mathcal{V}_{n}} \big|  \balpha_{\bSigma}^{{{\rm T}}} (\mS_n-\bSigma)  \bbeta_{\bSigma} \big| = \max_{\balpha, \bbeta \in \mathcal{V}_{n}} \bigg| n^{-1} \sn \big\{ (  \balpha_{\bSigma}^{{{\rm T}}} \bX_i )( \bbeta_{\bSigma}^{{{\rm T}}} \bX_i )-  \balpha_{\bSigma}^{{{\rm T}}} \bSigma   \bbeta_{\bSigma} \big\} \bigg|.  \label{eqA.9}
\ee
Applying Theorem 4 in \cite{A08} we obtain that, with probability at least $1-4e^{-t}$,
\begin{align}
	   & 	\max_{\balpha, \bbeta \in \mathcal{V}_{n}} \big|  \balpha_{\bSigma}^{{{\rm T}}} (\mS_n-\bSigma)  \bbeta_{\bSigma} \big|  \nn \\
	  	& \leq 2 \e \bigg\{ \,  	\max_{\balpha, \bbeta \in \mathcal{V}_{n}} \big|  \balpha_{\bSigma}^{{{\rm T}}} (\mS_n-\bSigma) \bbeta_{\bSigma} \big|  \bigg\} \nn \\
	  & \quad  + \max\bigg\{ 2 \sigma_{\mathcal{V}_n} \f{\s{t}}{n} + C\f{t}{n} \bigg\| \max_{1\leq i\leq n}\max_{\balpha, \bbeta \in \mathcal{V}_n}   \balpha_{\bSigma}^{{{\rm T}}} \bX_i \bX_i^{{{\rm T}}}  \bbeta_{\bSigma}  \bigg\|_{\psi_1} \bigg\}, \label{eqA.10}
\end{align}
where similarly to \eqref{eq7.6} and \eqref{eq7.8},
\begin{align}
	\sigma_{\mathcal{V}_{n}}^2 & := \max_{\balpha, \bbeta \in \mathcal{V}_{n}} \sn \e ( \balpha_{\bSigma}^{{{\rm T}}} \bX_i \,   \bbeta_{\bSigma}^{{{\rm T}}} \bX_i)^2 \nn \\
	&  \leq \max_{\balpha, \bbeta \in \mathcal{V}_{n}} \sn \big\{ \e ( \balpha_{\bSigma}^{{{\rm T}}} \bX_i)^4\big\}^{1/2} \big\{ \e (   \bbeta_{\bSigma}^{{{\rm T}}} \bX_i)^4 \big\}^{1/2} \leq 16 K_1^4 \, n \label{eqA.11}
\end{align}
and
\be
\bigg\| \max_{1\leq i\leq n}\max_{\balpha, \bbeta \in \mathcal{V}_{n}}   \balpha_{\bSigma}^{{{\rm T}}} \bX_i \bX_i^{{{\rm T}}}   \bbeta_{\bSigma} \bigg\|_{\psi_1} \lesssim  K_1^2  c_n(s,p)  . \label{eqA.12}
\ee	
From the moment inequality $\|  \balpha_{\bSigma}^{{{\rm T}}}  \bX_i \, \bbeta_{\bSigma}^{{{\rm T}}} \bX_i \|_{\psi_1} \leq 2 K_1^2$, another consequence of (7.1) is that
\begin{align*}
	 \P\bigg( \bigg| n^{-1}\sn \big\{ (  \balpha_{\bSigma}^{{{\rm T}}} \bX_i )( \bbeta_{\bSigma}^{{{\rm T}}} \bX_i ) & - \balpha_{\bSigma}^{{{\rm T}}} \bSigma \bbeta_{\bSigma} \big\}  \bigg| \geq t \bigg)  \\
	&   \leq 2 \exp\{ - c_{{\rm B}} n \min(  t^2/K_1^4 , t/ K_1^2 ) \}.
\end{align*}
Using this together with Lemma~7.4 we get
\be
	\e \bigg\{ \,  \max_{\balpha, \bbeta \in \mathcal{V}_n} \big|  \balpha_{\bSigma}^{{{\rm T}}} (\mS_n-\bSigma)  \bbeta_{\bSigma} \big|  \bigg\} \lesssim K_1^2 \s{\f{\log(|\mathcal{V}_n|)}{n}} + K_1^2 \f{\log(|\mathcal{V}_n|)}{n}. \label{eqA.13}
\ee
Consequently, combining \eqref{eqA.9}--\eqref{eqA.13} gives, with probability at least $1- 2n^{-1/2}c^{1/2}_n(s,p)$,
\be
	\De_{n,1} \lesssim  K_1^2 \, n^{-1/2}  \{ s  b_{n}(s,p) \}^{1/2} . \label{eqA.14}
\ee
Together, \eqref{eqA.6}, \eqref{eqA.8} and \eqref{eqA.14} imply that, with probability at least $1-5 n^{-1/2} c^{1/2}_n(s,p)$,
\be
	\De_{n } \lesssim  K_1^2 \, n^{-1/2}  \{ s   b_{n}(s,p) \}^{1/2}    .  \label{eqA.15}
\ee

\noi
{\it \textbf{Step~3.}} Finally, we study the sample $s$-sparse condition number $\hat\gamma_{s}$ in \eqref{eqA.2}. For every $\bu \in \bbr^p$, note that $( \|\bu\|_n / |\bu |_{\bSigma} )^2 =  \bu_{\bSigma}^{{{\rm T}}}  \hat \bSigma  \bu_{\bSigma}    = 1+    \bu_{\bSigma}^{{{\rm T}}}  \hat \bSigma  \bu_{\bSigma} -1$. For every $\epsilon\in (0,1)$, in view of the inequality $\sum_{j=1}^s {p\choose j} \leq (ep /s )^s$ that holds for all $1\leq s\leq p$, there exists an $\epsilon$-net of $\{ \mx \mapsto \lea \bu, \mx \ria: \bu \in \bS^{p-1}, 1\leq |\bu|_0 \leq s \}$ with its cardinality bounded by $\{ (2+\epsilon) ep/( \epsilon s )  \}^s$. Consequently, it follows from Lemma~7.2 and the previous display that, with probability at least $1-C n^{-1/2}   c^{1/2}_n(s,p)$,
\begin{align} \label{eqA.16}
	 \tfrac{1}{2} \leq   \big( \|\bu\|_n / |\bu |_{\bSigma} \big)^2 \leq \tfrac{3}{2} \ \ \mbox{ for all } \bu \in \bS^{p-1}  \, \mbox{ satisfying } \,  1\leq |\bu|_0\leq s
\end{align}
and hence, $\hat{\gamma}_s \leq 3\gamma_s$ whenever $n$ satisfies $n \gtrsim K_1^4 c_n(s,p)$.

Assembling \eqref{eqA.1}, \eqref{eqA.15} and \eqref{eqA.16} completes the proof of Theorem~3.2. \qed

\section{Discussion on the moment assumptions}
\label{discuss.sec}

As pointed out in Remark~3.3, the sub-exponential rate, i.e. $\log p \asymp n^c$ with some $c \in (0,1)$, requires a sub-Gaussian condition on the sampling distribution. In the following, we will discuss the main steps on how our analysis can be carried over under finite moment conditions, at the cost of imposing more stringent constraints on the dimension $p$ as a function of the sample size $n$.

Note that, inequality \eqref{eq7.26} in the proof of Lemma~7.5 holds with $B_1$--$B_3$ well-defined as long as the fourth moments of all coordinates of $\varepsilon \bX$ are finite. From the proof of Lemma~7.6 we see that the main difficulty comes from bounding
$$
	L_n = L_n(s, p  ) = \sup_{\alpha \in \mathcal{V}}  \frac{1}{\sqrt{n}} \sn \langle   \balpha_{\bSigma} , \varepsilon_i \bX_i  \rangle
$$
with $\balpha_{\bSigma} =  \balpha / | \bSigma^{1/2} \balpha |_2 $ and
$$
	\sup_{\balpha \in \mathcal{V}}  | \balpha^\T \hat{\bSigma} \balpha  - \balpha^\T \bSigma \balpha |,
$$
where $\mathcal{V} = \mathcal{V}(s,p) = \{ \balpha \in \mathbb{S}^{p-1}: | \balpha |_0 = s\}$ for $1\leq s\leq p$. Without loss of generality, we let $\hat{\bSigma} = \hat{\bSigma}_n = n^{-1} \sn \bX_i \bX_i^\T$, where $\bX_1, \ldots,\bX_n$ are i.i.d. copies of a random vector $\bX \in \bbr^p$.

 Deviation bounds for the above two terms are given in the following lemmas. Comparing \eqref{supp.eq1} and \eqref{supp.eq2} with \eqref{eq7.43} and (7.3), respectively,  we see that the consistency of normal approximation requires significantly more stringent condition on the dimension $p$ under finite fourth moment assumptions. In this case, the convergence in Kolmogorov distance
$$
	\sup_{t \geq 0} \big| \PP \big\{ \s{n}  \wh R_n(s,p) \leq t \big\} -    \PP\big\{  R^*(s,p) \leq t \big\}  \big|
$$
holds when $s \log(p) = o(\log n)$ as $n\to \infty$. Because our main focus is on characterizing spurious discoveries from variable selection methods for high-dimensional data with low-dimensional structure, the above result becomes less instructive and is not applicable to the statistical problems considered in this paper. Nonetheless, the study of distributional approximation for heavy-tailed data has its own interest and is also our ongoing work [\cite{SZF2016}].

\begin{assumption} \label{weak.moment}
The random variable $\varepsilon$ satisfies $\e \varepsilon=0$, $\e \varepsilon^2 = 1 $ and $R_0 = \e \varepsilon^4 <\infty$.
There exists a random vector $\bU$ such that $\bX = \bSigma^{1/2}\bU$, $\e ( \bU )=\mo$, $\e ( \bU  \bU^{{{\rm T}}} )= \bI_p$ and $R_1 = \sup_{\bu \in \bbr^p} \e \langle \bu, \bU \rangle^4  <\infty$.
\end{assumption}

\begin{lemma} \label{supp.lm1}
Assume that Condition~\ref{weak.moment} holds. Then, for any $\delta \in (0,1)$,
\begin{align}
	L_n \leq C \max \{ (R_0 R_1)^{1/4} n ^{-1/4}  , 1  \}  (5e\gamma_s p/s)^{s/4}
\delta^{-1/4}  \label{supp.eq1}
\end{align}
with probability greater than $1-\delta$, where $C>0$ is an absolute constant.
\end{lemma}

\begin{proof}[Proof of Lemma~\ref{supp.lm1}]
As before, define $\bW_n =n^{-1/2}\sn \by_i$ with $\by_i = \varepsilon_i \bX_i$ for $i=1,\ldots,n$. By \eqref{eq7.25}, for any $\epsilon \in (0, \gamma_s^{-1})$, there exists a finite set $\mathcal{N}_\epsilon \subseteq \mathcal{V}$ such that $|\mathcal{N}_\epsilon | \leq {p \choose s}(1+ 2/\epsilon)^s$ and
\begin{align}
	L_n \leq (1- \gamma_s \epsilon )^{-1} \max_{\balpha \in \mathcal{N}_\epsilon }  | \langle \balpha_{\bSigma}, \bW_n \rangle | . \label{Ln.bd1}
\end{align}

For every $\balpha \in \mathcal{N}_\epsilon$ and $t>0$, by Markov's inequality and the Rosenthal inequality we have
\begin{align}
 & \P  (   | \langle \balpha_{\bSigma}, \bW_n \rangle | > t  )  \nn \\
 &  \leq  t^{-4} \,\e \langle \balpha_{\bSigma}, \bW_n \rangle^4  \nn \\
 & \lesssim  \frac{1}{t^4 n^2} \bigg\{  \sn \e  \langle \balpha_{\bSigma}, \by_i \rangle^4 + \bigg( \sn \e \langle \balpha_{\bSigma}, \by_i \rangle^2 \bigg)^2 \bigg\}  \nn \\
 & \lesssim R_0 R_1 \, t^{-4} n^{-1}    + t^{-4}  . \nn
\end{align}
This, together with \eqref{Ln.bd1} with $\epsilon = (2\gamma_s)^{-1}$ and the union bound implies that for any $\delta \in (0,1)$, $L_n \lesssim  \max \{ (R_0 R_1)^{1/4} n ^{-1/4}  , 1  \}  \delta^{-1/4}$ with probability at least $1- (5e \gamma_s p/s)^s  \delta $. By a simple algebra, we obtain \eqref{supp.eq1} as required.
\end{proof}

\begin{lemma} \label{supp.lm2}
Assume that Condition~\ref{weak.moment} holds. Then, there exists some absolute constants $C>0$ such that, for any $\delta \in (0,1)$,
\begin{align}
	\sup_{\balpha \in \mathcal{V}} |   \balpha^\T \hat{\bSigma}  \balpha - \balpha^\T \bSigma \balpha | \leq C  \, \phi_{\max}(s)   \sqrt{R_1}  \, s^{1/4}   \sqrt{\frac{  (8e  p/s)^s }{ \delta n } }  \label{supp.eq2}
\end{align}
with probability greater than $1-\delta$, where $\phi_{\max}(s)$ is the $s$-sparse maximal eigenvalue of $\bSigma$.
\end{lemma}

\begin{proof}[Proof of Lemma~\ref{supp.lm2}]

For  every $\balpha \in \mathcal{V}$, define
$$
	\hat{Q}(\balpha) = \balpha^\T \hat{\bSigma} \balpha  =\frac{1}{n} \sn (\balpha^\T \bX_i)^2 \ \ \mbox{ and }  \ \  {Q}(\balpha) = \balpha^\T  {\bSigma} \balpha.
$$
Applying Chebyshev's inequality to the quadratic form $\hat{Q}(\balpha)$, we obtain that for every $\delta>0$,
\begin{align}
	\P\bigg\{ |  \hat{Q}(\balpha) - Q(\balpha) | \geq  \sqrt{ \frac{ \var ( \langle \balpha , \bX  \rangle^2 )  }{ \delta n} } \, \bigg\} \leq  \delta . \label{supp1}
\end{align}
In view of Proposition~6.2 in \cite{C2012}, this upper bound is tight under the finite fourth moment condition. Moreover, for any $\epsilon \in (0, 1/2)$, it follows from Lemma~2 in the supplement to \cite{WBS2016} that there exists $\mathcal{N}_\epsilon \subseteq \mathcal{V}$ with cardinality at most $\pi (1-\epsilon^2/16)^{-(s-1)/2} \sqrt{s} { p \choose s} (2/\epsilon)^{s-1} $ such that
\begin{align}
	\sup_{\balpha \in \mathcal{V}} | \hat{Q}(\balpha) - Q(\balpha)| \leq (1-2\epsilon)^{-1} \max_{\balpha \in \mathcal{N}_\epsilon}  | \hat{Q}(\balpha) - Q(\balpha)|. \label{supp2}
\end{align}
Together, \eqref{supp1} and \eqref{supp2} with $\epsilon=1/4$ yield
\begin{align}
	\P \bigg\{  \sup_{\balpha \in \mathcal{V}} | \hat{Q}(\balpha) - Q(\balpha)| \geq  \sqrt{ \frac{  \sup_{\balpha \in \mathcal{V}}  \var( \langle \balpha, \bX  \rangle^2 )  }{\delta n}  }  \, \bigg\} \lesssim     \sqrt{s} \, (8ep/s)^s \delta. \nn
\end{align}
Combining this with the fact that
$$
	\langle \balpha, \bX  \rangle^2 = \langle \bSigma^{1/2} \balpha, \bU \rangle^2 =  \balpha^\T \bSigma \balpha \cdot  \bigg\langle  \frac{\bSigma^{1/2} \balpha}{| \bSigma^{1/2} \balpha |_2 } , \bU \bigg\rangle^2
$$
proves \eqref{supp.eq2}.
\end{proof}

\section{Proof of Theorem~4.1}

Without loss of generality, we assume that $\sigma^2=\e (\varepsilon_i^2)=1$ and $s\leq n\leq p\leq e^n$. The dependence of $\hat{R}_n^{{\rm oracle}}$ on $p$ will be assumed without displaying. Let $\hat{\varepsilon}_i=\hat{\varepsilon}_i^{\,{\rm oracle}}$, $\hat{\bbeta} =(\hat{\bbeta}^{{{\rm T}}}_1, \hat{\bbeta}_2^{{{\rm T}}})^{{{\rm T}}}= \hat{\bbeta}^{{\rm oracle}}$ and $\bdelta = (\bdelta_1^{{{\rm T}}}, \bdelta_2^{{{\rm T}}} )^{{{\rm T}}} =\hat{\bbeta} - \bbeta^*\in \bbr^p$ with $\bdelta_1 = \hat{\bbeta}_1-\bbeta_1 \in \bbr^s$ and $\bdelta_2 = \hat{\bbeta}_2-\bbeta_2 = \mo$. As in the proof of Theorem~3.1, we first consider the following standardized version of $\hat{R}_n^{{\rm oracle}}$:
\begin{align}
R_n^{ {\rm oracle}} =    \max_{j\in [p]}     \bigg| n^{-1} \sn  \hat{\varepsilon}_i X_{ij} \bigg| ,  \label{eqB.1}
\end{align}
Recall that $\bdelta=(\bdelta_1^{{{\rm T}}}, \bdelta_2^{{{\rm T}}})^{{{\rm T}}}$ with $\bdelta_2 = \mo$. For every $j\in [p]$, from the identity $\mathbb{X}_{1}^{{{\rm T}}} \mathbb{X}_1 \bdelta_1= \mathbb{X}_{1}^{{{\rm T}}} \beps$ we derive that
\be \label{eqB.2}
 \bSigma_{11}\bdelta_1= n^{-1} \mathbb{X}^{{{\rm T}}}_{1} \beps  +  \mathbf{b}_1 \quad \mbox{ with } \quad \mathbf{b}_1 := -\big( n^{-1} \mathbb{X}_{1}^{{{\rm T}}} \mathbb{X}_{1} -\bSigma_{11} \big)\bdelta_1.
\ee
Together with (4.4) and some simple algebra, this implies
\begin{align}
	&	     n^{-1} \sn \hat{\varepsilon}_i X_{ij}  \nn \\
	      &= n^{-1} \mathbf{e}_j(p)^{{{\rm T}}} \big(  \mathbb{X}^{{{\rm T}}} \beps  -    \mathbb{X}^{{{\rm T}}}   \mathbb{X} \bdelta \big) \nn \\
	 & = n^{-1}\sn  \mathbf{e}_j(p)^{{{\rm T}}} \bP  {\bX}_i \,    \varepsilon_i \nn \\
	 & \quad  -   \mathbf{e}_j(p)^{{{\rm T}}} \small\left(
\begin{array}{c}
 \mo_{s\times 1} \\
 \bSigma_{21}\bSigma_{11}^{-1}\bb_1 + \big( n^{-1}\mathbb{X}_2^{{{\rm T}}} \mathbb{X}_1 - \bSigma_{21} \big) \bdelta_1
\end{array}
\right)  , \label{eqB.3}
\end{align}
where $ \mathbf{e}_j(p)=(0, \ldots, 0, 1, 0 , \ldots, 0)^{{{\rm T}}} $ is the unit vector in $\bbr^p$ with 1 on the $j$th position and
\begin{align}
     \bP = \small\left(
\begin{array}{cc}
\mo_{s\times s} & \mo_{s\times d} \\
- \bSigma_{21}\bSigma_{11}^{-1} & \bI_{d}
\end{array}
\right)  \in \bbr^{p\times p}. \label{eqB.4}
\end{align}
In view of \eqref{eqB.3}, we define
\be
 \widetilde{R}_n  = \max_{j\in [p]}  \bigg| n^{-1} \sn   \mathbf{e}_j(p)^{{{\rm T}}} \bP {\bX}_i \,  \varepsilon_i \bigg| .  \label{eqB.5}
\ee
Together, \eqref{eqB.1}, \eqref{eqB.2}, \eqref{eqB.3} and \eqref{eqB.5} imply
\begin{align}
	& \big|   R^{{\rm oracle}}_n - \widetilde{R}_n  \big|  \nn \\
	&  \leq \max_{j\in[d]} \bigg| \mathbf{e}_{j}(d)^{{{\rm T}}} \bSigma_{21}\bSigma_{11}^{-1/2} \big( n^{-1}  \bSigma_{11}^{-1/2}\mathbb{X}_1^{{{\rm T}}} \mathbb{X}_1 \bSigma_{11}^{-1/2} - \bI_s\big) \bSigma_{11}^{1/2}\bdelta_1 \bigg|
		  \nn \\
 & \quad +	  \max_{j\in [p] \setminus [s]} \bigg|  n^{-1} \sn  \big( X_{ij} \bX_{i,1}^{{{\rm T}}} \bSigma_{11}^{-1/2}- \mathbf{e}_{j-s}(d)^{{{\rm T}}}\bSigma_{21} \bSigma_{11}^{-1/2}  \big) \bSigma_{11}^{1/2}\bdelta_1 \bigg|  \nn \\
	& \leq  \max_{j\in [d]}  \big| \bSigma_{11}^{-1/2}\bSigma_{12} \, \mathbf{e}_{j}(d) \big|_2
	 \big\| n^{-1}  \bSigma_{11}^{-1/2}\mathbb{X}_1^{{{\rm T}}} \mathbb{X}_1 \bSigma_{11}^{-1/2} - \bI_s\big\|  \big|  \bSigma_{11}^{1/2}\bdelta_1 \big|_2  \nn \\
	& \quad + \sqrt{s} \,\big|  \bSigma_{11}^{1/2}\bdelta_1 \big|_2  \max_{j\in  [p]\setminus [s] \atop k\in [s]}   \bigg|  n^{-1} \sn ( {\rm Id}-\e )  X_{ij} \, \mathbf{e}_{k}(s)^{{{\rm T}}}\bSigma_{11}^{-1/2} \bX_{i,1}  \bigg|  \nn \\
	& := Q_1 + Q_2,   \label{eqB.6}
\end{align}
where $( {\rm Id} -\e )Y := Y - \e Y$ for any random variable $Y$.

With the above preparations, the rest of the proof involves three steps: First, we prove the Gaussian approximation of $\sqrt{n} \widetilde{R}_n$ by the Gaussian maximum $\widetilde{R}^* :=| \widetilde{\bZ} |_\infty$, where $ \widetilde{\bZ} \sta{d}{=} N(\mo, \bSigma_{22.1})$. Second, we prove that $\sqrt{n}(Q_1+Q_2)$ is negligible with high probability and that $\sqrt{n}\hat{R}^{{\rm oracle}}_n$ and $\sqrt{n} \widetilde{R}_n$ are close. Finally, we apply an anti-concentration argument to prove the convergence in the Kolmogorov distance.

\medskip
\noi
{\it \textbf{Step~1: Gaussian approximation.}} First we prove that, under Condition~4.1 in the main text, there exists a random variable $\widetilde T^* \sta{d}{=} \widetilde{R}^*$ such that, for every $\delta\in (0, K_0K_1]$,
\begin{align}
  \P\big( \big|\sqrt{n} \widetilde{R}_n  - \widetilde{T}^* \big|  \geq 16 \delta \big) \lesssim    (K_0K_1)^3 \frac{(\log p)^2}{\delta^3 \sqrt{n}} + (K_0K_1)^4 \frac{(\log p)^5}{\delta^4 n}. \label{eqB.7}
\end{align}

By the definition of $\bP$ in \eqref{eqB.4}, we have
$$
	\sqrt{n }\widetilde{R}_n = \max_{j\in [p]\setminus [s]}  \bigg| n^{-1/2} \sn \mathbf{e}_j(p)^{{{\rm T}}} \bP {\bX}_i  \, \varepsilon_i \bigg|,
$$
where $[p] \setminus [s]=\{s+1, \ldots, p\}$. In addition, write $\bX_{i}=(\bX_{i,1}^{{{\rm T}}}, \bX_{i,2}^{{{\rm T}}})^{{{\rm T}}}$ with $\bX_{i,1} \in \bbr^s,  \bX_{i,2}\in \bbr^{d}$ and define $\widetilde{\by}_i=\varepsilon_i \widetilde \bX_i$, where $\widetilde{\bX}_i=\bX_{i,2} - \bSigma_{21}\bSigma_{11}^{-1}\bX_{i,1}$ are such that $\e ( \widetilde{\bX}_i ) = \mo$ and $\e ( \widetilde{\bX}_i  \widetilde{\bX}_i^{{{\rm T}}} )   = \bSigma_{22.1}$. In this notation, we can rewrite $\sqrt{n }\widetilde{R}_n$ as
\be
	\sqrt{n }\widetilde{R}_n = \max_{j\in [d]}  \bigg|  n^{-1/2} \sn  \mathbf{e}_j(d)^{{{\rm T}}}   \widetilde{\by}_i  \bigg|.  \label{eqB.8}
\ee

Next, we use the coupling inequality \eqref{eq7.26} below with $d=p-s$ to the random vectors $\bV_1,\ldots,\bV_n$ which, on this occasion, are defined by $\bV_i=(V_{i1},\ldots, V_{id})^{{{\rm T}}}$ with $V_{ij}= \mathbf{e}_j(d)^{{{\rm T}}}   \widetilde \by_i $. Since $\e (\varepsilon_i \bX_i)=0$ and $\e (\varepsilon_i^2|\bX_i)=1$, we have $\e ( \bV_i) = \mo$ and $\e ( \bV_i \bV_i^{{{\rm T}}} ) = \bSigma_{22.1}$. Then there exists a random variable $\widetilde T^*  \sta{d}{=} |  \widetilde{\bZ} |_\infty$ such that, for every $\delta>0$,
\begin{align}
	&  \P\big( \big| \sqrt{n} \widetilde R_n  - \widetilde T^*  \big| \geq 16 \de \big) \nn \\
	&  \lesssim  B_1 \frac{  \log p }{\de^2  n}  + B_2  \frac{   (\log p)^2}{\de^3 n^{3/2}}  + B_3 \frac{  (\log p)^3 }{\de^{4}   n^2}  + \f{\log n}{n}  .   \label{eqB.9}
\end{align}
In addition, note that the random vectors $\bV_i=(V_{i1},\ldots, V_{id})^{{{\rm T}}}$ are such that $\e ( V_{ij}^2 ) = \widetilde \sigma_{jj} \leq 1 $ and
\begin{align}
 	& \max_{j\in [d]} \big\|	\mathbf{e}_j(d)^{{{\rm T}}}   \widetilde{\by}_i   \big\|_{\psi_1}  \nn \\
    & \leq 2K_0   \max_{j\in [d]} \big\|	\mathbf{e}_j(d)^{{{\rm T}}} \widetilde{\bX}_i   \big\|_{\psi_2}  \nn \\
   & \leq   2K_0   \max_{j\in [p]\setminus [s] } \big\|	\mathbf{e}_j(p)^{{{\rm T}}}  \bP \bSigma^{1/2} {\bU}_i   \big\|_{\psi_2} \nn \\
   &  \leq 2K_0K_1 \max_{j\in [p]\setminus [s]}  \big\{ \mathbf{e}_j(p)^{{{\rm T}}}  \bP \bSigma \bP^{{{\rm T}}} \mathbf{e}_j(d)  \big\}^{1/2}   \nn \\
   & = 2K_0K_1 \max_{j\in [d]}  \widetilde \sigma_{jj}^{1/2} \leq 2 K_0 K_1 . \label{eqB.10}
\end{align}
Consequently, similar arguments to those leading to \eqref{eq7.27},  \eqref{eq7.28} and  \eqref{eq7.29} in Appendix~\ref{AppE.4} can be used to derive that
\begin{align}
	B_1 \lesssim (K_0K_1)^2 \big\{ \sqrt{ n\log p} +  (\log p)^3 \big\}, \nn \\
	B_2  \lesssim  (K_0 K_1)^3    \big\{ n +  (\log p)^4 \big\}, \quad B_3 \lesssim (K_0K_1)^4 \, n (\log p)^2. \nn
\end{align}
Plugging the above bounds for $B_1$--$B_3$ into \eqref{eqB.9} proves \eqref{eqB.7}.

\medskip
\noi
{\it \textbf{Step~2.}} First we prove that $\sqrt{n}Q_1$ and $\sqrt{n}Q_2$ are negligible with high probability, starting with $\sqrt{n} Q_1$. Since $\bSigma_{22.1}=\bSigma_{22}-\bSigma_{21}\bSigma_{11}^{-1}\bSigma_{12}$ is positive definite,
\be
  \max_{j\in [d]}  \big| \bSigma_{11}^{-1/2}\bSigma_{12} \, \mathbf{e}_j(d)   \big|_2     \leq  \max_{j\in [d]}  \big\{ \mathbf{e}_j(d)^{{{\rm T}}}  \bSigma_{22} \, \mathbf{e}_j(d)   \big\}^{1/2}   =1 . \label{eqB.11}
\ee
Again, from the identity $\mathbb{X}_{1}^{{{\rm T}}} \mathbb{X}_1 \bdelta_1= \mathbb{X}_{1}^{{{\rm T}}} \beps$ we find that
\begin{align}
	  &  \bdelta_1^{{{\rm T}}}  (n^{-1}\mathbb{X}_1^{{{\rm T}}} \mathbb{X}_1 ) \bdelta_1 \nn \\
	  &	 =   \bdelta_1^{{{\rm T}}} \bSigma_{11}^{1/2}   n^{-1}\bSigma_{11}^{-1/2}\mathbb{X}_1^{{{\rm T}}} \beps    \label{eqB.12} \\
	& \leq \big|  \bSigma_{11}^{1/2} \bdelta_1 \big|_2   \cdot  \big|  n^{-1}\bSigma_{11}^{-1/2} \mathbb{X}_{1}^{{{\rm T}}} \beps \big|_2   \nn \\
	&  \leq  \big|  \bSigma_{11}^{1/2} \bdelta_1 \big|_2  \cdot \sqrt{s} \,  \big|  n^{-1}\bSigma_{11}^{-1/2} \mathbb{X}_{1}^{{{\rm T}}} \beps \big|_{\infty} . \nn
\end{align}
To bound the left-hand side of \eqref{eqB.12} from below, note that
\begin{align}
	&  \bdelta_1^{{{\rm T}}}  (n^{-1}\mathbb{X}_1^{{{\rm T}}} \mathbb{X}_1 ) \bdelta_1  \nn \\
	& =  \bdelta_1^{{{\rm T}}} \bSigma_{11}^{1/2} \big( n^{-1}\bSigma_{11}^{-1/2}\mathbb{X}_1^{{{\rm T}}} \mathbb{X}_1\bSigma_{11}^{-1/2} - \bI_s \big) \bSigma_{11}^{1/2}  \bdelta_1 +  \bdelta_1^{{{\rm T}}}  \bSigma_{11} \bdelta_1 \nn \\
 & \geq \big(  1-  \big\| n^{-1} \bSigma_{1 1}^{-1/2}\mathbb{X}_{1}^{{{\rm T}}} \mathbb{X}_{1}\bSigma_{1 1}^{-1/2} -\bI_s   \big\| \big) \bdelta_1^{{{\rm T}}}  \bSigma_{11} \bdelta_1. \label{eqB.13}
\end{align}
Recall that $\bX_{i}=(\bX_{i,1}^{{{\rm T}}}, \bX_{i,2}^{{{\rm T}}})^{{{\rm T}}}$, $\mathbb{X}_{1}\bSigma_{1 1}^{-1/2}=(\bSigma_{1 1}^{-1/2} \bX_{1,1} ,\ldots , \bSigma_{1 1}^{-1/2} \bX_{n,1}  )^{{{\rm T}}} \in \bbr^{n\times s}$. Under Condition~4.1,
\begin{align}
 \sup_{\balpha\in \bS^{s-1}}\big\|  \balpha^{{{\rm T}}} \bSigma_{1 1}^{-1/2} \bX_{i,1} \big\|_{\psi_2} &  = \sup_{\balpha\in \bS^{s-1}}\big\|  \balpha^{{{\rm T}}}\bSigma_{1 1}^{-1/2} (\bI_s, \mo) \bSigma^{1/2} \bU \big\|_{\psi_2}  \nn \\
 & \leq K_1 \sup_{\balpha\in \bS^{s-1}}\big|   \bSigma^{1/2} (\bI_s, \mo)^{{{\rm T}}}  \bSigma_{1 1}^{-1/2} \balpha \big|_2 \nn \\
 & = K_1 \sup_{\balpha\in \bS^{s-1}}|\balpha |_2 = K_1, \nn
\end{align}
which, together with Theorem~5.39 in \cite{V12} yields that, for every $t\geq 0$,
\be
  \big\| n^{-1} \bSigma_{1 1}^{-1/2}\mathbb{X}_{1}^{{{\rm T}}} \mathbb{X}_{1}\bSigma_{1 1}^{-1/2} -\bI_s   \big\|  \lesssim \max(\delta, \delta^2)    \label{eqB.14}
\ee
holds with probability at least $1-2\exp(-c_{{\rm B}}t^2)$, where $\delta = K_1^2 \, n^{-1/2}(\sqrt{s} + t )$. By \eqref{eqB.12}, \eqref{eqB.13} and taking $t=c_{{\rm B}}^{-1/2} \sqrt{\log (2n)}$ in \eqref{eqB.14}, we have with probability at least $1-n^{-1}$,
\be
  \frac{1}{2} \bdelta_1^{{{\rm T}}}  \bSigma_{11} \bdelta_1 \leq   \big( n^{-1}\mathbb{X}_1^{{{\rm T}}} \mathbb{X}_1 \big) \bdelta_1  \leq  \big|  \bSigma_{11}^{1/2} \bdelta_1 \big|_2  \cdot \sqrt{s} \,  \big|  n^{-1}\bSigma_{11}^{-1/2} \mathbb{X}_{1}^{{{\rm T}}} \beps \big|_{\infty}  \label{eqB.15}
\ee
whenever the sample size $n$ satisfies $n\gtrsim K_1^4 (s+ \log n)$.

To bound the right-hand side of \eqref{eqB.15}, we define $\xi_{ij} =  \mathbf{e}_j(s)^{{{\rm T}}}  \bSigma_{11}^{-1/2}\bX_{i,1}    \varepsilon_i $ such that $
|  n^{-1}\bSigma_{11}^{-1/2} \mathbb{X}_{1}^{{{\rm T}}} \beps |_{\infty} = \max_{j\in [s]} | n^{-1} \sn \xi_{ij} |$. Under Condition~4.1, we have $\e (\xi_{ij} )=0$, $\e (\xi_{ij}^2)=1$ and
\begin{align}
 \| \xi_{ij} \|_{\psi_1} &  \leq 2\| \varepsilon \|_{\psi_2} \big\| \mathbf{e}_j(s)^{{{\rm T}}} \bSigma_{11}^{-1/2}(\bI_s, \mo ) \bSigma^{1/2} \bU \big\|_{\psi_2} \nn \\
 &  \leq 2 K_0 K_1 \big| \mathbf{e}_j(s)^{{{\rm T}}} \bSigma_{11}^{-1/2}(\bI_s, \mo ) \bSigma^{1/2} \big|_2 = 2K_0 K_1. \nn
\end{align}
Using the union bound and inequality (7.1) in the main text implies that, for every $t>0$,
\begin{align}
	\P\Big\{ \big|  \bSigma_{11}^{-1/2} \mathbb{X}_{1}^{{{\rm T}}} \beps \big|_{\infty} >2K_0K_1 \max\big( \sqrt{n t} ,   t \big)    \Big\}   \leq 2s \exp(-c_{{\rm B}} t).  \label{eqB.16}
\end{align}
Taking respectively $t=c_{{\rm B}}^{-1/2}\sqrt{\log(2n)}$ and $t=c_{{\rm B}}^{-1} \log(2sn)$ in \eqref{eqB.14} and \eqref{eqB.16} yields, with probability at least $1 - 2n^{-1}$,
\begin{align}
	\big\| n^{-1}  \bSigma_{11}^{-1/2}\mathbb{X}_1^{{{\rm T}}} \mathbb{X}_1 \bSigma_{11}^{-1/2} - \bI_s\big\|  \big|  \bSigma_{11}^{1/2}\bdelta_1 \big|_2 \lesssim   K_0K_1^3 \, n^{-1} s \log n  \nn
\end{align}
whenever $n\gtrsim K_1^4 (s+ \log n)$. Combining this with \eqref{eqB.11}, we have with the same probability,
\be
	\sqrt{n} Q_1 \lesssim K_0 K_1^3 \, n^{-1/2}  s\log n   \label{eqB.17}
\ee
whenever $n\gtrsim K_1^4 (s+ \log n)$.

Turning to $Q_2$, we define $\xi_{i,jk}=X_{ij} \,  \mathbf{e}_k(s)^{{{\rm T}}}\bSigma_{11}^{-1/2} \bX_{i,1} $ such that
\begin{align*}
 Q_{21} & := \max_{j\in  [p]\setminus [s] \atop k\in [s]}   \bigg|  n^{-1} \sn ( {\rm Id} - \e)  X_{ij}  \, \mathbf{e}_k(s)^{{{\rm T}}}\bSigma_{11}^{-1/2} \bX_{i,1}  \bigg|  \\
 &  = \max_{j\in  [p]\setminus [s] \atop k\in [s]}   \bigg| n^{-1}\sn \big( \xi_{i,jk}- \e \xi_{i,jk} \big) \bigg|
\end{align*}
and $Q_2\leq \sqrt{s} \,  | \bSigma_{11}^{1/2}\bdelta_1 |_2 \, Q_{21}$. To bound $Q_{21}$, note that $\|  \xi_{i,jk}- \e \xi_{i,jk} \|_{\psi_1}\leq 2\| \xi_{i,jk}\|_{\psi_1}$ and
\begin{align}
	  \big\| \xi_{i,jk} \big\|_{\psi_1} & \leq 2 \| X_{ij} \|_{\psi_2} \big\| \mathbf{e}_k(s)^{{{\rm T}}}\bSigma_{11}^{-1/2} \bX_{i,1}  \big\|_{\psi_2}  \nn \\
	  & =    2 \big\| \mathbf{e}_j(p)^{{{\rm T}}} \bSigma^{1/2} \bU \big\|_{\psi_2} \big\| \mathbf{e}_k(s)^{{{\rm T}}}\bSigma_{11}^{-1/2} (\bI_s, \mo)\bSigma^{1/2}\bU  \big\|_{\psi_2} \nn \\
	& \leq 2K_1^2  \big| \bSigma^{1/2} \mathbf{e}_j(p) \big|_2 \cdot  \big| \bSigma^{1/2} (\bI_s, \mo)^{{{\rm T}}} \bSigma_{11}^{-1/2} \mathbf{e}_k(s) \big|_2 = 2K_1^2. \nn
\end{align}
Then using inequality (7.1) and the union bound again, we obtain that for every $t > 0$,
\begin{align}
	\P \big\{ Q_{21} \geq   4K_1^2 \max\big( \sqrt{t/n} , t/n \big) \big\} \leq 2(p-s)\exp(-c_{{\rm B}} t).   \nn
\end{align}
Taking $t=c_{{\rm B}}^{-1}\log (2pn)$, we conclude from the bound on $|  \bSigma_{11}^{1/2}\bdelta_1 |_2$ established earlier that, with probability at least $1-3n^{-1}$,
\be
  \sqrt{n} Q_2 \lesssim K_0 K_1^3 \, n^{-1/2}  s \log p  \label{eqB.18}
\ee
whenever $n\gtrsim K_1^4 (s+ \log n)$.

Putting \eqref{eqB.6}, \eqref{eqB.17} and \eqref{eqB.18} together implies that, with probability at least $1-3n^{-1}$,
\begin{align}
	\big| \sqrt{n}R_n^{{\rm oracle}} - \sqrt{n} \widetilde R_n \big|  \lesssim K_0 K_1^3 \, n^{-1/2}  s\log p   \label{eqB.19}
\end{align}
whenever $n\gtrsim K_1^4 (s+ \log n)$.

Next, we prove that $\sqrt{n} \hat R_n^{{\rm oracle}}$ and $\sqrt{n} \widetilde R_n$ are close with high probability. To this end, set  $\hat{\beps}= (\hat{\varepsilon}_1,\ldots, \hat{\varepsilon}_n)^{{{\rm T}}}$ and define $\hat{\sigma}^2 = n^{-1}\sn (\hat{\varepsilon}_i -\bar{\varepsilon})^2$, $\hat{\sigma}_{j}^2=n^{-1}\sn (X_{ij}-\bar{X}_j)^2$, where $\bar{\varepsilon}=  \mathbf{e}_n^{{{\rm T}}} \, \hat \beps$ and $\mathbf{e}_n = (1/n, \ldots ,1/n)^{{{\rm T}}}\in \bbr^n$. In this notation, we have
$$
	\sqrt{n} \hat R_n^{{\rm oracle}} =  \hat{\sigma}^{-1} \max_{j\in [p]} \hat{\sigma}_{j}^{-1} \bigg| n^{-1/2} \sn  \hat{\varepsilon}_i X_{ij} -  \sqrt{n} \, \bar{\varepsilon} \bar{X}_j \bigg|.
$$
Combined with \eqref{eqB.1}, this implies
\begin{align}
	& \big| \sqrt{n} \hat R_n^{{\rm oracle}} - \sqrt{n} R_n^{{\rm oracle}} \big| \nn \\
	&  \leq \max_{j\in [p]} \big| (\hat{\sigma} \,\hat{\sigma}_j)^{-1} -1 \big| \cdot \sqrt{n} R^{{\rm oracle}} +  \sqrt{n}  \, \hat{\sigma}^{-1}|\bar{\varepsilon}| \max_{j\in [p]}\hat{\sigma}_j^{-1} |\bar{X}_j|. \label{eqB.20}
\end{align}
In view of \eqref{eqB.19}, it suffices to show that the right-hand side of \eqref{eqB.20} is negligible with high probability. The following lemma provides deviation inequalities for the variance estimators $\hat{\sigma}^2$ and $\hat{\sigma}_j^2$ as well as the sample means $|\bar{\varepsilon}|$ and $|\bar{X}_j|$. The proof is deferred to Section~\ref{Appendix.F}.

\begin{lemma}  \label{lemB.1}
Assume that Condition~4.1 holds. Then, with probability at least $1-C n^{-1}$,
\begin{align}
	   \max_{j\in [p]} |\bar{X}_j| \lesssim K_1 \sqrt{\frac{\log p}{n}}, \quad \max_{j\in [p]} \big| \hat{\sigma}_j^2-1 \big| \lesssim K_1^2 \bigg(  \sqrt{\frac{\log p}{n}} + \frac{\log p}{n} \bigg)  \label{eqB.21}
\end{align}
and
\begin{align}
  |\bar{\varepsilon} | \lesssim K_0 K_1^2  \frac{s\log n}{n} , \quad \big|  \hat{\sigma}^2 -1 \big| \lesssim  K_0^2 \sqrt{\frac{\log n}{n}} + (K_0K_1)^2  \frac{s\log n}{n}, \label{eqB.22}
\end{align}
provided that $n \gtrsim  K_1^4  s\log n$.
\end{lemma}

In addition, for $\sqrt{n}\widetilde{R}_n$ in \eqref{eqB.8}, it follows from the union bound, inequality (7.1) in the main text and \eqref{eqB.10} that, with probability at least $1-2d\exp(-c_{{\rm B}}t)$, $\sqrt{n}\widetilde{R}_n \lesssim K_0K_1 \max( \sqrt{t}, n^{-1/2}  t )$. This implies by taking $t=c_{{\rm B}}^{-1}\log(2pn)$ that, with probability at least $1-n^{-1}$,
\begin{align}
\sqrt{n}\widetilde{R}_n \lesssim K_0 K_1\sqrt{\log p}.  \label{eqB.23}
\end{align}

Combining \eqref{eqB.19}, \eqref{eqB.20} and \eqref{eqB.23}, we conclude from Lemma~\ref{lemB.1} that, with probability at least $1-C n^{-1}$,
\begin{align}
\big| \sqrt{n} \hat R^{{\rm oracle}} - \sqrt{n} \widetilde{R}_n \big|  \lesssim (K_0\vee K_1)^2 K_0K_1 \, n^{-1/2} s \log p ,   \label{eqB.24}
\end{align}
provided $n\gtrsim (K_0\vee K_1)^4  s \log p$.

\medskip
\noi
{\it \textbf{Step~3.}} From \eqref{eqB.7} with $\delta=(K_0K_1)^{3/4} \min \{ 1, n^{-1/8} (\log p)^{3/8} \}$ and \eqref{eqB.24}, we obtain that, with probability at least $1-C (K_0K_1)^{3/4} n^{-1/8} (\log p)^{7/8}$,
\begin{align} \label{eqB.25}
	\big| \sqrt{n} \hat{R}^{{\rm oracle}}_n - \widetilde{T}^* \big| \lesssim (K_0K_1)^{3/4} \frac{(\log p)^{3/8}}{n^{1/8}} + (K_0\vee K_1)^2 K_0K_1 \frac{s \log p}{\sqrt{n}},
\end{align}
where $\widetilde T^* \sta{d}{=} \widetilde R^*$. For $\widetilde{R}^* =| \widetilde{\bZ} |_\infty $, it follows from Theorem~3, (ii) in \cite{CCK14b} and the fact $\max_{j\in [d]}\widetilde \sigma_{jj}\leq 1$ that, for every $\epsilon >0$,
\begin{align} \label{eqB.26}
 \sup_{t\geq 0} \P\big(  \big| \widetilde{R}^* -t \big| \leq \epsilon \big) \leq \widetilde C  \epsilon \big\{ \sqrt{\log d} + \sqrt{ \log(1 / \epsilon)} \, \big\},
\end{align}
where $ \widetilde C>0$ depends only on $\widetilde \sigma_{\min} = \min_{j\in [d]}\widetilde{\sigma}_{jj}$, which under Condition~4.2, is bounded away from zero.

Finally, combining \eqref{eqB.25} and \eqref{eqB.26} leads to
\begin{align}
	&  \sup_{t\geq 0} \big|    \P\big\{ \sqrt{n} \hat{R}_n^{{\rm oracle}}   \leq t \big\}    -  \P \big( \widetilde{R}^* \leq t \big) \big| \nn \\
	&\quad \quad \quad \quad     \lesssim (K_0 K_1)^{3/4}  n^{-1/8} (\log p)^{7/8} +   (K_0 \vee K_1 )^2 K_0 K_1 \, n^{-1/2} s\log p. \nn
\end{align}
The conclusion of the theorem follows immediately. \qed

\section{Proof of Theorem~4.2}

In view of Theorem~4.1, we only need to prove the strong oracle property of $\hat{\bbeta}^{{\rm lla}}$, i.e.
\be
 \P\big( \hat{\bbeta}^{{\rm lla}} = \hat{\bbeta}^{{\rm oracle}} \big) \rightarrow 1 \quad  \mbox{ as } n\rightarrow \infty.    \label{lla.1}
\ee
Together, \eqref{lla.1} and (4.6) prove (4.9).

To prove \eqref{lla.1}, define events
\begin{align}
   \mathcal{A}_1 = \bigg\{      \max_{1\leq j\leq p}   n^{-1} \sn X_{ij}^2   \leq 2 \bigg\}  , \ \ \mathcal{A}_2 =  \Big\{ \kappa(s,3,\mS_n) \geq \tfrac{1}{2} \kappa(s, 3,\bSigma) \Big\},    \label{lla.2}
\end{align}
where $\mS_n = n^{-1}   \mathbb{X}^{{{\rm T}}}\mathbb{X}$. Given $\{\bX_i\}_{i=1}^n$ and on the event $\mathcal{A}_1 \cap \mathcal{A}_2$, applying Theorem~1 and Corollary~3 in \cite{FXZ14} gives, with conditional probability at least $1-2p\exp( -c_0 n \lambda_{{\rm lasso}}^2 / K_0^2 )- 2(p-s)\exp( -c_1 n \lambda^2/ K_0^2)$ over $\{\varepsilon_i\}_{i=1}^n$, the computed estimator $ \hat{\bbeta}^{{\rm lla}}$ equals the oracle estimator $\hat{\bbeta}^{{\rm oracle}}$, provided $\lambda\geq \frac{8\sqrt{s}\, \lambda_{{\rm lasso}}}{\kappa(s,3,\bSigma)}$, where $c_0,c_1>0$ are absolute constants. Taking into account the randomness of $\{\bX_i\}_{i=1}^n$, we obtain that
\begin{align*}
 &	\P\big( \hat{\bbeta}^{{\rm lla}} \neq  \hat{\bbeta}^{{\rm oracle}} \big) \\
 &  \leq   2p\exp( -c_0 n \lambda_{{\rm lasso}}^2 / K_0^2 ) + 2(p-s)\exp( -c_1 n \lambda^2/ K_0^2) + \P(\mathcal{A}_1^{{\rm c}}) + \P(\mathcal{A}_2^{{\rm c}}) .
\end{align*}

It remains to show that the events $\mathcal{A}_1$ and $\mathcal{A}_2$ in \eqref{lla.2} hold with overwhelming probability. Using the union bound and the one-sided version of inequality (7.1), we find that the probability of the complementary event $\mathcal{A}_1^{{\rm c}}$ satisfies $\P( \mathcal{A}_1^{{\rm c}}) \leq  p\exp(-c_2 n /K_1^4)$. Under Condition~4.1, $\bX_i = \bSigma^{1/2}\bU_i$, where $\bU_1,\ldots, \bU_n$ are i.i.d. $\bbr^p$-valued isotropic random vectors. Then it follows from Theorem~6 and Remark~15 in \cite{RZ13} by taking $\delta=1-\sqrt{2}/2$, $s_0=s$, $k_0=3$, $q=p$, $\alpha=K_1$, $A=\bSigma^{1/2}$ and $\Psi=(\bU_1, \ldots, \bU_n)^{{{\rm T}}}\in \bbr^{n\times p}$ there that, $
 \P(\mathcal{A}_2^{{\rm c}}) \leq  2\exp(-c_3n/K_1^4 ) $ whenever the sample size $n$ satisfies $n \gtrsim    \frac{K_1^4 s\log p}{\kappa(s,3+\epsilon, \bSigma)} $. Here, $c_2, c_3>0$ are absolute constants. The proof of Theorem~4.2 is then complete. \qed

\section{Proof of Proposition~3.1}

\subsection{Preliminaries}

First, we introduce basic notation and definitions that will be used to prove Proposition~3.1.

\subsubsection{$m$-generated convex set} \label{pre.sec1}

For any convex set $A \subseteq \bbr^p$, its support function is defined as $\bv \mapsto S_A(\bv) := \sup\{\bw^{{{\rm T}}} \bv: \bw \in A\}$ for $\bv \in \bS^{p-1}$, such that $A$ can be written as
$$
	A=\bigcap_{\bv \in \bS^{p-1}} \big\{	\bw \in \bbr^p: \bw^{{{\rm T}}} \bv \leq S_A(\bv) \big\}.
$$
Following \cite{CCK14}, we say that $A$ is {\it $m$-generated} if it is generated by intersections of $m$ half-spaces; that is, there exists a subset $\mathcal{V}(A) \subseteq \bS^{p-1}$ consisting of $m$ unit vectors that are outward normal to the faces of $A$ such that
\begin{align}
A = \bigcap_{\bv \in \mathcal{V}(A) } \big\{ \bw \in \bbr^p:  \bw^{{{\rm T}}} \bv \leq S_A(\bv) \big\}. \nn
\end{align}
Moreover, for $m\geq 1$ and $\epsilon>0$, we say that a convex set $A$ admits an approximation with precision $\epsilon$ by an $m$-generated convex set $A^m$ if $A^m \subseteq A \subseteq A^{m,\epsilon}$, where $A^{m,\epsilon}:= \cap_{\bv\in \mathcal{V}(A^m)} \{ \bw \in \bbr^p: \bw^{{{\rm T}}} \bv \leq S_{A^m}(\bv) + \epsilon \}$.

\subsubsection{Sparsely convex set} \label{pre.sec2}

In this section, we consider a particular class of convex sets that can be approximated by $m$-generated convex sets with a pre-specified precision for some finite $m\geq 1$.

\begin{definition}[Sparsely convex sets]  \label{defD.1}
Let $1\leq s\leq p$ and $Q \geq 1$ be two integers. We say that $A\subseteq \bbr^p$ is an $(s,Q)$-sparsely convex set if $A = \cap_{q=1}^Q A_q$, where for each $q$, $A_q$ is a convex set and is such that the map $\bw\mapsto  I (\bw \in  A_q )$ depends at most on $s$ components of $\bw=(w_1,\ldots, w_p)^{{{\rm T}}} \in \bbr^p$. We refer to $A =\cap_{q=1}^Q A_q$ as a sparse representation of $A$.
\end{definition}

The class of {\em sparsely convex sets} can be regarded as a generalization of the class of the rectangles. We refer to \cite{CCK14} for a detailed introduction and more concrete examples. In particular, the following result which is Lemma~D.1 there shows that under suitable conditions, sparsely convex sets can be approximated by $m$-generated convex sets with pre-specified precisions.

\begin{lemma} \label{lemD.1}
Assume that A is an $(s,Q)$-sparsely convex set satisfying (i). $\mo \in A$, (ii). $\sup_{\bw\in A}|\bw|_2 \leq R$ for some $R>0$ and (iii). $A=\cap_{q=1}^Q A_q$, where for each $q$, $-A_1 \subseteq \mu A_q$ for some $\mu \geq 1$. Then for every $\gamma > e/8$, there exists $\epsilon_0 = \epsilon_0 (\gamma )>0$ such that for any $0<\epsilon < \epsilon_0$,  $A$ admits an approximation with precision $R\epsilon $ by an $m$-generated convex set $A^m$ satisfying that (i). $|\bv |_0\leq s$ for all $\bv \in \mathcal{V}(A^m)$, and (ii). $ m \leq Q   \left( \gamma \sqrt{\frac{\mu+1}{\epsilon} } \log \frac{1}{\epsilon }\right)^{s^2}$.
\end{lemma}

\subsubsection{Central limit theorem for simple convex sets}

Let $\bX_1, \ldots , \bX_n$ be i.i.d. $p$-dimensional random vectors with mean zero and covariance matrix $\bSigma$, and let $\bZ$ be a $p$-dimensional centered Gaussian random vector with the same covariance matrix. Assume that diag$(\bSigma)=\bI_p$. Write $\bW = n^{-1/2}\sn \bX_i$. For a given class $\mathcal{A}$ of Borel sets in $\bbr^p$, the problem of bounding the quantity $\rho_n (\mathcal{A} ) = \sup_{A\in \mathcal{A}} | \P(\bW \in A ) - \P(\bZ \in A) |$,
which characterizes the rate of convergence to normality with respect to $\mathcal{A}$, is of long-standing interest. In this section, we focus on a particular class of convex sets for which a Berry-Esseen theorem can be established in high dimensions.

For integers $1\leq s\leq p$, $m \geq 1$ and for $\delta \geq 0$, we denote by $\mathcal{A}^{{\rm sc}}(s,m,\delta)$ the class of convex sets in $\bbr^p$ satisfying that, every $A\in \mathcal{A}^{{\rm sc}}(s,m,\delta)$ admits an approximation with precision $\delta$ by an $m$-generated convex set $A^m$ which can be chosen to satisfy $|\bv|_0 \leq s$ for all $\bv \in A^m$. We refer to $\mathcal{A}^{{\rm sc}}(s,m,\delta)$ as a class of simple convex sets. The following Berry-Esseen-type result is a modification of Proposition~3.2 in \cite{CCK14}.

\begin{lemma}
\label{lemD.2}
There exists some integer $1\leq s\leq p$ such that
$$
	K= \sup_{\bu \in \bS^{p-1}: |\bu|_0 \leq s }\| \bu^{{{\rm T}}} \bX_1 \|_{\psi_1} <\infty \,  \mbox{ and } \,  \sigma_{\min}^2 = \inf_{ \bu \in \bS^{p-1}: | \bu |_0\leq s}  \e (\bu^{{{\rm T}}} \bX_1 )^2 >0.
$$
Then, there exists an absolute constant $C>0$ such that for any $m\geq 1$ and $\delta>0$,
\begin{align}
  & \sup_{A\in \mathcal{A}^{{\rm sc}}(s,m, \delta)} |\P(\bW \in A) - \P(\bZ \in A) |   \nn \\
  & \qquad \qquad \qquad  \leq  C  \sigma_{\min}^{-1} \Big[
  K   n^{-1/6} \{\log(m n)\}^{7/6} +  \delta \sqrt{\log m}  \, \Big]. \nn
\end{align}

\end{lemma}

\subsection{Proof of the proposition}

First, we define the following standardized counterparts of $\sqrt{n}\hat{R}_n(k,p)$ for $k=1,\ldots, s$:
\be
 L_n(k,p) = \max_{\bu \in \bS^{p-1}: | \bu |_0=k  }  | \bu^{{{\rm T}}} \bW | =\max_{S\subseteq [p]: |S |=k} |\bW_S |_2 ,  \nn
\ee
where $\bW = n^{-1/2}\sn \by_i$ with $\by_i = \varepsilon_i \bX_i$ for $i=1,\ldots, n$.

The following lemma shows that, after properly normalized, the joint distribution of $\{ L_n(k,p)\}_{k=1}^s$ can be consistently estimated by that of the top $s$ order statistics of i.i.d. chi-square random variables with 1 degree of freedom. Recall that $\bZ=(Z_1, \ldots, Z_p)^{{{\rm T}}} \sta{d}{=} N(\mo, \bI_p )$, and $Z^2_{(1)} \leq Z_{(2)}^2 \leq \cdots \leq Z_{(p)}^2$ denote the order statistics of $\{Z_1^2, \ldots, Z_p^2\}$.

\begin{lemma} \label{lemD.3}
Assume that Conditions~2.1 and 2.2 in the main text hold with $\bSigma=\bI_p$. Then there is an absolute constant $C>0$ such that
\begin{align}
  \sup_{0<t_1<\cdots<t_k \leq 2}& \bigg|   \P  \bigg[  \bigcap_{k=1}^s  \big\{ n^{-1/2}L_n(k,p)   \leq t_k \big\}  \bigg]  \nn \\
  & -    \P\bigg[  \bigcap_{k=1}^s \big\{  n^{-1/2} R^{*}(k,p) \leq   t_k   \big\} \bigg] \bigg|   \leq C  (K_0K_1)^{1/3} \frac{  \{ s^2 \log(pn) \}^{7/6}  }{n^{1/6}} , \nn
\end{align}
where $R^{*}(k,p)^2 = \max_{\bu \in \bS^{p-1}: |\bu |_0=k } ( \bu^{{{\rm T}}} \bZ )^2=\sum_{\nu =1}^k Z^2_{(\nu)}$.
\end{lemma}

Further, define
$$
	\hat{\bL}=\sqrt{n}\big( \hat{R}_n(1,p), \ldots, \hat{R}_n(s,p)  \big)^{{{\rm T}}}, \ \  \bL=\big( L_n(1,p),\ldots, L_n(s,p) \big)^{{{\rm T}}} .
$$
Then it is easy to see that $ | \hat{\bL} - \bL |_\infty \leq \max_{1\leq k\leq s}   | \sqrt{n}\hat{R}_n(k,p) - L_n(k,p) |$. Taking $\mathcal{V} :=\cup_{k=1}^s \mathcal{V}(k,p)$ with $\mathcal{V}(k,p)=\{ \bx \mapsto \bu^{{{\rm T}}} \bx: \bu \in \bS^{p-1}, |\bu|_0=k \}$ as in Lemma~7.2 the same conclusions there hold by a similar argument. Consequently, it follows from a modification of Lemma~7.6 that, with probability greater than $1-C_1  n^{-1/2}\{  c_n(s,p) \}^{1/2}$,
\be
	 | \hat{\bL} - \bL |_\infty \leq C_2 (K_0\vee K_1)^2  K_0 K_1 \, n^{-1/2} c_n(s,p)   \label{eqD.1}
\ee
whenever $n\geq C_3 (K_0\vee K_1)^4 c_n(s,p)$, where $c_n(s,p):= s\log( ep/s)  \vee \log n$.

Together, \eqref{eqD.1} and Lemma~\ref{lemD.3} imply that for any $0<t_1<\cdots < t_s <1$ and all sufficiently large $n$,
\begin{align}
  & \P\bigg[ \bigcap_{k=1}^s  \big\{ \hat{R}_n(k,p) \leq t_k \big\} \bigg]  \nn \\
  & \leq \P\bigg[ \bigcap_{k=1}^s  \big\{  L_n(k,p) \leq \sqrt{n}(  t_k + \epsilon_{n} ) \big\} \bigg] +  C_1 n^{-1/2} \{ c_n(s,p) \}^{1/2} \nn \\
  & \leq  \P\bigg[  \bigcap_{k=1}^s \big\{  n^{-1/2} R^{*}(k,p) \leq   t_k +\epsilon_n  \big\} \bigg]  + C  (K_0K_1)^{1/3} n^{-1/6} \{s^2 \log(pn)\}^{7/6}  \nn \\
 & \leq    \P\bigg[  \bigcap_{k=1}^s \big\{  n^{-1/2} R^{*}(k,p) \leq   t_k   \big\} \bigg]  \nn \\
 & \quad + \P  \big\{ \sqrt{n}  t_s < R^{*}(s,p) \leq \sqrt{n} ( t_s + \epsilon_n ) \big\}    + C   (K_0K_1)^{1/3} n^{-1/6} \{s^2 \log(pn)\}^{7/6}   \nn \\
 & \leq  \P\bigg[  \bigcap_{k=1}^s \big\{  n^{-1/2} R_k \leq   t_k   \big\} \bigg]   \nn \\
 & \quad  + C \bigg[   \epsilon_n \sqrt{ns\log (ep/s ) }   +  (K_0K_1)^{1/3}  n^{-1/6} \{s^2 \log(pn)\}^{7/6} \bigg]  , \nn
\end{align}
where $\epsilon_n = \epsilon_n(s,p)= C_2(K_0\vee K_1)^2 K_0 K_1 \, n^{-1} c_n(s,p) \leq 1 $ for all sufficiently large $n$. A similar argument leads to the reverse inequality, and hence completes the proof. \qed

\subsection{Proof of Lemma~\ref{lemD.2}}

This proof is similar to that of Proposition~3.2 in \cite{CCK14} with slight modification. We reproduce them here for the sake of readability.

For every $A\in \mathcal{A}^{{\rm sc}}(s, m , \delta)$, let $A^m$ be the approximating $m$-generated convex set of $A$ such that $A^m \subseteq A \subseteq A^{m,\delta}$. Put
$$
	\rho = \max\big\{   | \P (\bW \in A^m ) - \P(\bZ \in A^m) | ,    | \P (\bW \in A^{m,\delta} ) - \P(\bZ \in A^{m,\delta}) | \big\}.
$$
Applying Lemma~A.1 in \cite{CCK14} and Theorem~20 in \cite{KOS08} to the $m$-dimensional Gaussian random vector $(\bv^{{{\rm T}}} \bZ)_{\bv\in \mathcal{V}(A^m)}$ implies that
\begin{align}
| \P(\bW \in A ) - \P(\bZ \in A ) | \leq  \sigma_{\min}^{-1} \big( \sqrt{2\log m} + 2 \big) \delta + \rho. \label{eqD.2}
\end{align}

Recall that $K=\sup_{ \bu \in \bS^{p-1}: |\bu|_0\leq s} \| \bu^{{{\rm T}}} \bX_i  \|_{\psi_1} <\infty $. Then, for every $q \geq 3$ and $\bv \in \mathcal{V}(A^m)$ with $| \bv |_0\leq s$,
$$
	n^{-1}\sn \e |\bv^{{{\rm T}}} \bX_i |^{q} \leq n^{-1} \sn (  q \|  \bv^{{{\rm T}}} \bX_i \|_{\psi_1}  )^q \leq (qK)^q.
$$
Consequently, it follows from Proposition~2.1 in \cite{CCK14} that
\begin{align}
 \rho \lesssim \sigma_{\min}^{-1} K \,  n^{-1/6}  \{\log (mn)\}^{7/6}.  \label{eqD.3}
\end{align}

Together, \eqref{eqD.2} and \eqref{eqD.3} complete the proof of the lemma. \qed

\subsection{Proof of Lemma~\ref{lemD.3}}

For any $0<t_1 <\cdots < t_s\leq 2$, we have
\begin{align*}
 \P \big\{  L_n(1,p) \leq \sqrt{n} t_1 , L_n(2,p)\leq \sqrt{n}   t_2, \cdots & ,  L_n(s,p) \leq \sqrt{n}  t_s \big\} \\
 &  = \P \bigg\{    \bW \in \bigcap_{k=1}^s A_k(w_k) \bigg\} ,
\end{align*}
where $w_k= n t_k^2$ for $k=1,\ldots, s$ and for $ t \geq 0$,
\beq
	A_k( t) := \big\{ \bw \in \bbr^p : | \bw_S |_2^2 \leq t  \, \mbox{ for all $S\subseteq [p]$ with $|S|=k$} \big\}.
\eeq
Put $A(\bt)=\cap_{k=1}^s A_k( w_k)$, where $\bt=(t_1,\ldots, t_s)$. For every $1\leq k\leq s$, let $\{S_{k \ell}\}_{\ell=1}^{  {p\choose k}}$ be all the subsets of $[p]$ with cardinality $k$. In this notation, we can further write the set $A(\bt)$ as
\begin{align*}
A(\bt) =  \bigcap_{k=1}^s \bigcap_{\ell=1}^{{p \choose k}}  A_{k\ell}  =  \bigcap_{k=1}^s \bigcap_{\ell=1}^{{p \choose k}}  \big\{ \bw \in \bbr^p :  | \bw_{S_{k\ell}} |_2^2 \leq w_k  \big\} .
\end{align*}
It is easy to see that the indicator function $\bw =(w_1,\ldots, w_p)^{{{\rm T}}} \in \bbr^p \mapsto  I( \bw\in A_{k\ell} )$ depends only on $k\,(\leq s)$ components of $\bw$. By Definition~\ref{defD.1}, $A(\bt)$ is an $ ( s, (ep/s)^s  )$-sparsely convex set. Then it follows from Lemma~\ref{lemD.1} with $R=2 (pn)^{1/2}$ and $\gamma = \mu = 1$ that there exists some constant $\epsilon_0>0$ such that for every $\epsilon \in (0, \epsilon_0)$ and $0<t_1<\cdots<t_s <1$, the set $A(\bt)$ admits an approximation with precision $2\epsilon (pn)^{1/2}$ by an $m$-generated convex set $A^m$, where
\begin{align}
	m \leq \bigg( \frac{ep}{s} \bigg)^s  \bigg( \sqrt{\frac{2}{\epsilon}} \log \frac{1}{\epsilon} \bigg)^{s^2}. \nn
\end{align}
In particular, taking $\epsilon=(pn)^{-1}$ yields, for any $\bt=(t_1,\ldots, t_s)$ with $0<t_1<t_2<\cdots < t_s \leq 2$, $A(\bt) \in \mathcal{A}^{{\rm sc}}\big( s, (ep/s)^s \{ \sqrt{2pn}\log(pn) \}^{s^2} ,2(pn)^{-1/2} \big)$. This, together with Lemma~\ref{lemD.2} and the inequality $\| \bu^{{{\rm T}}}\by_i \|_{\psi_1} \leq 2\| \varepsilon_i \|_{\psi_2} \| \bu^{{{\rm T}}}\bX_i \|_{\psi_2}  \leq 2 K_0 K_1$ that holds for all $\bu \in \bS^{p-1}$ completes the proof of the lemma. \qed

\section{Proof of Proposition~3.2}

Observe that $Z_1^2, \ldots, Z_p^2$ are i.i.d. chi-square random variables with 1 degree of freedom. For $t\in \bbr$ fixed, let $t_p=a_p + t $ with $a_p= 2\log p-\log(\log p)$ such that as $p\rightarrow \infty$,
\begin{align*}
	\P \big\{  Z_{(p)}^2 \leq t_p  \big\}  &= \big\{ 1- \P(Z_1^2 > t_p) \big\}^p   \rightarrow  \exp\bigg(  -\f{1}{\s{\pi}}\, e^{- t/2} \bigg).
\end{align*}
In other words, $Z_{(p)}^2 - a_p$ converges weakly to a Gumbel distribution with a cumulative distribution function given by $G(t)=\exp(  - \pi^{-1/2} e^{-t/2} ) $ for $t\in \bbr$. Consequently, for every $s\geq 2$ fixed, the $s$-dimensional vector
$$
	\big( Z_{(p)}^2-a_p, \ldots, Z_{(p-s+1)}^2-a_p \big)
$$
has a limiting distribution with joint density function given by [\cite{DN03}]
\begin{align*}
	g_s(t_1, \ldots, t_s ) = G(t_s) \prod_{j=1}^s \f{g(t_i)}{G(t_i)}, \ \  t_1 > t_2 > \cdots > t_s,
\end{align*}
where $g(t) = G'(t) =  \f{e^{-t/2}}{2\s{\pi}} \, G(t)$ for $t\in \bbr$, and it is easy to verify that
\beq
	g_s(t_1, \ldots, t_s) =  \bigg( \f{1}{2\s\pi}\bigg)^{s-1} \exp\bigg( -\f{1}{2}\sum_{j=1}^{s-1} t_j  \bigg) g(t_s), \quad t_1 > t_2 > \cdots > t_s.
\eeq
Consequently, as $p\to \infty$,
\begin{align}
	& \P\big\{  Z_{(p)}^2 + \cdots + Z^2_{(p-s+1)} - s  a_p \leq t  \big\}  \nn \\
	& \rightarrow \idotsint_{t_1+\cdots t_s \leq t,  t_1>\cdots > t_s} g_s(t_1, \ldots, t_s)\, dt_1 \cdots \, d t_s  \nn  \\
	& = \bigg( \f{1}{ 2\s \pi}\bigg)^{s-1}\int_{-\infty}^{t/s} \bigg( \idotsint_{t_1+\cdots + t_{s-1} \leq t- t_s, t_1> \cdots > t_{s-1} > t_s} \prod_{j=1}^{s-1} e^{-t_j/2}  \, d  t_j  \bigg) g(t_s) \, dt_s  \nn  \\
	& = \bigg( \f{1}{ \s \pi}\bigg)^{s-1}\int_{-\infty}^{t/s} \bigg( \idotsint_{t_1+\cdots + t_{s-1} \leq  (t- t_s)/2,  t_1> \cdots > t_{s-1} > t_s/2} \prod_{j=1}^{s-1} e^{-t_j }  \, d  t_j  \bigg) g(t_s) \, dt_s \nn  \\
		& = \bigg( \f{1}{  \s \pi}\bigg)^{s-1}\int_{-\infty}^{t/s} \, dt_s \, e^{-(s-1)t_s/2} g(t_s)  \nn \\
	& \quad  \times  \bigg( \idotsint_{u_1 + \cdots + u_{s-1} \leq ( t-s   t_s)/2 , \, u_1 > \cdots > u_{s-1} >0}  e^{ - u_1 - \cdots -  u_{s-1} } \, d u_1 \, \cdots \, du_{s-1}   \bigg). \label{eqE.1}
\end{align}

Now, let $E_1, \ldots, E_{s-1}$ be i.i.d. standard exponential distributed random variables and let $E_{(1)} \geq E_{(2)} \geq \cdots \geq E_{(s-1)}$ be the corresponding order statistics. It is known that the joint density function of $(E_{(1)}, \ldots, E_{(s-1)})$ is $	(s-1)!  e^{-t_1-\cdots - t_{s-1}}$, $t_1 > t_2 > \cdots > t_{s-1} \geq 0$. Therefore, the last multiple integral on the right side of \eqref{eqE.1} is equal to
\begin{align*}
	& \f{1}{(s-1)!}\P\big\{  E_{(1)} + \cdots + E_{(s-1)} \leq (t- s   t_s)/2 \big\} \\
	& = \f{1}{(s-1)!}\P\big\{ E_1 + \cdots + E_{s-1} \leq (t- s  t_s)/2 \big\} \\
	&  = \f{1}{(s-1)! \Gamma(s-1)} \int_0^{(t-s t_s)/2}  u^{s-2}  e^{-u} \, du ,
\end{align*}
where we used the fact that $E_1+ \cdots + E_{s-1} \sta{d}{=} \mbox{Gamma}(s-1, 1)$. Putting the above calculations together yields (3.7).

To prove (3.8), observe that for any $a>0$ and positive integer $\ell$,
$$
\f{1}{\Gamma(\ell)}\int_{0}^a u^{\ell-1}  e^{-u} \, du =   1- \sum_{j=0}^{\ell-1} { a^j \over j!} e^{-a}.
$$
Hence, for $s\geq 2$,
\begin{align}
&
\f{1}{\Gamma(s-1)}\int_{-\infty}^{t/s} \bigg\{ \int_0^{(t-s v)/2}    u^{s-2} e^{-u} \, d u \bigg\}  e^{-(s-1)v/2}  g(v)  \, dv  \nn  \\
& =   \int_{-\infty}^{t/s} e^{-(s-1)v/2} g(v) \, dv   \nn \\
& \quad  - \sum_{j=0}^{s-2} \f{1}{j!  2^j}\int_{-\infty}^{t/s}(t-sv)^j  e^{-(t-sv)/2  - (s-1)v/2}  g(v) \, dv \nn \\
& =     \int_{-\infty}^{t/s} e^{-(s-1)v/2}  g(v) \, dv - e^{-t/2}  \sum_{j=0}^{s-2} \f{1}{j! 2^j} \int_{-\infty}^{t/s}  (t-sv)^j  e^{ v /2 } g(v) \, dv .  \label{eqE.2}
\end{align}
Further, using integration by parts repeatedly gives
\begin{align}
	  & \int_{-\infty}^{t/s} e^{-(s-1)v/2} g(v) \, dv   =   \int_{-\infty}^{t/s} e^{-(s-1)v/2}   \, d G(v) \nn \\
	  & = G( t/s )  e^{-(s-1)t/(2s)} +(s-1) \s{\pi} \int_{-\infty}^{t/s}   e^{-(s-2)v/2}  \, dG(v)  \nn \\
	  & = G(t/s) \Big\{ e^{-(s-1)t/(2s)} +(s-1) \s{\pi} e^{-(s-2)t/(2s)} \Big\} \nn \\
	  & \quad   + (s-1)(s-2)\pi \int_{-\infty}^{t/s}   e^{-(s-3)v/2} \, d G(v)  \nn  \\
	  & = \cdots  \nn \\
	  &  =  G(t/s) \bigg\{ \pi^{(s-1)/2}(s-1)!  + e^{-t/2+  t/(2s)} \label{eqE.3} \\
	  &  \qquad \qquad \qquad  +    e^{-t/2} \sum_{j=1}^{s-2} \pi^{j/2} e^{(j+1)t/(2s)}  \prod_{\ell=1}^j(s-\ell) \bigg\}.  \nn
\end{align}
The first summand of the last term on the right-hand side of \eqref{eqE.2} reads to
\be
	\int_{-\infty}^{t/s}   e^{v/2 }  g(v) \, dv  = e^{ t/(2s)}G( t/s ) -  \s{\pi}  \int_{-\infty}^{t/s} e^v g(v) \, dv. \label{eqE.4}
\ee

Assembling \eqref{eqE.2}--\eqref{eqE.4} completes the proof of Proposition~3.2. \qed

\section{Proof of Lemma~C.1}
\label{Appendix.F}

We continue to adopt the notation in the proof of Theorem~4.1. To prove \eqref{eqB.21}, consider the inequality
$ |\hat \sigma_j^2 - 1  | \leq | n^{-1}\sn X_{ij}^2 -1   | + \bar{X}_j^2$. Analogously to (7.4), for every $t_1, t_2>0$ we have, with probability at least $1-2p\exp(-c_{{\rm H}} t_1) - 2p\exp(-c_{{\rm B}} t_2)$, $\max_{j\in [p]}|\bar{X}_j|\leq K_1 \, n^{-1/2}\sqrt{t_1}$ and
$$
	\max_{j\in [p]}\big|\hat \sigma_j^2 - 1 \big| \leq  K_1^2 \,  n^{-1} t_1  + 4 K_1^2 \max\big( n^{-1/2}\sqrt{t_2} , n^{-1} t_2 \big).
$$
In particular, taking $t_1=c_{{\rm H}}^{-1}\log(2pn)$ and $t_2=c_{{\rm B}}^{-1}\log(2pn)$ proves \eqref{eqB.21}.

Next we prove \eqref{eqB.22}. Recall that $\mathbb{Y}=\mathbb{X}_1 \bbeta_1 + \beps$ and $\hat{\beps}= (\hat{\varepsilon}_1,\ldots, \hat{\varepsilon}_n)^{{{\rm T}}} = \mathbb{Y}-\mathbb{X}_1 \hat{\bbeta}_1$. Therefore, we have
\begin{align}
	  \sn \hat{\varepsilon}_i^{\,2} = \hat{\beps}^{{{\rm T}}} \hat{\beps}  =  \beps^{{{\rm T}}} \beps - 2 \beps^{{{\rm T}}} \mathbb{X}_1 \bdelta_1 +   | \mathbb{X}_1\bdelta_1|_2^2,  \nn \\
	     \bar{\varepsilon} = \mathbf{e}_n^{{{\rm T}}} \, \hat{\beps} =   \mathbf{e}_n^{{{\rm T}}} \beps  - n^{-1} \sn \bX_{i,1}^{{{\rm T}}}  \bdelta_1. \nn
\end{align}
This and \eqref{eqB.16} yield $ | \hat{\beps}^{{{\rm T}}} \hat{\beps} -\beps^{{{\rm T}}} \beps | \leq  3\sqrt{s} \,  | \bSigma_{11}^{-1/2} \mathbb{X}_1^{{{\rm T}}} \beps |_{\infty}  \cdot | \bSigma_{11}^{1/2} \bdelta_1 |_2$
and $|  \mathbf{e}_n^{{{\rm T}}}( \hat{\beps} - \beps ) | \leq  | \bSigma_{11}^{1/2} \bdelta_1 |_2 \cdot \sqrt{s} \, |   n^{-1} \sn \bSigma_{11}^{-1/2} \bX_{i,1} |_{\infty}$. Applying the union bound and inequality (7.2) we obtain that, for every $t>0$, $|  n^{-1/2} \sn \bSigma_{11}^{-1/2} \bX_{i,1} |_{\infty}  \leq   K_1  t$ holds with probability at least $1-2s\exp(-c_{{\rm H}} t^2)$. Hence, taking $t=c_{{\rm H}}^{-1/2}\sqrt{\log(2sn)}$, we conclude from \eqref{eqB.12}, \eqref{eqB.15} and \eqref{eqB.16} that, with probability at least $1-3 n^{-1}$,
\begin{align}
\big|   \hat{\beps}^{{{\rm T}}} \hat{\beps} - \beps^{{{\rm T}}} \beps \big|  \lesssim  (K_0K_1)^2  s\log n, \quad \big|   \mathbf{e}_n^{{{\rm T}}}( \hat{\beps} - \beps )\big|   \lesssim K_0 K_1^2 \, n^{-1}  s\log n ,  \label{eqF.1}
\end{align}
provided that $n \gtrsim K_1^4  (s+\log n)$.

Finally, from \eqref{eqF.1} and (7.4) we obtain, with probability at least $1-5n^{-1}$,
$$
	\big| \hat{\sigma}^2 -1 \big| \lesssim K_0^2  \sqrt{\frac{\log n}{n}} +  (K_0K_1)^2 \frac{s\log n}{n}
$$
whenever $n\gtrsim K_1^4 \, s\log n$. This, together with \eqref{eqF.1} proves \eqref{eqB.22}. \qed

\section{Additional simulation results}
\label{Appendix.H}

In this section, we present additional numerical results for detecting spurious discoveries in the case of non-Gaussian design and noise. We continue with the setup in Section~6.3 by taking $n=120, 160$, $p=400$ and $\bbeta^* = (1, 0, -0.8, 0, 0.6, 0, -0.4, 0, \ldots, 0)^{{{\rm T}}} \in \bbr^{p}$. For $r\in \{ 120, 200, 280, 360\}$, we let $\bx = (x_1, \ldots, x_r)^{{\rm T}}$, where $x_1,\ldots, x_r$ are i.i.d. random variables following the continuous uniform distribution on $[-1,1]$. The rows of the design matrix $\mathbb{X}$ are sampled as i.i.d. copies from $\bGamma_r \bx \in \bbr^p$, where $\bGamma_r$ is a $p\times r$ matrix satisfying $\bGamma_r^{{{\rm T}}} \bGamma_r=\bI_r$. Moreover, the noise variable $\varepsilon$ follows a standardized $t$-distribution with 4 degrees of freedom. We compute the empirical SDP based on 200 simulations. The results are provided in Table~\ref{tab5}.

\begin{table}[h]
\centering
\caption{\label{size} The empirical $\alpha$-level Spurious Discovery Probability (ESDP) based on 200 simulations when $p=400$, $n=120, 160$ and $\alpha=5\%$.}
\label{tab5}
\scriptsize{\begin{tabular}{ccccc}
   &     $r=120$ &    $r=200$ & $r=280$ & $r=360$   \\ \midrule
$ n=120 $ &   $ 0.9100 $ &$ 0.8000 $&$ 0.7200 $ &$ 0.5900 $   \\
\midrule%\midrule
$ n=160 $& 0.7650    & 0.6100  &  0.3600 &  0.2750  \\
\bottomrule
\end{tabular}}
\end{table}\par

\end{document}